\documentclass[english,reqno,12pt]{amsart}
\usepackage{tgtermes}
\usepackage[T1]{fontenc}
\usepackage{textcomp}
\usepackage[utf8]{inputenc}
\usepackage{color}
\usepackage{babel}
\usepackage{verbatim}
\usepackage{mathrsfs}
\usepackage{mathtools}
\usepackage{amstext}
\usepackage{amsthm}
\usepackage{amssymb}
\usepackage{esint}
\usepackage[bookmarks=true,bookmarksnumbered=true,bookmarksopen=true,bookmarksopenlevel=2,
 breaklinks=false,pdfborder={0 0 1},backref=false,colorlinks=true]
 {hyperref}
\hypersetup{pdftitle={The maximal destabilizers for Chow and K-stability},
 pdfauthor={Yi Yao},
 bookmarksdepth=3,linkcolor=red,citecolor=blue}

\makeatletter

\providecommand{\tabularnewline}{\\}

\numberwithin{equation}{section}
\numberwithin{figure}{section}
\theoremstyle{plain}
\newtheorem{thm}{\protect\theoremname}
\theoremstyle{plain}
\newtheorem{cor}[thm]{\protect\corollaryname}
\theoremstyle{plain}
\newtheorem{prop}[thm]{\protect\propositionname}
\theoremstyle{remark}
\newtheorem{rem}[thm]{\protect\remarkname}
\theoremstyle{definition}
\newtheorem{defn}[thm]{\protect\definitionname}
\theoremstyle{definition}
\newtheorem{example}[thm]{\protect\examplename}

\usepackage[twoside,inner=1.3in,outer=1in,top=1in,bottom=1in]{geometry}

\makeatother

\usepackage[style=alphabetic,backref=true, isbn=false, url=false, doi=false, backend=biber]{biblatex}
\providecommand{\corollaryname}{Corollary}
\providecommand{\definitionname}{Definition}
\providecommand{\examplename}{Example}
\providecommand{\propositionname}{Proposition}
\providecommand{\remarkname}{Remark}
\providecommand{\theoremname}{Theorem}

\begin{document}
\title{The maximal destabilizers for Chow and K-stability}
\author{YI YAO}
\address{School of Mathematics, Hunan University, Changsha, China. }
\email{yeeyoe@163.com}
\begin{abstract}
Donaldson showed that the constant scalar curvature Kähler metrics
can be quantized by the balanced Hermitian norms on the spaces of
global sections. We explore an analogous problem in the unstable situation.
For a K-unstable manifold $(X,L)$, its projective embedding via $\left|kL\right|$
will be Chow-unstable when $k$ is sufficiently large and divisible.
There is a unique filtration on $\mathrm{H}^{0}(X,kL)$, that corresponds
to the maximal destabilizer for Chow-stability of the embedded variety.
On the other hand, there is a maximal destabilizer for K-stability
after the work of Xia and Li, which corresponds to the steepest descent
direction of K-energy. Based on Boucksom--Jonsson's non-Archimedean
pluripotential theory and some idealistic assumptions, we provide
a route to show that maximal K-destabilizers are quantized by the
maximal Chow-destabilizers. 
\end{abstract}

\maketitle
\global\long\def\lam{\lambda}%
 
\global\long\def\vphi{\varphi}%
 
\global\long\def\cch{\mathcal{H}}%
 
\global\long\def\bbc{\mathbb{C}}%
 
\global\long\def\bbcs{\mathbb{C}^{*}}%
 
\global\long\def\bbr{\mathbb{R}}%
 
\global\long\def\bbz{\mathbb{Z}}%
 
\global\long\def\bbn{\mathbb{N}}%

\global\long\def\ra{\rightarrow}%
 
\global\long\def\lead{\leadsto}%
 
\global\long\def\bdot{\bm{\cdot}}%
 
\global\long\def\'{\prime}%
 
\global\long\def\cleq{\preceq}%
\global\long\def\cgeq{\succeq}%
 
\global\long\def\na{{\scriptscriptstyle \mathsf{NA}}}%
 
\global\long\def\an{\mathrm{an}}%
 
\global\long\def\ma{\mathrm{MA}}%

\tableofcontents{}

\section{Introduction}

A central problem in Kähler geometry is the existence of canonical
metrics in a given Kähler class. Let $(X,L)$ be a polarized Kähler
manifold, the Yau--Tian--Donaldson conjecture predicts that the
existence of constant scalar curvature Kähler (cscK) metrics in $c_{1}(L)$
is equivalent to the K-polystability (or some enhanced version) of
$(X,L)$. During the last decade, this conjecture has been verified
for the Fano case $(X,-K_{X})$ by various approaches, such as \cite{chen_kahlereinstein_2015,tian_k-stability_2015,berman_variational_2021}.
Following the variational approach of \cite{berman_variational_2021},
there has been big progress for the general polarization, see \cites{li_geodesic_2022}{boucksom_non-archimedean_2023}{li_canonical_2022}. 

The quantization method plays a key role in the study of canonical
metrics problems. The main idea is that the infinite-dimensional space
$\mathcal{H}(L)$ of Kähler potentials can be approximated (quantized)
by the Riemannian symmetric space $\mathcal{H}(R_{k})$ of Hermitian
norms on $R_{k}\coloneqq\mathrm{H}^{0}(X,kL)$, as the Planck's constant
$\hbar\sim k^{-1}$ goes to zero. A cornerstone is the Tian--Yau--Zelditch's
asymptotic expansion of the Bergman kernel. There are many studies
involving quantization in this area. We mention one work that motivates
our study. 

The seminar work of Donaldson \cite{donaldson_scalar_2001,donaldson_scalar_2005}
shows that if $(X,L)$ admits a cscK metric $\omega$ and the automorphism
group is discrete, then there exists balanced norms on $R_{k}$ for
$k\gg1$. Moreover, the sequence of Fubini--Study metrics induced
by balanced norms will converge to $\omega$, see Theorem \ref{thm: Donaldson}
for a review. In a nutshell, balanced norms quantize cscK metrics.
By \cites{zhang_heights_1996}{luo_geometric_1998}, the existence
of balanced norms on $R_{k}$ is equivalent to the Chow-polystability
of the embedding image $X\hookrightarrow\mathbb{P}R_{k}^{*}$. Hence
in this case, the existence of cscK metrics can imply the asymptotic
Chow-stability. Conversely, Ross--Thomas \cite{ross_study_2006}
showed that the asymptotic Chow-semistability can imply the K-semistability
of $(X,L)$. In a rough sense, one can say that Chow-stability quantizes
K-stability. 

In this paper, we explore the parallel story on the unstable side.
If $(X,L)$ is K-unstable, by Theorem \ref{thm: K-uns imply Chow-uns}
there exists $r,N\in\bbn_{+}$ such that the embedding image $X\hookrightarrow\mathbb{P}R_{kr}^{*}$
is Chow-unstable for $k>N$. Replacing $L$ by $rL$, we can assume
$r=1$. In GIT, after the work of Kempf \cite{kempf_instability_1978}
and Rousseau \cite{rousseau_immeubles_1978}, one can consider the
maximal destabilizing 1-parameter subgroups to an unstable point,
which minimizes the normalized Hilbert--Mumford weight and is unique
up to conjugations. Hence for Chow-unstable subvariety $X\hookrightarrow\mathbb{P}R_{k}^{*}$,
we obtain a set of 1-PS in $\mathrm{SL}(R_{k})$ which are conjugate
to each other. To eliminate the conjugacy ambiguity, we consider the
filtration (i.e. NA norm) on $R_{k}$ induced by these 1-PS, that
is unique up to scaling. We call this NA norm on $R_{k}$ the \emph{maximal
Chow-destabilizer} for $X\hookrightarrow\mathbb{P}R_{k}^{*}$. Comparing
with the picture \cite{donaldson_scalar_2001} in the stable case,
we ask: 
\begin{quote}
How do this sequence of maximal Chow-destabilizers relate to the maximal
K-destabilizer?
\end{quote}
The maximal K-destabilizer was first studied by Székelyhidi \cite{szekelyhidi_optimal_2008}
in the toric setting, where it is given by a concave function $\Theta$
(optimal test-configuration) on the moment polytope $P$ that minimizes
the Donaldson--Futaki invariant. Assuming that the Calabi flow exists
for all time, Székelyhidi showed that the scalar curvature under the
flow converges to $\Theta$. Moreover, if $\Theta$ is piecewise linear,
it induces a polyhedral decomposition of $P$, such that each piece
is relatively K-semistable in a certain sense. These support a conjecture
of Donaldson \cite{donaldson_scalar_2002,donaldson_conjectures_2004},
which predicts that when the extremal metrics are absent, Calabi flow
will break up the manifold into pieces, and each piece admits a complete
extremal metric or collapses. Also see \cite{szekelyhidi_calabi_2009}
for the case of ruled surfaces. 

In the general case, maximal K-destabilizer is defined in the work
of Xia \cite{xia_sharp_2021}. The K-energy $\mathcal{M}$ can be
extended to a convex functional over the space $\mathcal{E}^{2}(L)$
of $L^{2}$-finite energy metrics on $L$, which is a Hadamard space.
By considering the associated gradient flow, Xia showed that in the
geodesically unstable case, modulo the asymptotic relations, there
exists a unique geodesic ray $\ell^{\mathrm{K}}$ in $\mathcal{E}^{2}(L)$
along which $\mathcal{M}$ achieves the minimal limit slope, namely
it minimizes the radial K-energy $\mathcal{M}^{\mathrm{rad}}$ (\ref{eq:radial K-energy}).
We call $\ell^{\mathrm{K}}$ the \emph{maximal K-destabilizing ray}.
By the construction in \cite{berman_variational_2021}, $\ell^{\mathrm{K}}$
induces a non-Archimedean metric $\phi^{\mathrm{K}}$ on the Berkovich
analytification $L^{\an}$ of $L$ (see §\ref{subsec: limit NA metric}).
By \cite{li_geodesic_2022}, $\ell^{\mathrm{K}}$ is maximal in the
sense of \cite{berman_variational_2021}, so it can be recovered from
the NA data $\phi^{\mathrm{K}}$. We call $\phi^{\mathrm{K}}$ the
\emph{maximal K-destabilizer} for $(X,L)$, which is unique up to
scaling, see Definition \ref{def: max K-desta} for details. In the
toric setting, it can be reduced into Székelyhidi's optimal test-configuration,
see Example \ref{exa:toric optimal}. 

In order to connect the maximal Chow-destabilizers to the maximal
K-destabilizer, first we need to characterize them in a unified way,
namely as the ``steepest descent direction'' for some energy functionals.
Instead of being a 1-PS that minimizes the Chow-weight, the maximal
Chow-destabilizer can also be represented by the norm geodesic ray
along which the Kempf--Ness functional achieves the minimal limit
slope. 

\subsection*{Maximal Chow-destabilizers}

Let $X\subset\mathbb{P}V^{*}$ be a projective variety and $\mathcal{H}(V)$
be the space of Hermitian norms on $V$. By works \cite{zhang_heights_1996,luo_geometric_1998,phong_stability_2003,paul_geometric_2004},
the Kempf--Ness functional for the Chow-stability of $X$ can be
identified with the following functional on $\mathcal{H}(V)$, 
\[
\mathbf{M}_{X}(H)\coloneqq\mathcal{E}_{X}\left(\mathbf{FS}(H),\mathbf{FS}(H_{r})\right)+\frac{1}{\dim V}\log\left(\frac{\det H}{\det H_{r}}\right),
\]
where $\mathbf{FS}(H)$ is the Fubini--Study metric on $\mathcal{O}(1)|_{X}$
induced by $H$, $\mathcal{E}_{X}$ is the Monge--Ampère energy,
and $H_{r}$ is a reference norm. Chow-weights are obtained by taking
the limit slope of $\mathbf{M}_{X}$ along the norm geodesic rays.
Consider the asymptotic cone $\hat{\mathcal{H}}(V)$ of $\mathcal{H}(V)$,
which is a Hadamard space consisting of geodesic rays in $\mathcal{H}(V)$
modulo the asymptotic relations. By \cite{boucksom_variational_2019},
a norm geodesic ray $H=(H_{t})_{t\geq0}$ induces a NA norm $\chi_{H}$
on $V$. This gives an isometry $\hat{\mathcal{H}}(V)\xrightarrow{\sim}\mathcal{N}(V)$,
the target is the space of NA norms on $V$. We define the Chow-weight
$M_{X}:\mathcal{N}(V)\rightarrow\bbr$ as follows 
\begin{equation}
M_{X}(\chi)\coloneqq\lim_{t\rightarrow\infty}\frac{1}{t}\mathbf{M}_{X}(H_{t}),\label{eq:Chow w INTRO}
\end{equation}
where $(H_{t})$ is any norm geodesic ray such that $\chi_{H}=\chi$.
It is well-defined since $\mathbf{M}_{X}$ is convex and Lipschitz
(Proposition \ref{prop: conv of Mk}). Note that $\mathbf{M}_{X}$
depends on the choice of norm (\ref{eq:norm chow}) on Chow-forms,
but Chow-weight does not. Since $M_{X}$ is homogeneous under scaling,
we normalize it and then consider the minimization problem. 
\begin{thm}[Theorem \ref{thm:existence-uniqueness of Chow}]
\label{thm:Chow exist INTRO}For any projective manifold $X\subset\mathbb{P}V^{*}$,
consider the minimization problem for $L^{p}$-normalized Chow-weight
\begin{equation}
\inf\left\{ \bar{M}_{X,p}(\chi)\coloneqq\frac{M_{X}(\chi)}{\left\Vert \chi\right\Vert _{p}}\mid\chi_{tr}\neq\chi\in\mathcal{N}(V)\right\} ,\label{eq:min Chow INTRO}
\end{equation}
where $\left\Vert \chi\right\Vert _{p\in[1,\infty]}$ is the $p$-norm
(\ref{eq:p norm}). (1) there exists a minimizer for any $p\in[1,\infty]$.
(2) when $X$ is Chow-unstable (i.e. $\inf M_{X}<0$) and $p\in(1,\infty)$,
the minimizer is unique up to scaling. 
\end{thm}

When $p=2$, the minimizers are called the \textit{maximal Chow-destabilizers}
for $X$. %
{} The existence is not straightforward since the unit sphere in $\mathcal{N}(V)$
is not compact unless $\dim V=1$ (see Remark \ref{rem:not loc cpt}).
For this, we use another locally compact topology under which $M_{X}$
is lsc. The uniqueness follows from the convexity of $M_{X}$ and
the strict convexity of $p$-norm when $p\in(1,\infty)$. 

\subsection*{Symmetries of maximal destabilizers}

The absolute uniqueness implies that the maximal Chow-destabilizer
can inherit all symmetries of embedding $X\subset\mathbb{P}V^{*}$. 
\begin{cor}[Theorem \ref{thm:sym of Chow-dest}]
 Let $X\subset\mathbb{P}V^{*}$ be a Chow-unstable projective manifold.
Suppose that $G$ is a group with a homomorphism $G\rightarrow\mathrm{GL}(V)$
such that $X\subset\mathbb{P}V^{*}$ is preserved by $G$. With respect
to the natural $G$-action on $\mathcal{N}(V)$, $M_{X}$ is $G$-invariant
and the maximal Chow-destabilizer is also $G$-invariant.
\end{cor}

For instance, let $(X,L)$ be the polarized toric manifold associated
to a Delzant polytope $P$. Take $G=T_{\bbc}$, the complex torus
and $V=R_{k}$ ($k\gg1$), which carries a $T_{\bbc}$-action. If
$X\hookrightarrow\mathbb{P}R_{k}^{*}$ is Chow-unstable, the above
corollary implies that the maximal Chow-destabilizer is $T_{\bbc}$-invariant.
Hence we only need to search the minimizer among the $T_{\bbc}$-invariant
norms. We will discuss this in detail in \cite{yao_quantizing_2025}. 

Similarly, the maximal K-destabilizer also inherits all symmetries
of $(X,L)$, see Theorem \ref{thm: sym of K-destab}. 

\subsection*{A route to connecting maximal Chow-destabilizers to K-destabilizers }

From now on, we suppose that $(X,L)$ is K-unstable. Then $X\hookrightarrow\mathbb{P}R_{k}^{*}$
is Chow-unstable for sufficiently large $k$. Up to scaling, Theorem
\ref{thm:Chow exist INTRO} yields a sequence of maximal Chow-destabilizers
$\chi_{k}\in\mathcal{N}(R_{k})$. Under some assumptions, we show
that under proper normalization, $\mathrm{FS}_{k}(\chi_{k})$ converges
to the maximal K-destabilizer, where $\mathrm{FS}_{k}$ is the NA
Fubini--Study map (\ref{eq: def FS_k}). This approach is based on
the Boucksom--Jonsson's non-Archimedean pluripotential theory \cite{boucksom_global_2022,boucksom_non-archimedean_2024,boucksom_non-archimedean_2023}. 

Let $(X^{\an},L^{\an})$ be the Berkovich analytification of $(X,L)$
over $\bbc$ with the trivial norm (see §\ref{subsec:Berkovich}).
As a topological space, $X^{\an}$ is Hausdorff and compact. Boucksom--Jonsson
\cite{boucksom_global_2022} define the NA psh metrics on $L^{\an}$
to be the decreasing limit of Fubini--Study type functions (\ref{eq: def FS metric}).
Denote by $\mathrm{PSH}^{\na}(L)$ ($\mathrm{CPSH}^{\na}(L)$) the
space of (continuous) psh metrics, we may omit ``$L$'' when it
is fixed. In \cite{boucksom_non-archimedean_2018}, they introduced
two maps which play the role of quantization maps in the Archimedean
setting, namely the NA Fubini--Study map and the super-norm map (see
§\ref{subsec:FS SN}): 
\[
\mathrm{FS}_{k}:\mathcal{N}(R_{k})\rightarrow\mathrm{CPSH}^{\na},\ \mathrm{SN}_{k}:\mathcal{L}^{\infty}\rightarrow\mathcal{N}(R_{k}),
\]
where $\mathcal{L}^{\infty}$ is the space of bounded functions on
$X^{\an}$. They behave similarly to the Legendre transform, see Proposition
\ref{prop: FS vs SN}. Using NA functionals, we obtain a formula for
the Chow-weight (\ref{eq:Chow w INTRO}). 
\begin{thm}[Theorem \ref{thm: NA express Chow}, \ref{thm: property of Chow-destab}]
Let $X\subset\mathbb{P}V^{*}$ be a projective manifold. 

(1) For any $\chi\in\mathcal{N}(V)$ we have 
\[
M_{X}(\chi)=\mathbb{E}\left(\mathrm{FS}(\chi)\right)-E_{V}(\chi),
\]
where $\mathbb{E}$ is the NA Monge--Ampère energy, $\mathrm{FS}$
is the NA Fubini--Study map (\ref{eq:def NA FS}) and $E_{V}(\chi)$
is the volume (\ref{eq: volume NA}), i.e. the average of jumping
numbers of $\chi$. 

(2) If $X$ is Chow-unstable and $\bar{\chi}$ is a maximal Chow-destabilizer,
then $\mathrm{SN}\circ\mathrm{FS}(\bar{\chi})=\bar{\chi}$. 
\end{thm}

(2) can be described as the concavity of maximal Chow-destabilizers.
Now we state our main result. Unfortunately, it relies on some idealistic
assumptions. 
\begin{thm}
\label{thm: main}Let $(X,L)$ be a K-unstable polarized manifold.
We make assumptions \textbf{A1}-\textbf{A4}, see §\ref{subsec: proof main}.
Let $\phi^{\flat}\in\mathrm{CPSH}^{\na}(L)$ be the maximal K-destabilizer
normalized by (\ref{eq:cano normal K}). When $k$ is sufficiently
large and divisible, the embedding image $X\hookrightarrow\mathbb{P}R_{k}^{*}$
is Chow-unstable, let $\chi_{k}^{\flat}\in\mathcal{N}(R_{k})$ be
the maximal Chow-destabilizer normalized by (\ref{eq:cano normal Chow}).
Then $\mathrm{FS}_{k}(\chi_{k}^{\flat})$ converges to $\phi^{\flat}$
in the sense specified by \textbf{A4}. 
\end{thm}

If the convergence holds, we can say the maximal Chow-destabilizers
quantize the maximal K-destabilizer. 

\subsubsection*{Discussion of the assumptions \& proof }

Recall that in the Archimedean setting, Donaldson \cite{donaldson_scalar_2005}
provided a variational perspective on the relationship between balanced
norms and cscK metrics, which is reviewed in §\ref{subsec: Donaldson work}.
Our approach is analogous to that.

Since $\phi^{\mathrm{K}}$ is induced from the maximal K-destabilizing
ray, first we need to re-characterize it as the minimizer of the normalized
NA K-energy $\overline{\mathbb{M}}$ (\ref{eq: min NA M}). Let $\mathcal{E}_{\na}^{1}\subset\mathrm{PSH}^{\na}(L)$
be the space of finite-energy metrics, and $\mathcal{R}^{1}$ be the
space of finite-energy geodesic rays in $\mathcal{E}^{1}(L)$. Then
we have BBJ's embedding $\iota:\mathcal{E}_{\na}^{1}\hookrightarrow\mathcal{R}^{1}/_{\sim}$,
where $\sim$ is the $d_{1}$-asymptotic relation. Assumption \textbf{A1}
says that 
\[
\mathcal{M}^{\mathrm{rad}}(\iota(\phi))=\mathbb{M}(\phi),\ \forall\phi\in\mathcal{E}_{\na}^{1}.
\]
This is a conjecture in \cite{li_geodesic_2022} and has been verified
very recently by \cites{boucksom_yau-tian-donaldson_2025} and \cite{darvas_ytd_2025}
by different method. Under \textbf{A1}, by a trick of Inoue \cite[§4.2.2]{inoue_entropies_2022-1},
we see that 
\begin{equation}
\inf\left\{ \breve{\mathbb{M}}(\phi)\coloneqq\mathbb{M}(\phi)+\frac{1}{2}\left\Vert \phi\right\Vert _{2}^{2}\mid\phi\in\mathcal{E}_{\na}^{2}\coloneqq\iota^{-1}\left(\mathcal{R}^{2}/_{\sim}\right)\right\} \label{eq: min M horn-INTRO}
\end{equation}
admits a unique minimizer $\phi^{\flat}\in\mathcal{E}_{\na}^{2}$,
where $\left\Vert \phi\right\Vert _{2}\coloneqq\left\Vert \iota(\phi)\right\Vert _{2}$
is the speed of geodesic ray and $\phi^{\flat}$ is the unique maximal
K-destabilizer normalized by the condition $\left\Vert \phi^{\flat}\right\Vert _{2}^{2}=-\mathbb{M}(\phi^{\flat})$.
The advantage of replacing $\overline{\mathbb{M}}$ by $\breve{\mathbb{M}}$
is that the latter preserves convexity and has an absolutely unique
minimizer.

Next we approximate $\breve{\mathbb{M}}$ by some quantized functionals.
\textbf{A3} assumes that for any $\phi\in\mathrm{CPSH}^{\na}$, we
have
\[
E_{k}\circ\mathrm{SN}_{k}(\phi)=\mathbb{E}(\phi)-\frac{1}{2}\mathbb{M}(\phi)k^{-1}+\mathrm{o}(k^{-1}),
\]
where $E_{k}$ is the normalized volume. The leading-term asymptotic
is due to \cite{boucksom_spaces_2021}. We show \textbf{A3} holds
for $\phi$ induced by an ample test-configuration, see Theorem \ref{thm: main expand for TC}.
In the toric setting, it reduces to the Riemann sum expansion for
convex functions. To use the above expansion, we have to assume \textbf{A2}:
$\phi^{\flat}$ is continuous. Note that in the toric setting, Li--Lian--Sheng
\cite{li_extremal_2023} showed that the optimal test-configuration
is a bounded concave function over the polytope. 

On the other hand, the $L^{2}$-norm $\left\Vert \phi\right\Vert _{2}$
can be quantized by the super-norms. 
\begin{thm}[Theorem \ref{thm:quantize L2 norm}]
For any $\phi\in\mathrm{CPSH}^{\na}(L)$, we have 
\[
\left\Vert \phi\right\Vert _{2}\coloneqq\left\Vert \iota(\phi)\right\Vert _{2}=\lim_{k\rightarrow\infty}\frac{1}{k}\left\Vert \mathrm{SN}_{k}(\phi)\right\Vert _{2}.
\]
\end{thm}

We approximate (\ref{eq: min M horn-INTRO}) by 
\begin{equation}
\inf\left\{ \breve{\mathbb{Q}}_{k}(\phi)\coloneqq2k\left(\mathbb{E}(\phi)-E_{k}\circ\mathrm{SN}_{k}(\phi)\right)+\frac{1}{2k^{2}}\left\Vert \mathrm{SN}_{k}(\phi)\right\Vert _{2}^{2}\mid\phi\in\mathrm{CPSH}^{\na}\right\} .\label{eq: approx problem INTRO}
\end{equation}
Under \textbf{A3}, $\breve{\mathbb{Q}}_{k}(\phi)$ converges to $\breve{\mathbb{M}}(\phi)$
for $\phi\in\mathrm{CPSH}^{\na}$. By the intertwining properties
between $\mathrm{FS}_{k}$ and $\mathrm{SN}_{k}$, we find that (\ref{eq: approx problem INTRO})
can be reduced into a finite-dimensional problem 
\[
\inf\left\{ \breve{M}_{k}(\chi)\coloneqq2k\left(\mathbb{E}\circ\mathrm{FS}_{k}(\chi)-E_{k}(\chi)\right)+\frac{1}{2k^{2}}\left\Vert \chi\right\Vert _{2}^{2}\mid\chi\in\mathcal{SN}_{k}\right\} ,
\]
where $\mathcal{SN}_{k}\subset\mathcal{N}(R_{k})$ is the image set
of $\mathrm{SN}_{k}$. Note that the first term of $\breve{M}_{k}$
is the NA expression for Chow-weight. This problem is equivalent to
(\ref{eq:min Chow INTRO}, $V=R_{k}$) plus a constraint: $\chi\in\mathcal{SN}_{k}$.
But in Theorem \ref{thm: property of Chow-destab}, we show that maximal
Chow-destabilizers automatically belong to $\mathcal{SN}_{k}$, so
this constraint is actually superfluous. 

It turns out that (\ref{eq: approx problem INTRO}) has a unique minimizer
$\mathrm{FS}_{k}(\chi_{k}^{\flat})$. Finally, \textbf{A4} assumes
that this minimizer sequence converges to the minimizer of (\ref{eq: min M horn-INTRO}),
i.e. $\phi^{\flat}$. An evidence for \textbf{A4} is that $\mathbb{Q}_{k}$
(\ref{eq:  NA Qk}) is convex, see Proposition \ref{prop:Qk convex}.

\subsection*{Recent related works \& questions }

There is a research trend focused on describing the limit behavior
of a geometric flow by some algebraic data. Besides the Calabi flow
\& maximal K-destabilizer, there are many instances (not completed): 

(1) normalized Kähler--Ricci flow on Fano manifolds $(X,-K_{X})$;
After the affirmation of Hamilton--Tian conjecture, the picture has
become relatively complete. Flow will converge (in the sense of Gromov--Hausdorff)
to a Kähler--Ricci soliton on $\mathbb{Q}$-Fano variety $X_{\infty}$,
and $X_{\infty}$ is the two-steps degeneration of $X$ induced by
the unique special $\bbr$-test-configuration which minimizes the
NA $H$-functional $H^{\na}$, refer to \cite[etc]{chen_kahlerricci_2018,dervan_kahlerricci_2020,han_algebraic_2024,blum_existence_2023}.
For varieties with large symmetries, see \cite[etc]{delcroix_k-stability_2020,li_k-stability_2022,li_singular_2024,li_equivariant_2023-1,wang_kahlerricci_2024}. 

(2) Inverse Monge--Ampère flow on Fano manifolds, which is introduced
in \cite{collins_inverse_2022} as the gradient flow of Ding functional.
It is expected that the limit behavior is encoded in the minimizer
of normalized NA Ding functional $D^{\na}$. Does a similar version
of the Hamilton--Tian conjecture hold? This is not clear in the toric
setting, refer to \cites(etc){xia_sharp_2021}{yao_mabuchi_2022}{collins_inverse_2022}{hisamoto_geometric_2023}{klemyatin_convergence_2024}. 

(3) Wess--Zumino--Witten equation \& Harder--Narasimhan potentials,
see \cite{finski_about_2024}, which is a generalization of the Hermitian--Einstein
metric problem. 

As we did for the NA K-energy, one can consider the quantization problem
for the minimizers of $H^{\na}$ and $D^{\na}$. For $H^{\na}$, Hisamoto
\cite{hisamoto_quantization_2024} considered the quantization of
$H$-entropy and its NA counterpart ($S_{k}^{\na}$ in \cite{hisamoto_quantization_2024}),
and he established a quantum moment-weight equality. 

Inoue \cites{inoue_entropies_2022-1}{inoue_entropies_2022}{inoue_toric_2023}
studied the optimization problem for his NA $\mu$-entropy, which
relates to the optimization problem for $H^{\na}$ and $\breve{\mathbb{M}}$,
see \cite[§4.3]{inoue_entropies_2022}. 

\subsection*{Notations \& conventions}

Abbreviations: NA (non-Archimedean); psh (plurisubharmonic); MA (Monge--Ampère);
FS (Fubini--Study); SN (super-norm); GIT (geometric invariant theory).
In order to indicate the analogy between various functionals, we denote
them by the same letter but different fonts in the following way. 
\begin{center}
\begin{tabular}{|c|c|c|}
\hline 
Functionals & Classical & Quantum\tabularnewline
\hline 
Archimedean & $\mathcal{E}$ (\ref{eq: MA energy}) $\mathcal{M}$ (\ref{eq: K-energy})
$\mathcal{Q}_{k}$ (\ref{eq: Archi Qk}) & $\mathbf{E}_{V}$ (\ref{eq: Archi E_V}) $\mathbf{M}_{X}$ (\ref{eq:KN final})\tabularnewline
\hline 
non-Archimedean & $\mathbb{E}$ (\ref{eq: NA MA}) $\mathbb{M}$ (\ref{eq: NA K-energy})
$\mathbb{Q}_{k}$ (\ref{eq:  NA Qk}) & $E_{V}$ (\ref{eq: volume NA}) $M_{X}$ (\ref{eq:def Chow weight})\tabularnewline
\hline 
\end{tabular}
\par\end{center}

We use $\mathbf{FS}_{k}$ and $\mathrm{FS}_{k}$ to denote the Archimedean
and NA Fubini--Study map respectively. In this paper, $\mathcal{R}^{p}(L)$
denotes the space of all $L^{p}$-geodesic rays in the finite-energy
space $\mathcal{E}^{p}(L)$ with \textit{free} starting-point. 

\subsection*{Acknowledges }

We are grateful to Chi Li, Minghao Miao, Li Sheng and Yalong Shi for
inspiring discussions and valuable suggestions. Special thanks to
Mingchen Xia for technical support, and Siarhei Finski for sharing
his experience on the study of Wess--Zumino--Witten equations. The
author is supported by the National Natural Science Foundation of
China (NSFC) with Grant No. 12231006. 

\section{Space of Hermitian norms and non-Archimedean norms}

In this section, we review the space $\mathcal{H}(V)$ of all Hermitian
norms on a vector space $V$ and its asymptotic cone $\hat{\mathcal{H}}(V)$.
We verify that $\hat{\mathcal{H}}(V)$ is isometric to the space $\mathcal{N}(V)$
of NA norms on $V$, this is mentioned in the survey \cites{boucksom_variational_2019}.
We will take $\mathcal{H}(V)$ as the domain of the Kempf-Ness functional
for Chow-stability, and $\hat{\mathcal{H}}(V)$ ($\cong\mathcal{N}(V)$)
as the domain of the Chow-weight. The main references of this section
are \cites{boucksom_variational_2019}{boucksom_spaces_2021}{darvas_quantization_2020}. 

\subsection{Symmetric space of Hermitian norms }

Let $V$ be a complex vector space of dimension $N$, which will be
the space of global sections $\mathrm{H}^{0}(X,kL)$. A Hermitian
form $H:V\times V\rightarrow\mathbb{C}$ is anti-linear w.r.t. the
second variable. If $H(x,x)>0$ for all $x\neq0$, we call it a Hermitian
norm, and denote $H(x)\coloneqq H(x,x)^{1/2}$. Let $\mathrm{Herm}(V)$
be the real vector space of all Hermitian forms on $V$. Hermitian
norms constitute an open subset $\mathcal{H}(V)\subset\mathrm{Herm}(V)$.
Since the general linear group $\mathrm{GL}(V)$ acts transitively
on $\mathcal{H}(V)$, $\mathcal{H}(V)$ is homeomorphic to $\mathrm{GL}(N,\mathbb{C})/\mathrm{U}(N)$. 

The tangent space of $\mathcal{H}(V)$ at each point $H$ is $\mathrm{Herm}(V)$,
which can be identified with the space of $H$-self-adjoint operators.
For $S\in\mathrm{Herm}(V)$, the associated operator $A:V\rightarrow V$
is defined by 
\[
S(x,y)=H(Ax,y),\ \forall x,y\in V.
\]
We usually denote $A$ by the matrix notation $A=SH^{-1}$. There
is a canonical Riemannian metric $g$ on $\mathcal{H}(V)$ defined
by 
\[
g|_{H}\left(A_{1},A_{2}\right)\coloneqq\frac{1}{N}\mathrm{tr}(A_{1}A_{2}),\ A_{1},A_{2}\in T_{H}\mathcal{H}(V).
\]
One can check that $(\mathcal{H}(V),g)$ is a contractible, complete
Riemannian symmetric space with non-positive sectional curvature,
see \cite{eberlein_geometry_1996}. Group $\mathrm{GL}(V)$ acts on
$\mathcal{H}(V)$ by isometries. 

The geodesic ray starting from $H_{0}$ with tangent vector $A$ is
given by 
\[
H_{t}(x,y)=H_{0}(e^{tA}x,y),\ \forall x,y\in V.
\]
Any two points $H_{0},H_{1}\in\mathcal{H}(V)$ are connected by a
unique geodesic. Let $A$ be the unique $H_{0}$-self-adjoint operator
such that $e^{A}=H_{1}H_{0}^{-1}$, then the geodesic between $H_{0}$
and $H_{1}$ is 
\begin{equation}
H_{t}(x,y)=H_{0}(e^{tA}x,y),\ \forall x,y\in V,\ t\in[0,1].\label{eq: Her norm geodesic}
\end{equation}
The induced length metric is 
\[
d(H_{0},H_{1})^{2}=\frac{1}{N}\mathrm{tr}(A^{2}),\ e^{A}=H_{1}H_{0}^{-1}.
\]
More generally, for each $p\in[1,\infty]$, there is a Finsler metric
on $\mathcal{H}(V)$. For $A\in T_{H}\mathcal{H}(V)$ with eigenvalues
$\lambda_{i}$, we define Finsler norms 
\begin{equation}
\left\Vert A\right\Vert _{p,H}\coloneqq\left(\frac{1}{N}\sum_{i}\left|\lambda_{i}\right|^{p}\right)^{1/p},\ \left\Vert A\right\Vert _{\infty,H}\coloneqq\max_{i}\left|\lambda_{i}\right|.\label{eq: Finsler norm}
\end{equation}
Under these Finsler norms, (\ref{eq: Her norm geodesic}) are also
the shortest path between two points. For $H_{0},H_{1}\in\mathcal{H}(V)$,
let $e^{\mu_{i}}$ be the eigenvalues of operator $H_{1}H_{0}^{-1}$,
then the induced length metric $d_{p}$ is 
\[
d_{p}(H_{0},H_{1})^{p}=\frac{1}{N}\sum_{i=1}^{N}\left|\mu_{i}\right|^{p},\ d_{\infty}(H_{0},H_{1})\coloneqq\max_{1\leq i\leq N}\left|\mu_{i}\right|.
\]
Among these metrics, only $d_{2}=d$ makes $\mathcal{H}(V)$ a CAT(0)
space. Metric $d_{\infty}$ is called the Goldman--Iwahori distance,
which is the smallest $C>0$ such that $e^{-C}H_{0}\le H_{1}\leq e^{C}H_{0}$.
We have comparison 
\begin{equation}
N^{-\frac{1}{2}}d_{\infty}\leq d\leq d_{\infty}.\label{eq: compare dp}
\end{equation}

\subsubsection{Relative volumes }

A norm $H\in\mathcal{H}(V)$ induces a norm on the top-wedge line
$\wedge^{N}V$, which is denoted by $\det H$. For any basis $(s_{i})$
of $V$, we have 
\[
\left(\det H\right)(s_{1}\wedge\cdots\wedge s_{N})^{2}=\det[H(s_{i},s_{j})].
\]
For two norms $H_{0}$, $H_{1}\in\mathcal{H}(V)$, the relative volume
is defined as 
\begin{equation}
\mathbf{E}_{V}(H_{1},H_{0})\coloneqq-\frac{1}{N}\log\left(\frac{\det H_{1}}{\det H_{0}}\right)=-\frac{1}{N}\log\left(\frac{\det[H_{1}(s_{i},s_{j})]}{\det[H_{0}(s_{i},s_{j})]}\right).\label{eq: Archi E_V}
\end{equation}
So if $(s_{i})$ is a $H_{0}$-orthonormal basis such that $H_{1}(s_{i},s_{j})=e^{\mu_{i}}\delta_{ij}$,
then $\mathbf{E}_{V}(H_{1},H_{0})=-\frac{1}{N}\sum\mu_{i}.$ By Cauchy--Schwartz
inequality, we have 
\begin{equation}
\left|\mathbf{E}_{V}(H_{1},H_{0})\right|\leq d(H_{1},H_{0}).\label{eq: Lip of EV}
\end{equation}

\subsection{Asymptotic cone of $\mathcal{H}(V)$}

For a general CAT(0) space $\mathcal{X}$, one can consider all the
geodesic rays modulo the asymptotic relations. It constitutes another
CAT(0) space, called the asymptotic cone (or infinity cone) of $\mathcal{X}$,
refer to \cites{bridson_metric_2013}. Here we specialize to the case
$\mathcal{X}=(\mathcal{H}(V),d)$. 

Recall some notions in metric geometry. Given a metric space $(X,d)$,
a geodesic segment is a map $\gamma:[0,1]\rightarrow X$ such that
$d(\gamma_{s},\gamma_{t})=d(\gamma_{0},\gamma_{1})\left|s-t\right|$
for any $s,t\in[0,1]$. We say $(X,d)$ is a geodesic space if any
two points can be connected by a geodesics segment. A \emph{CAT(0)
space} is a geodesic space satisfying a triangle comparison property.
By \cite[Theorem 1.3.3]{bacak_convex_2014}, this property is equivalent
to that for any geodesic $\gamma:[0,1]\rightarrow X$ and point $x\in X$,
we have 
\[
d(x,\gamma_{t})^{2}\leq(1-t)d(x,\gamma_{0})^{2}+t\cdot d(x,\gamma_{1})^{2}-t(1-t)d(\gamma_{0},\gamma_{1})^{2},\ \forall t\in[0,1].
\]
In particular, the function $d(x,\cdot)^{2}$ is strictly convex.
This inequality implies Busemann's convexity, i.e. for any two geodesics
$\gamma$ and $\gamma'$, the function $t\mapsto d\left(\gamma_{t},\gamma'_{t}\right)$
is convex. Busemann's convexity implies the uniqueness of the geodesic
between two points. A complete CAT(0) space is called a \emph{Hadamard
space}, such as $\mathcal{H}(V)$. 

Given two geodesic rays $(H_{t})_{t\geq0}$ and $(G_{t})_{t\geq0}$
in $\mathcal{H}(V)$, which can be stationary points. Busemann's convexity
implies that $t\mapsto d(H_{t},G_{t})$ is convex. Moreover, by the
triangle inequality, we have $d(H_{t},G_{t})\leq At+B$ for some constants
$A,B>0$. Hence the limit slope 
\begin{equation}
\hat{d}\left((H_{t}),(G_{t})\right)\coloneqq\lim_{t\rightarrow+\infty}\frac{d(H_{t},G_{t})-d(H_{0},G_{0})}{t}\uparrow\label{eq: d hat}
\end{equation}
exists. It is easy to check that $\hat{d}$ satisfies the triangle
inequality. If $\hat{d}=0$, we say $(H_{t})$ and $(G_{t})$ are
asymptotic to each other, this is an equivalent relation between geodesic
rays. Note that by the convexity, $\hat{d}=0$ is equivalent to the
boundness of $d(H_{t},G_{t})$. An equivalence class is called an
\emph{asymptotic class}, which will be denoted by $\eta,\xi$, etc.
Let $\hat{\mathcal{H}}(V)$ be the set of asymptotic classes, then
$\hat{d}$ descends to a metric on $\hat{\mathcal{H}}(V)$. By a general
theorem \cite[Theorem 4.8]{ballmann_lectures_1995}, $\left(\hat{\mathcal{H}}(V),\hat{d}\right)$
is also a Hadamard space. 

Denote by $o$ the class of stationary rays. Then $\hat{d}(\eta,o)$
is the speed of rays in class $\eta$. Note that $\left(\hat{\mathcal{H}}(V),\hat{d}\right)$
is not locally compact when $\dim V>1$, e.g. sphere $\{\eta\mid\hat{d}(\eta,o)=\delta\}$
is not compact, see Remark \ref{rem:not loc cpt} and the isometry
(\ref{eq: Isometry}). 

\subsubsection{Negatively curved path of Hermitian norms}

Although we mainly deal with geodesics, many results hold for more
general paths with curvature condition. We adopt the curvature condition
in \cite{berndtsson_lelong_2020}. 

For an open interval $I\subset\mathbb{R}$, we define a domain 
\begin{equation}
D_{I}\coloneqq\left\{ \tau\in\mathbb{C}^{*}\mid-\log\left|\tau\right|^{2}\in I\right\} .\label{eq:domain DI}
\end{equation}
A path $(H_{t})_{t\in I}\subset\mathcal{H}(V)$ defines a metric $h$
on the trivial bundle $V\times D_{I}\rightarrow D_{I}$ by letting
$h|_{V\times\{\tau\}}=H_{-\log\left|\tau\right|^{2}}$. We say a smooth
path $(H_{t})_{t\in I}$ is \emph{negatively (positively) curved}
if the Chern curvature tensor $\Theta_{h}\leq0\ (\geq0)$ in the sense
of Griffiths (see \cite[Chapter VII 6]{demailly_complex_2012}). Path
$(H_{t})$ is negatively curved if and only if the path of dual norms
$(H_{t}^{*})\subset\mathcal{H}(V^{*})$ is positively curved. 

Fix a basis of $V$, the associated Gram matrix is still denoted by
$H_{t}$. By a direct computation, $(H_{t})_{t\in I}$ is negatively
(positively) curved if and only if 
\begin{equation}
\ddot{H}-\dot{H}H^{-1}\dot{H}\geq0\ (\leq0),\ \forall t\in I,\label{eq:curv condition matrix}
\end{equation}
where $\dot{H}=\frac{\mathrm{d}}{\mathrm{d}t}H_{t}$ and $\geq0$
means positive semi-definite. 
\begin{prop}
\label{prop: logH(s) convex}For a negatively curved smooth path $(H_{t})_{t\in I}\subset\mathcal{H}(V)$
and $s\in V\backslash\{0\}$, the function $t\mapsto\log H_{t}(s)$
is convex. 
\end{prop}

\begin{proof}
Denote $f(t)=\log H_{t}(s,s)$. For any $t_{0}\in I$, we show $f''(t_{0})\geq0$.
Take a $H_{t_{0}}$-orthonormal basis $(s_{i})$ such that $\dot{H}_{t_{0}}=\mathrm{diag}\left(\lambda_{1},\cdots,\lambda_{N}\right)$.
Curvature condition (\ref{eq:curv condition matrix}) implies that
\[
a\ddot{H}_{t_{0}}a^{*}\geq\sum_{i=1}^{N}\lambda_{i}^{2}\left|a_{i}\right|^{2},\ \forall a=(a_{1},\cdots,a_{N})\in\mathbb{C}^{N}.
\]
Rescaling $s$ such that $H_{t_{0}}(s)=1$. Let $s=\sum_{i}a_{i}s_{i}$,
then $f(t)=\log aH_{t}a^{*}$ and $\sum\left|a_{i}\right|^{2}=1$.
By the above inequality, we have 
\[
f''(t_{0})=a\ddot{H}_{t_{0}}a^{*}-\left(a\dot{H}_{t_{0}}a^{*}\right)^{2}\geq\sum_{i}\lambda_{i}^{2}\left|a_{i}\right|^{2}-\left(\sum_{i}\lambda_{i}\left|a_{i}\right|^{2}\right)^{2}\geq0,
\]
where $\geq0$ is due to the Cauchy--Schwartz inequality. 
\end{proof}
\begin{rem}
Things are not symmetric, for a positively curved path $(H_{t})$,
the function $t\mapsto\log H_{t}(s)$ may not be concave. Besides
that, the convexity of $\log H_{t}(s)$ for all $s\in V\backslash\{0\}$
is not enough to ensure the path is negatively curved. Actually, the
latter is equivalent to that for any non-vanishing holomorphic section
of $V\times D_{I}$ (i.e. a holomorphic map $s:D_{I}\rightarrow V\backslash\{0\}$),
the function $D_{I}\ni\tau\mapsto\log h\left(s(\tau)\right)$ is subharmonic,
refer to \cite{raufi_singular_2015} §2. 
\end{rem}

If $(E,h)$ is a Hermitian holomorphic bundle with semi-positive (semi-negative)
curvature (in the sense of Griffiths), then so is the line bundle
$(\det E,\det h)$. Take $E=V\times D_{I}$, if $(H_{t})_{t\in I}\subset\mathcal{H}(V)$
is positively (negatively) curved, then the function $t\mapsto\mathbf{E}_{V}(H_{t},H_{r})$
(\ref{eq: Archi E_V}) is convex (concave) for a fixed reference norm
$H_{r}$. 

\subsection{Space of non-Archimedean norms}

The asymptotic cone $\hat{\mathcal{H}}(V)$ can be identified with
the space $\mathcal{N}(V)$ of NA norms on $V$ over the trivially-normed
$\bbc$. Refer to \cite{boucksom_spaces_2021} for a comprehensive
discussion on such spaces. 

Let $\left|\cdot\right|_{tr}$ be the trivial norm on $\mathbb{C}$,
i.e. $\left|a\right|_{tr}=1$ for all $a\in\bbc^{*}$. A \emph{non-Archimedean
norm} is a function $\left\Vert \cdot\right\Vert :V\rightarrow\mathbb{R}_{\geq0}$
satisfying 
\[
\left\Vert as\right\Vert =\left|a\right|_{tr}\left\Vert s\right\Vert ,\ \left\Vert s+s'\right\Vert \leq\max\left(\left\Vert s\right\Vert ,\left\Vert s'\right\Vert \right),\ \forall a\in\bbc,\ s,s'\in V
\]
and $\left\Vert s\right\Vert =0$ only when $s=0$. Follow \cite{boucksom_non-archimedean_2024},
we use the valuation notation. Let 
\[
\chi(\cdot)\coloneqq-\log\left\Vert \cdot\right\Vert \in\mathbb{R}\cup\{+\infty\},
\]
then we have 
\[
\chi(as)=\chi(s),\ \forall a\in\bbc^{*},\ \chi(s+s')\geq\min(\chi(s),\chi(s')).
\]
Let $\mathcal{N}(V)$ be the set of all NA norms on $V$. There is
a trivial norm $\chi_{tr}$ such that $\chi_{tr}(s)=0$ for all $s\neq0$.
For any $a>0$ and $b\in\bbr$, $a\chi+b$ is also a NA norm. Note
that for $\chi_{1},\chi_{2}\in\mathcal{N}(V)$, $\chi_{1}+\chi_{2}$
may not be a NA norm. 

A NA norm $\chi$ induces a left-continuous non-increasing filtration
$(F^{\lambda}V)_{\lambda\in\mathbb{R}}$ of $V$ and vice versa, 
\[
F^{\lambda}V=\left\{ s\in V\mid\chi(s)\geq\lambda\right\} ,\ \chi(s)=\sup\left\{ \lambda\mid s\in F^{\lambda}V\right\} .
\]
The \emph{jumping numbers} $\lambda_{1}(\chi)\geq\cdots\geq\lambda_{N}(\chi)$
of filtration are defined as 
\[
\lambda_{i}(\chi)\coloneqq\max\{\lambda\mid\dim F^{\lambda}V\geq i\}.
\]
We say that $\chi$ is diagonalized by a basis $(s_{i})$ of $V$
if we have 
\[
\chi\left(\sum_{i=1}^{N}a_{i}s_{i}\right)=\min_{1\leq i\leq N}\chi(a_{i}s_{i})=\min_{a_{i}\neq0}\chi(s_{i}),\ \forall(a_{i})\in\bbc^{N}.
\]
At this time, the jumping numbers are the non-increasing rearrangement
of $\chi(s_{i})$. By \cite[Proposition 1.14]{boucksom_spaces_2021},
any two norms $\chi,\chi'$ can be diagonalized by a common basis
$(s_{i})$. Then we arrange the numbers $\chi(s_{i})-\chi'(s_{i})$
as follows 
\[
\lambda_{1}(\chi,\chi')\geq\cdots\geq\lambda_{N}(\chi,\chi'),
\]
which are independent of the choices of $(s_{i})$. They are called
the \emph{relative spectra}. In particular, $\lambda_{i}(\chi,\chi_{tr})=\lambda_{i}(\chi)$. 

By \cite[Theorem 3.1]{boucksom_spaces_2021}, for each $p\in[1,\infty]$,
there is a metric $d_{p}$ on $\mathcal{N}(V)$ defined by 
\begin{equation}
d_{p}(\chi,\chi')^{p}\coloneqq\frac{1}{N}\sum_{i=1}^{N}\left|\lambda_{i}(\chi,\chi')\right|^{p},\ d_{\infty}(\chi,\chi')\coloneqq\max_{1\leq i\leq N}\left|\lambda_{i}(\chi,\chi')\right|.\label{eq: NA dp}
\end{equation}
They are Lipschitz equivalent to each other. We define the ``$p$-norm''
of a NA norm by 
\begin{equation}
\left\Vert \chi\right\Vert _{p}\coloneqq d_{p}(\chi,\chi_{tr})=\left(\frac{1}{N}\sum_{i=1}^{N}\left|\lambda_{i}(\chi)\right|^{p}\right)^{1/p},\ \left\Vert \chi\right\Vert _{\infty}\coloneqq\max_{1\leq i\leq N}\left|\lambda_{i}(\chi)\right|.\label{eq:p norm}
\end{equation}
The relative volume and the volume are defined as 
\begin{equation}
E_{V}(\chi,\chi')\coloneqq\frac{1}{N}\sum_{i=1}^{N}\lambda_{i}(\chi,\chi'),\ E_{V}(\chi)\coloneqq E_{V}(\chi,\chi_{tr})=\frac{1}{N}\sum_{i=1}^{N}\lambda_{i}(\chi).\label{eq: volume NA}
\end{equation}

\begin{rem}
\label{rem:not loc cpt}When $\dim V>1$, for any $p\in[1,\infty]$,
the metric space $(\mathcal{N}(V),d_{p})$ is not locally compact.
For any nonzero $s\in V$, we define $\chi_{s}\in\mathcal{N}(V)$
by letting $\chi_{s}(s)=\delta>0$ and $\chi_{s}(s')=0$ when $s'\notin\mathbb{C}s$.
The associated filtration is $V\supset\mathbb{C}s\supset\{0\}$ with
jumping number $\lambda_{1}=\delta>0=\lambda_{2}=\cdots=\lam_{N}$.
For $p\in[1,\infty)$ ($p=\infty$ is similar), we have $\left\Vert \chi_{s}\right\Vert _{p}=\delta N^{-1/p}$.
If $s,s'\in V$ are linearly independent, expand them to a basis $\{s,s',s_{3},\cdots,s_{N}\}$,
which diagonalizes both $\chi_{s}$ and $\chi_{s'}$. It is easy to
see $d_{p}(\chi_{s},\chi_{s'})=2^{1/p}\delta N^{-1/p}$. It follows
that for any $\epsilon>0$, we can find a sequence in $\left\{ \chi:\left\Vert \chi\right\Vert _{p}\leq\epsilon\right\} $
which has no convergent subsequence. Hence $\chi_{tr}$ has no pre-compact
neighborhood.
\end{rem}

\subsection{The limit NA norm associated to a ray}

A geodesic ray in $\mathcal{H}(V)$ induces a NA norm on $V$. This
construction is extended by Berndtsson \cite{berndtsson_lelong_2020}
to negatively curved rays $(H_{t})_{t>0}$ with moderate growth, where
\textit{moderate growth} means that $H_{t}\leq e^{Ct}H_{1}$ for $t\gg1$
and $C>0$ is a constant. By Proposition \ref{prop: logH(s) convex},
for any $s\in V\backslash\{0\}$, $f(t)=\log H_{t}(s)^{2}$ is convex.
Growth condition implies that $f(t)\leq at+C$ for $t\gg1$. Hence
the limit slope 
\begin{equation}
-\chi_{H}(s)\coloneqq\lim_{t\rightarrow\infty}\frac{f(t)-f(0)}{t}\uparrow=\lim_{t\rightarrow\infty}\frac{1}{t}\log H_{t}(s)^{2}\label{eq:def NA limit}
\end{equation}
exits and is finite. Set $\chi_{H}(0)=+\infty$. The triangle inequality
for $H_{t}$ implies that 
\[
\chi_{H}(s+s')\geq\min\left(\chi_{H}(s),\chi_{H}(s')\right),\ \forall s,s'\in V.
\]
Hence $\chi_{H}$ defines a NA norm on $V$. 
\begin{defn}[limit NA norm]
For a negatively curved ray $(H_{t})_{t>0}\subset\mathcal{H}(V)$
with moderate growth, we call $\chi_{H}$ (\ref{eq:def NA limit})
the \emph{limit NA norm} of $(H_{t})_{t>0}$. 
\end{defn}

If two such rays $H_{t},G_{t}$ are asymptotic to each other (i.e.
$d(H_{t},G_{t})\leq C$), then they induce the same NA norm. Since
by (\ref{eq: compare dp}), we have $e^{-C}G_{t}\le H_{t}\leq e^{C}G_{t}$
for some constant $C$. We denote by $\chi_{\eta}$ the limit NA norm
associated to an asymptotic class $\eta$. 
\begin{example}
\label{exa: NAlimit of geodesic}Determine the limit NA norm of geodesic
ray $H_{t}(\cdot,\cdot)=H_{0}(e^{-tA}\cdot,\cdot)$, where $A$ is
$H_{0}$-self-adjoint. Take a $H_{0}$-orthonormal basis $(s_{i})$
such that $As_{i}=\lambda_{i}s_{i}$, then $H_{t}(s_{i},s_{j})=e^{-t\lambda_{i}}\delta_{ij}$.
For a nonzero $s=\sum a_{i}s_{i}$, we have 
\[
\chi_{H}(s)=-\lim_{t\rightarrow\infty}\frac{1}{t}\log H_{t}(s)^{2}=-\lim_{t\rightarrow\infty}\frac{1}{t}\log\sum_{i}\left|a_{i}\right|^{2}e^{-t\lambda_{i}}=\min_{a_{i}\neq0}\lambda_{i}.
\]
Hence $\chi_{H}$ is diagonalized by $(s_{i})$ and $\chi_{H}(s_{i})=\lambda_{i}$.
The associated filtration is 
\begin{equation}
\left\{ s\in V\mid\chi_{H}(s)\geq\lambda\right\} =\mathrm{span}\left\{ s_{i}\mid\lambda_{i}\geq\lambda\right\} =\bigoplus_{\mu\geq\lambda}E_{A}^{\mu},\label{eq:eigen filtra}
\end{equation}
where $E_{A}^{\mu}\coloneqq\{s\in V\mid As=\mu s\}$ is the $\mu$-eigenspace
of $A$.

The following fact is sketched in \cite{boucksom_variational_2019}.
We give a detailed proof. 
\end{example}

\begin{thm}
The NA norm associated to an asymptotic class of norm geodesic rays
defines an isometry 
\begin{equation}
\mathrm{NAlim}:\left(\hat{\mathcal{H}}(V),\hat{d}\right)\rightarrow\left(\mathcal{N}(V),d_{2}\right),\ \eta\mapsto\chi_{\eta},\label{eq: Isometry}
\end{equation}
where $\hat{d}$ is (\ref{eq: d hat}) and $d_{2}$ is (\ref{eq: NA dp}).
In particular, the target space is also a Hadamard space. 
\end{thm}

\begin{proof}
First we show $\mathrm{NAlim}$ is surjective. For any $\chi\in\mathcal{N}(V)$
and $H_{0}\in\mathcal{H}(V)$, let $E^{\lambda}$ be the $H_{0}$-orthogonal
complement of the subspace 
\[
F^{>\lambda}\coloneqq\left\{ s\in V\mid\chi(s)>\lambda\right\} \ \textrm{in}\ F^{\lambda}\coloneqq\left\{ s\in V\mid\chi(s)\geq\lambda\right\} .
\]
Then we have $V=\bigoplus_{\lambda}E^{\lambda}$, here $\lambda$
are the jumping numbers of $\chi$. Let $A:V\rightarrow V$ acts on
$E^{\lambda}$ by multiplying $\lambda$. Then by Example \ref{exa: NAlimit of geodesic},
$\chi$ is the limit NA norm of geodesic ray $H_{t}=H_{0}(e^{-tA}\cdot,\cdot)$.

Next we show $\mathrm{NAlim}$ is injective. If two geodesic rays
$H_{t},G_{t}$ such that $\chi_{H}=\chi_{G}$, we show that they are
asymptotic to each other. By a parallel translation, we can assume
$H_{0}=G_{0}$ and $H_{t}=H_{0}(e^{-tA}\cdot,\cdot)$, $G_{t}=H_{0}(e^{-tB}\cdot,\cdot)$,
where $A,B$ are $H_{0}$-self-adjoint operators. By Example \ref{exa: NAlimit of geodesic},
$A$ and $B$ have the same eigenvalues $\lambda_{1}\geq\cdots\geq\lambda_{N}$,
and $\lambda_{i}$ are the jumping numbers of $\chi_{H}$. Take two
$H_{0}$-orthonormal basis $(s_{i}^{A})$ and $(s_{i}^{B})$ which
diagonalize $A$ and $B$ respectively. Let $P$ be the unitary matrix
such that 
\[
\left(s_{1}^{B},\cdots,s_{N}^{B}\right)=\left(s_{1}^{A},\cdots,s_{N}^{A}\right)P.
\]
By (\ref{eq:eigen filtra}), we know $P$ is upper triangular. Combining
with $P^{-1}=\bar{P}^{\mathrm{T}}$, it implies that $P$ is diagonal.
Thus $s_{i}^{B}=P_{ii}s_{i}^{A}$, and the basis $(s_{i}^{A})$ can
diagonalize both $A$ and $B$. Since $A,B$ have the same eigenvalues,
it implies $A=B$, so $H_{t}=G_{t}$ for all $t$. 

Finally, we check that $\mathrm{NAlim}$ preserves the metrics. For
any $\chi,\chi'\in\mathcal{N}(V)$, take a basis $(s_{i})$ diagonalizing
both of them and let $\lambda_{i}=\chi(s_{i})$ and $\mu_{i}=\chi'(s_{i})$.
Let $H_{0}\in\mathcal{H}(V)$ taking $(s_{i})$ as an orthonormal
basis. Define operators $A$ and $B$ such that $As_{i}=\lambda_{i}s_{i}$
and $Bs_{i}=\mu_{i}s_{i}$ for all $1\leq i\leq N$, then geodesic
rays 
\[
H_{t}=H_{0}(e^{-tA}\cdot,\cdot)\ \textrm{and}\ G_{t}=H_{0}(e^{-tB}\cdot,\cdot)
\]
induce NA norms $\chi$ and $\chi'$ respectively. Note that $G_{t}=H_{t}(e^{tA}e^{-tB}\cdot,\cdot)$
and the eigenvalues of $e^{tA}e^{-tB}$ are $\exp t(\lambda_{i}-\mu_{i})$,
thus we have 
\[
\hat{d}\left((H_{t}),(G_{t})\right)\coloneqq\lim_{t\rightarrow+\infty}\frac{1}{t}d(H_{t},G_{t})=\left(\frac{1}{N}\sum_{i=1}^{N}(\lambda_{i}-\mu_{i})^{2}\right)^{\frac{1}{2}}=d_{2}(\chi,\chi').
\]
\end{proof}

\subsubsection{Geodesics in $\mathcal{N}(V)$ and $\hat{\mathcal{H}}(V)$}

Consider the geodesics in $\mathcal{N}(V)$. For any $\chi_{0},\chi_{1}\in\mathcal{N}(V)$,
take a basis $(s_{i})$ diagonalizing both of them. Let $\chi_{s}$
($s\in[0,1]$) be the norm diagonalized by $(s_{i})$ and such that
\begin{equation}
\chi_{s}(s_{i})=(1-s)\chi_{0}(s_{i})+s\chi_{1}(s_{i}).\label{eq:geodesic NA norm}
\end{equation}
By the definition (\ref{eq: NA dp}) for $d_{p}$, one can check that
$(\chi_{s})_{s\in[0,1]}$ is a $d_{p}$-geodesic between $\chi_{0}$
and $\chi_{1}$ for all $p\in[1,\infty]$. Furthermore, when $p\in(1,\infty)$,
one can show the uniqueness of geodesic. 

By the isometry (\ref{eq: Isometry}), we can determine the geodesics
in $\left(\hat{\mathcal{H}}(V),\hat{d}\right)$. For any $\eta_{0},\eta_{1}\in\hat{\mathcal{H}}(V)$,
we take a basis $(s_{i})$ diagonalizing both $\chi_{\eta_{0}}$ and
$\chi_{\eta_{1}}$. Let the Hermitian norm $H_{0}$ taking $(s_{i})$
as an orthonormal basis. Define operators $A_{0},A_{1}$ such that
$A_{0}s_{i}=\chi_{\eta_{0}}(s_{i})s_{i}$ and $A_{1}s_{i}=\chi_{\eta_{1}}(s_{i})s_{i}$
for all $i$. Then the asymptotic class $\eta_{s}$ of the geodesic
ray 
\begin{equation}
H_{t}^{s}\coloneqq H_{0}(e^{-tA_{s}}\cdot,\cdot),\ A_{s}=(1-s)A_{0}+sA_{1},\ s\in[0,1],\ t\geq0.\label{eq:geodesic H hat}
\end{equation}
gives the geodesic between $\eta_{0}$ and $\eta_{1}$. Since the
limit NA norm of $t\mapsto H_{t}^{s}$ gives the geodesic between
$\chi_{\eta_{0}}$ and $\chi_{\eta_{1}}$. Moreover, for a fixed $t\geq0$,
$s\mapsto H_{t}^{s}$ is also a geodesic and $d(H_{t}^{0},H_{t}^{1})$
is linear in $t$. However, if we take an arbitrary $H_{0}$, these
nice properties may not hold. 

\section{Finite-energy metrics and geodesic rays}

In order to discuss the maximal K-destabilizer introduced in \cite{xia_sharp_2021},
we review the space of finite-energy metrics and the geodesic rays
in it, refer to \cites{darvas_geometric_2019}{darvas_geodesic_2020}{berman_variational_2021}. 

\subsection{Finite-energy metrics and energy functionals}

Let $(X,L)$ be a polarized manifold. We use the additive notation
$\phi,\psi$ for metrics on $L$. The curvature form is denoted by
$dd^{c}\phi$, where $dd^{c}=\frac{\sqrt{-1}}{2\pi}\partial\bar{\partial}$.
Let $\mathcal{H}(L)$ be the space of smooth metrics on $L$ with
positive curvature. Fix a reference metric $\phi_{r}\in\mathcal{H}(L)$
Let $\mathrm{PSH}(L)$ be the space of psh metrics on $L$. The Monge--Ampère
measure associated to $\phi\in\mathrm{PSH}(L)$ is $\mathrm{MA}(\phi)\coloneqq(L^{n})^{-1}(dd^{c}\phi)^{n}$,
which is the non-pluriproduct of closed positive $(1,1)$-current
$dd^{c}\phi$. The total mass (MA mass) of $\mathrm{MA}(\phi)$ is
$\leq1$. Let $\mathcal{E}_{\mathrm{full}}(L)$ be the space of psh
metrics with full MA mass. It includes all the bounded psh metrics. 

For each $p\in[1,\infty)$, there is a Finsler metric on $\mathcal{H}(L)$
which generalizes Mabuchi's $L^{2}$-Riemannian metric when $p=2$.
It induces a length metric $d_{p}$ on $\mathcal{H}(L)$. Darvas \cite{darvas_mabuchi_2015,darvas_geometric_2019}
showed that the metric completion of $(\mathcal{H}(L),d_{p})$ coincides
with the space of $L^{p}$-finite-energy metrics 
\[
\mathcal{E}^{p}(L)\coloneqq\left\{ \phi\in\mathcal{E}_{\mathrm{full}}(L)\mid\int_{X}\left|\phi-\phi_{r}\right|^{p}\mathrm{MA}(\phi)<\infty\right\} ,\ \forall\phi_{r}\in\mathcal{H}(L),
\]
which is introduced by Cegrell--Guedj--Zeriahi. For each $p\in[1,\infty)$,
$(\mathcal{E}^{p}(L),d_{p})$ is a completely geodesic metric space,
and a CAT(0) space when $p=2$, see \cite{darvas_mabuchi_2017}. 

The Monge--Ampère energy $\mathcal{E}:\mathcal{H}(L)\times\mathcal{H}(L)\rightarrow\mathbb{R}$
is defined by 
\[
\mathcal{E}(\phi_{1},\phi_{0})=\int_{0}^{1}\mathrm{d}t\int_{X}\dot{\phi}_{t}\ \mathrm{MA}(\phi_{t}),
\]
where $(\phi_{t})_{t\in[0,1]}\subset\mathcal{H}(L)$ is any smooth
path. It satisfies the cocycle property 
\[
\mathcal{E}(\phi_{2},\phi_{0})=\mathcal{E}(\phi_{2},\phi_{1})+\mathcal{E}(\phi_{1},\phi_{0}).
\]
Taking the affine path $\phi_{t}=(1-t)\phi_{0}+t\phi_{1}$, we obtain
\begin{equation}
\mathcal{E}(\phi,\psi)\coloneqq\frac{1}{(n+1)L^{n}}\sum_{j=0}^{n}\int_{X}(\phi-\psi)\left(dd^{c}\phi\right)^{j}\wedge\left(dd^{c}\psi\right)^{n-j}.\label{eq: MA energy}
\end{equation}
When a reference metric $\phi_{r}\in\mathcal{H}(L)$ is fixed, we
write $\mathcal{E}(\phi)\coloneqq\mathcal{E}(\phi,\phi_{r})$. 

MA-energy can be extended to an u.s.c. functional on $\mathrm{PSH}(L)$
by defining 
\[
\mathcal{E}(\phi)=\inf\left\{ \mathcal{E}(\psi)\mid\mathcal{H}(L)\ni\psi\geq\phi\right\} \in\mathbb{R}\cup\{-\infty\}.
\]
We have 
\[
\mathcal{E}^{1}(L)=\left\{ \phi\in\mathrm{PSH}(L)\mid\mathcal{E}(\phi)>-\infty\right\} .
\]

Mabuchi's K-energy $\mathcal{M}:\mathcal{H}(L)\rightarrow\bbr$ is
defined by 
\begin{equation}
\mathcal{M}(\phi)\coloneqq\mathcal{M}(\phi,\phi_{r})\coloneqq\int_{0}^{1}\mathrm{d}t\int_{X}\dot{\phi}_{t}\left(\bar{S}-S(\phi_{t})\right)\mathrm{MA}(\phi_{t}),\ \bar{S}=\frac{nc_{1}(X)\cdot L^{n-1}}{L^{n}},\label{eq: K-energy}
\end{equation}
where $(\phi_{t})_{t\in[0,1]}\subset\mathcal{H}(L)$ is any smooth
path from $\phi_{r}$ to $\phi$, and $S(\phi_{t})$ is the scalar
curvature of metric $dd^{c}\phi_{t}$. Its critical points are cscK
metrics and it satisfies the cocycle property. %
{} In \cite{berman_convexity_2017}, K-energy is extended to $\mathcal{M}:\mathcal{E}^{1}(L)\rightarrow(-\infty,\infty]$.
They showed that the extension is $d_{1}$-lsc and convex along the
finite-energy geodesics. 

\subsection{Finite-energy geodesic rays }

Let $I\subset\mathbb{R}$ be any interval. In the following, a path
$(\phi_{t})_{t\in I}\subset\mathrm{PSH}(L)$ is a map $I\ni t\mapsto\phi_{t}\in\mathrm{PSH}(L)$.
Recall the annular set $D_{I}$ (\ref{eq:domain DI}). When $I$ is
open, a path $(\phi_{t})_{t\in I}\subset\mathrm{PSH}(L)$ induces
a $S^{1}$-invariant metric $\Phi$ on the line bundle $L\times D_{I}$
over $X\times D_{I}$ by letting 
\[
\Phi|_{L\times\tau}=\phi_{-\log\left|\tau\right|^{2}}.
\]
We say that $(\phi_{t})_{t\in I}$ is a \emph{psh path} if $\Phi$
is a psh metric. Take a reference metric $\phi_{r}\in\mathcal{H}(L)$
and $\omega_{r}\coloneqq dd^{c}\phi_{r}$. Let $\phi_{t}=\phi_{r}+u_{t}$,
then $(\phi_{t})_{t\in I}$ is a psh path iff $U(x,\tau)\coloneqq u_{-\log\left|\tau\right|^{2}}(x)$
is a $\pi_{1}^{*}\omega_{r}$-psh function on $X\times D_{I}$. This
implies that $u_{t}(x)$ is convex in $t$, so is $t\mapsto\sup_{X}u_{t}$.
A \emph{psh ray} is a psh path $(\phi_{t})_{t>0}$. We say a psh ray
is \emph{sublinear} if $\sup_{X}u_{t}=\mathrm{O}(t)$. 

Given $p\in[1,\infty)$ and $\phi_{0},\phi_{1}\in\mathcal{E}^{p}(L)$,
consider the usc regularization of envelop 
\[
\phi_{t}\coloneqq\mathrm{usc}\sup\left\{ \psi_{t}\mid(\psi_{t})_{t\in(0,1)}\ \textrm{psh path},\lim_{t\rightarrow0^{+}}\psi_{t}\leq\phi_{0},\ \lim_{t\rightarrow1^{-}}\psi_{t}\leq\phi_{1}\right\} .
\]
Then $(\phi_{t})_{t\in[0,1]}$ gives a geodesic segment in metric
space $\left(\mathcal{E}^{p}(L),d_{p}\right)$, which is called the
\emph{finite-energy geodesic }connecting $\phi_{0}$ and $\phi_{1}$.
When $p>1$, this is the unique geodesic between $\phi_{0}$ and $\phi_{1}$. 

A \emph{$L^{p}$-geodesic ray} is $\ell=(\ell_{t})_{t\geq0}\subset\mathcal{E}^{p}(L)$
such that for any $T>0$, $(\ell_{t})_{t\in[0,T]}$ is the finite-energy
geodesic between $\ell_{0}$ and $\ell_{T}$. Fix the starting-point
$\phi\in\mathcal{E}^{p}(L)$, Darvas--Lu \cite{darvas_geodesic_2020}
considered the space 
\[
\mathcal{R}_{\phi}^{p}(L)\coloneqq\left\{ \ell=(\ell_{t})_{t\geq0}\mid\ell\ \textrm{is}\ L^{p}\textrm{-geodesic ray, }\ell_{0}=\phi\right\} ,
\]
and endow it with the chordal metric
\begin{equation}
d_{p}^{c}\left(\ell,\ell^{\prime}\right)\coloneqq\lim_{t\rightarrow+\infty}\frac{1}{t}d_{p}(\ell_{t},\ell_{t}^{\prime}).\label{eq:chord metric}
\end{equation}
They showed that $\mathcal{R}_{\phi}^{p}(L)$ is a completely geodesic
metric space. When $p=2$, it is a CAT(0) space, see \cite[Theorem 4.7]{xia_sharp_2021}.
If we take another starting-point $\phi'$, by \cite[Theorem 1.3]{darvas_geodesic_2020},
there is a parallelism isometry 
\[
P_{\phi,\phi'}:\mathcal{R}_{\phi}^{p}(L)\rightarrow\mathcal{R}_{\phi'}^{p}(L)
\]
which sends $\ell$ to the unique geodesic ray emanating from $\phi'$
and asymptotic to $\ell$. 

When we consider the action of automorphisms on geodesic rays, it
is more natural to allow the starting-point to be free. Let $\mathcal{R}^{p}(L)$
be the space of all $L^{p}$-geodesic rays with arbitrary starting-point,
and then consider 
\begin{eqnarray*}
\mathcal{R}^{p}(L)/_{\sim} & \coloneqq & \left\{ \ell=(\ell_{t})_{t\geq0}\mid\ell\ \textrm{is }L^{p}\textrm{-geodesic ray}\right\} /\sim_{d_{p}},
\end{eqnarray*}
where $\sim_{d_{p}}$ is the $d_{p}$-asymptotic relation: $\ell\sim_{d_{p}}\ell^{\prime}$
$\Leftrightarrow$ $d_{p}(\ell_{t},\ell_{t}^{\prime})$ is bounded.
Due to the following fact, we can replace $\sim_{d_{p}}$ by $\sim_{d_{1}}$. 
\begin{prop}
For any two $L^{q}$-geodesic rays $\ell^{1}$ and $\ell^{2}$, if
$\ell^{1}\sim_{d_{p}}\ell^{2}$ ($q>p\geq1$), then $\ell^{1}\sim_{d_{q}}\ell^{2}$. 
\end{prop}

\begin{proof}
Apply the parallelism operator, there is a $L^{q}$-geodesic ray $\eta$
such that $\eta_{0}=\ell_{0}^{1}$ and $\eta\sim_{d_{q}}\ell^{2}$.
Then we have $\eta\sim_{d_{p}}\ell^{2}$, thus $\eta\sim_{d_{p}}\ell^{1}$
by transitivity, i.e. $h(t)\coloneqq d_{p}(\eta_{t},\ell_{t}^{1})$
is bounded. By the Busemann convexity \cite[Theorem 1.5]{chen_constant_2018},
$h(t)$ is convex. This implies that $h(t)\leq h(0)=0$. So $\eta=\ell^{1}$
and $\ell^{1}\sim_{d_{q}}\ell^{2}$. 
\end{proof}
By the triangle inequality, the limit (\ref{eq:chord metric}) only
depends on the asymptotic classes, so it defines a metric on $\mathcal{R}^{p}(L)/_{\sim}$.
With this metric, $\mathcal{R}^{p}(L)/_{\sim}$ is isometric to any
space $\mathcal{R}_{\phi}^{p}(L)$ via the parallelism operator. For
any $q>p\geq1$, we have inclusion $\mathcal{R}^{q}(L)/_{\sim}\hookrightarrow\mathcal{R}^{p}(L)/_{\sim}$. 

\section{Chow-stability and its maximal destabilizers}

In this section, we study the existence, uniqueness and the symmetries
of the maximal destabilizer for Chow-stability. 

\subsection{\label{subsec:Chow-stability}Chow-stability and the Kempf--Ness
functional }

Let $V$ be a $\bbc$-vector space of dimension $N$, and $X\subset\mathbb{P}V^{*}$
be a $n$-dimensional projective variety with degree $d$. Denote
by $\mathrm{Gr}\coloneqq\mathrm{Gr}(N-n-2,\mathbb{P}V^{*})$ the Grassmanian
manifold parameterizing all $(N-n-2)$-dimensional projective subspaces
in $\mathbb{P}V^{*}$. Let $p:\mathrm{Gr}\hookrightarrow\mathbb{P}\left(\wedge^{n+1}V\right)$
be the Plücker embedding, and $\mathcal{O}(1)$ the restricted hyperplane
line bundle on $\mathrm{Gr}$. The $(N-n-2)$-dimensional projective
subspaces that intersect with $X$ forms an irreducible hypersurface
$\mathcal{Z}_{X}\subset\mathrm{Gr}$ with degree $d$. Then $\mathcal{Z}_{X}$
is the zero locus of a unique (up to scaling) section $f_{X}\in\mathrm{H}^{0}(\mathrm{Gr},\mathcal{O}(d))$,
which is called the \emph{Chow-form} of $X$. The special linear group
$\mathrm{SL}(V)$ acts on $\mathrm{H}^{0}(\mathrm{Gr},\mathcal{O}(d))$
in a natural way. 
\begin{defn}[Chow-stability]
 \label{def: Chow-stab}In Mumford's GIT, we call $X\subset\mathbb{P}V^{*}$
is \emph{Chow-semistable} if the zero vector is not in the Zariski-closure
of the orbit $\mathrm{SL}(V)\cdot f_{X}$. Otherwise, $X$ is called
\emph{Chow-unstable}. We call $X$ is \emph{Chow-polystable} if the
orbit $\mathrm{SL}(V)\cdot f_{X}$ is Zariski-closed, and $X$ is
called \emph{Chow-stable} if, further, $f_{X}$ has a finite stabilizer. 
\end{defn}

The Hilbert--Mumford criterion allows us to only check the orbits
under all 1-PS $\lambda:\mathbb{C}^{*}\hookrightarrow\mathrm{SL}(V)$.
Suppose that in $\mathbb{P}\mathrm{H}^{0}(\mathrm{Gr},\mathcal{O}(d))$,
the point $[\lambda(t)\cdot f_{X}]$ converges to $[f_{X}^{\lambda}]$
as $t\rightarrow0$, then $\lambda(t)\cdot f_{X}^{\lambda}=t^{-w(X,\lambda)}f_{X}^{\lambda}$
for some integer $w(X,\lambda)$, which is called the \emph{Chow-weight}.
Hilbert--Mumford criterion says that $X$ is Chow-semistable iff
$w(X,\lambda)\geq0$ for all 1-PS $\lambda$. 

In GIT, it is well-known that $w(X,\lambda)$ is also equal to the
limit slope of the associated Kempf--Ness functional along a geodesic
ray. To define Kempf--Ness functional, we need to choose a norm for
the Chow-forms. In \cite{phong_stability_2003}, there is a canonical
choice such that the resulting Kempf--Ness functional can be expressed
by the Monge--Ampère energy on $X$. We review this below. 

First we take a reference Hermitian norm $H_{r}$ on $V$, it induces
a norm on $\wedge^{n+1}V$. Via the Plücker embedding $p:\mathrm{Gr}\hookrightarrow\mathbb{P}\left(\wedge^{n+1}V\right)$,
the line bundle $\mathcal{O}(d)$ is equipped a metric $\left\Vert \right\Vert $.
Let $\Omega$ be the restriction on $\mathrm{Gr}$ of the Fubini--Study
metric. In \cite{phong_stability_2003} §4, they define a ``norm''
$\left\Vert \right\Vert _{\mathrm{Chow}}$ on $\mathrm{H}^{0}(\mathrm{Gr},\mathcal{O}(d))$
by 
\begin{equation}
\log\left\Vert f\right\Vert _{\mathrm{Chow}}^{2}\coloneqq\fint_{\mathrm{Gr}}\log\left\Vert f\right\Vert ^{2}\ \Omega^{m},\ f\in\mathrm{H}^{0}(\mathrm{Gr},\mathcal{O}(d))\backslash\{0\},\label{eq:norm chow}
\end{equation}
where $m=\dim\mathrm{Gr}$. Note that the integrand is a quasi-psh
function, so it is integrable. It satisfies $\left\Vert \lambda f\right\Vert _{\mathrm{Chow}}=\left|\lambda\right|\left\Vert f\right\Vert _{\mathrm{Chow}}$
for $\forall\lambda\in\mathbb{C}$ and $\left\Vert f\right\Vert _{\mathrm{Chow}}=0$
iff $f=0$. Moreover, it is invariant under the action of $\mathrm{SU}(V,H_{r})$
(special unitary group w.r.t. $H_{r}$) on $\mathrm{H}^{0}(\mathrm{Gr},\mathcal{O}(d))$. 

Although the triangle inequality is absent, $\left\Vert \right\Vert _{\mathrm{Chow}}$
induces a Hermitian metric on the tautological bundle $\mathcal{O}(-1)$
over $\mathbb{P}\mathrm{H}^{0}(\mathrm{Gr},\mathcal{O}(d))$. The
associated Kempf--Ness functional is defined by 
\begin{equation}
\mathbf{F}:\mathrm{SU}(V,H_{r})\backslash\mathrm{SL}(V)\rightarrow\mathbb{R},\ \mathbf{F}\left([g]\right)=\log\frac{\left\Vert g\cdot f_{X}\right\Vert _{\mathrm{Chow}}^{2}}{\left\Vert f_{X}\right\Vert _{\mathrm{Chow}}^{2}}.\label{eq: KN for Chow, origin}
\end{equation}
This functional can be expressed by the MA-energy on $X$. This was
originally shown in \cite{zhang_heights_1996} using Deligne's pairings,
and then by Phong--Sturm \cite[Theorem 5]{phong_stability_2003}
and Paul \cite{paul_geometric_2004} by other methods. 

For $H\in\mathcal{H}(V)$, denote by $\mathbf{FS}(H)$ the restricted
Fubini--Study metric on $\mathcal{O}(1)|_{X}$ via $X\hookrightarrow\mathbb{P}V^{*}$.
More explicitly, take a reference metric $\phi_{r}$ on $\mathcal{O}(1)|_{X}$,
we have 
\begin{equation}
\mathbf{FS}(H)=\phi_{r}+\log\sup_{s\in V\backslash\{0\}}\frac{\left|s\right|_{\phi_{r}}^{2}}{H(s,s)}=\phi_{r}+\log\sum_{i=1}^{N}\left|s_{i}\right|_{\phi_{r}}^{2},\label{eq:def of FS}
\end{equation}
where $(s_{i})\subset V$ is any $H$-orthonormal basis, and we take
$s\in V$ as a section of $\mathcal{O}(1)|_{X}$. 
\begin{thm}[\cites{zhang_heights_1996}{phong_stability_2003}{paul_geometric_2004}]
\label{thm: PS} Let $X\subset\mathbb{P}V^{*}$ be a smooth projective
variety with dimension $n$ and degree $d$. Let $f_{X}\in\mathrm{H}^{0}(\mathrm{Gr},\mathcal{O}(d))$
be the associated Chow form, and $\mathbf{F}$ (\ref{eq: KN for Chow, origin})
be the associated Kempf--Ness functional. Then we have 
\[
\mathbf{F}\left([g]\right)=\mathcal{E}_{X}\left(\mathbf{FS}(H_{r}\circ g),\mathbf{FS}(H_{r})\right),\ \forall[g]\in\mathrm{SU}(V,H_{r})\backslash\mathrm{SL}(V),
\]
where $\mathcal{E}_{X}$ is the Monge--Ampère energy for $(X,\mathcal{O}(1)|_{X})$. 
\end{thm}

To reduce the dependence on $H_{r}$, we change and enlarge the domain
of $\mathbf{F}$. Let 
\[
\mathcal{H}(V,H_{r})\coloneqq\left\{ H\in\mathcal{H}(V)\mid\det H=\det H_{r}\right\} ,
\]
where $\det H$ denotes the induced metric on $\wedge^{\mathrm{top}}V$.
It is a totally geodesic hypersurface in $\mathcal{H}(V)$, see \cite[Lemma 10.52]{bridson_metric_2013}.
Group $\mathrm{SL}(V)$ transitively acts on $\mathcal{H}(V,H_{r})$
(from the right side) by $H\cdot g\coloneqq H\circ g$, and the isotropy
subgroup of $H_{r}$ is $\mathrm{SU}(V,H_{r})$. Thus $\mathrm{SU}(V,H_{r})\backslash\mathrm{SL}(V)$
can be identified with $\mathcal{H}(V,H_{r})$ by sending $[g]\mapsto H_{r}\circ g$.
We identify $\mathbf{F}$ with a function $\tilde{\mathbf{F}}$ on
$\mathcal{H}(V,H_{r})$, that is 
\[
\tilde{\mathbf{F}}(H)=\mathcal{E}_{X}\left(\mathbf{FS}(H),\mathbf{FS}(H_{r})\right),\ H\in\mathcal{H}(V,H_{r}).
\]
Next we extend $\tilde{\mathbf{F}}$ to $\mathcal{H}(V)$. Define
a contraction map $\rho:\mathcal{H}(V)\rightarrow\mathcal{H}(V,H_{r})$
by $\rho(H)=e^{\mathbf{E}_{V}(H,H_{r})}H$, recall $\mathbf{E}_{V}$
(\ref{eq: Archi E_V}). Then the pullback of $\tilde{\mathbf{F}}$
along $\rho$ is 
\begin{equation}
\mathbf{M}_{X}(H)\coloneqq\tilde{\mathbf{F}}\left(\rho(H)\right)=\mathcal{E}_{X}\left(\mathbf{FS}(H),\mathbf{FS}(H_{r})\right)-\mathbf{E}_{V}(H,H_{r}),\ H\in\mathcal{H}(V).\label{eq:KN final}
\end{equation}
It is invariant under scaling $H\rightsquigarrow aH$. When we change
$H_{r}$, by the cocycle property, $\mathbf{M}_{X}$ is just changed
by adding a constant. 

The Chow-weights are obtained by taking the limit slopes of $\mathbf{F}$
along the geodesic rays in $\mathrm{SU}(V,H_{r})\backslash\mathrm{SL}(V)$,
which are corresponding to the geodesic rays in $\mathcal{H}(V,H_{r})$
of the form $H_{t}=H_{r}\circ e^{tA}$ ($\mathrm{tr}A=0$). To check
the Chow-semistability of $X$, it is equivalent to check the non-negativity
of limits slopes of $\mathbf{M}_{X}$. This is due to the facts: (1)
$\mathcal{H}(V,H_{r})$ is totally geodesic; (2) $\tilde{\mathbf{F}}$
is the restriction of $\mathbf{M}_{X}$; (3) for a geodesic $H_{t}$
in $\mathcal{H}(V)$, since $\mathbf{E}_{V}(H_{t},H_{r})$ is linear
in $t$, $\rho(H_{t})$ is also a geodesic in $\mathcal{H}(V,H_{r})$.

For these reasons, we take $\mathbf{M}_{X}$ as the Kempf--Ness functional
for Chow-stability. 
\begin{defn}[Kempf--Ness functional]
\label{def: KN functional} For the Chow-stability of a smooth projective
variety $X\subset\mathbb{P}V^{*}$, we define the (extended) Kempf--Ness
functional is $\mathbf{M}_{X}:\mathcal{H}(V)\rightarrow\mathbb{R}$.
This is exactly the functional $\tilde{Z}$ in \cite{donaldson_scalar_2005}.
Its critical points are balanced norms, see §\ref{subsec: Donaldson work}. 
\end{defn}

Before taking the limit slopes, we give some basic properties of $\mathbf{M}_{X}$. 
\begin{prop}
\label{prop: conv of Mk}Let $X\subset\mathbb{P}V^{*}$ be a smooth
projective variety, and $\dim V=N$. Denote $L=\mathcal{O}(1)|_{X}$. 

(1) If $(H_{t})_{t\in I}\subset\mathcal{H}(V)$ is a positively curved
smooth path, then $\left(\mathbf{FS}(H_{t})\right)_{t\in I}\subset\mathcal{H}(L)$
is a psh path. 

(2) Kempf--Ness functional $\mathbf{M}_{X}$ is convex along the
geodesics in $\mathcal{H}(V)$. 

(3) $\mathbf{M}_{X}$ is globally Lipschitz, 
\begin{equation}
\left|\mathbf{M}_{X}(H)-\mathbf{M}_{X}(G)\right|\leq\left(\sqrt{N}+1\right)d(H,G),\ \forall H,G\in\mathcal{H}(V).\label{eq:Lip of Mk}
\end{equation}
\end{prop}

\begin{proof}
(1) Path $(H_{t})_{t\in I}$ induces a metric $h$ on the trivial
bundle $E\coloneqq V\times D_{I}\rightarrow D_{I}$ with curvature
$\geq0$ (in the sense of Griffiths). Then the curvature of $E^{*}$
is $\leq0$. By \cite[Proposition 3.3 (iv)]{darvas_griffiths_2022},
the induced Fubini--Study metric on $\mathcal{O}(1)$ (over $\mathbb{P}E^{*}=\mathbb{P}V^{*}\times D_{I}$)
has semi-positive curvature (note that in \cite{darvas_griffiths_2022}
$L(E)=\mathcal{O}(-1)$). The pullback of $\mathcal{O}(1)$ along
$X\times D_{I}\hookrightarrow\mathbb{P}E^{*}$ is the line bundle
$L\times D_{I}$. The pullback metric on $L\times D_{I}$ is given
by the path $\left(\mathbf{FS}(H_{t})\right)$. Since the pullback
metric also has semi-positive curvature, by definition, $\left(\mathbf{FS}(H_{t})\right)$
is a psh path. 

(2) This is showed in \cite{donaldson_scalar_2005} by a direct computation.
Suppose $(H_{t})$ is a geodesic, by (1), $\left(\mathbf{FS}(H_{t})\right)$
is a psh path. It implies that $\mathcal{E}_{X}\left(\mathbf{FS}(H_{t}),\mathbf{FS}(H_{r})\right)$
is convex in $t$. Since $\mathbf{E}_{V}(H_{t},H_{r})$ is affine
in $t$, $\mathbf{M}_{X}(H_{t})$ is convex. 

(3) By the ``sup'' formula (\ref{eq:def of FS}) for $\mathbf{FS}$
and (\ref{eq: compare dp}), we have 
\[
\left|\mathbf{FS}(H)-\mathbf{FS}(G)\right|\leq d_{\infty}(H,G)\leq\sqrt{N}d(H,G),\ \forall H,G\in\mathcal{H}(V).
\]
By the cocycle property and the formula (\ref{eq: MA energy}) for
$\mathcal{E}_{X}$, we have
\[
\left|\mathcal{E}_{X}\left(\mathbf{FS}(H),\mathbf{FS}(H_{r})\right)-\mathcal{E}_{X}\left(\mathbf{FS}(G),\mathbf{FS}(H_{r})\right)\right|=\left|\mathcal{E}_{X}\left(\mathbf{FS}(H),\mathbf{FS}(G)\right)\right|\leq\sqrt{N}d(H,G).
\]
For $\mathbf{E}_{V}$, by (\ref{eq: Lip of EV}) we have 
\[
\left|\mathbf{E}_{V}(H,H_{r})-\mathbf{E}_{V}(G,H_{r})\right|=\left|\mathbf{E}_{V}(H,G)\right|\leq d(H,G).
\]
Combining them together, we obtain (\ref{eq:Lip of Mk}). 
\end{proof}
\begin{rem}
It seems that the convexity of $\mathbf{M}_{X}$ cannot be extended
to more general paths. If $(H_{t})$ is positively curved, by (1)
$\mathcal{E}_{X}\left(\mathbf{FS}(H_{t})\right)$ is convex in $t$.
But meanwhile, $\mathbf{E}_{V}(H_{t},H_{r})$ is also convex, $\mathbf{M}_{X}$
is the difference of two convex functions. On the other hand, if $(H_{t})$
is negatively curved, $\mathbf{FS}(H_{t})$ is not necessarily a psh
ray. 
\end{rem}

Let $(H_{t})_{t\geq0}\subset\mathcal{H}(V)$ be a geodesic ray, Proposition
\ref{prop: conv of Mk} implies that $t\mapsto\mathbf{M}_{X}(H_{t})$
is convex and Lipschitz over $[0,\infty)$. Thus 
\[
\frac{1}{t}\left(\mathbf{M}_{X}(H_{t})-\mathbf{M}_{X}(H_{0})\right)
\]
will non-decreasingly converge to a finite number. Moreover, (\ref{eq:Lip of Mk})
implies that the limit only depends on the asymptotic class of $(H_{t})$. 
\begin{defn}[Chow-weights]
 Let $X\subset\mathbb{P}V^{*}$ be a smooth projective variety. We
define the Chow-weight $M_{X}$ as a function on space $\hat{\mathcal{H}}(V)\cong\mathcal{N}(V)$
(\ref{eq: Isometry}). For an asymptotic class $\eta\in\hat{\mathcal{H}}(V)$
with NA limit $\chi\in\mathcal{N}(V)$, we define 
\begin{equation}
M_{X}(\eta)=M_{X}(\chi)\coloneqq\lim_{t\rightarrow\infty}\frac{1}{t}\mathbf{M}_{X}(H_{t})\in\mathbb{R},\ \forall(H_{t})\in\eta.\label{eq:def Chow weight}
\end{equation}
It can be expressed by non-Archimedean functionals, see Theorem (\ref{thm: NA express Chow}). 
\end{defn}

It is easy to see that for $\forall a>0$ and $c\in\bbr$, we have
\[
M_{X}(a\chi)=aM_{X}(\chi),\ M_{X}(\chi+c)=M_{X}(\chi).
\]
By the discussion above Definition \ref{def: KN functional}, $X$
is Chow-semistable if and only if $M_{X}\geq0$ on $\mathcal{N}(V)$. 
\begin{prop}
\label{prop:convexity of Chow}(1) Chow-weight $M_{X}$ is convex
along the geodesics in $\left(\hat{\mathcal{H}}(V),\hat{d}\right)$
and $\left(\mathcal{N}(V),d_{2}\right)$.

(2) $M_{X}:\mathcal{N}(V)\rightarrow\mathbb{R}$ is Lipschitz with
respect to any metric $d_{p\in[1,\infty]}$. 
\end{prop}

\begin{proof}
(1) Consider the geodesic (\ref{eq:geodesic H hat}) between $\eta_{0},\eta_{1}\in\hat{\mathcal{H}}(V)$.
Since for any $t\geq0$, $[0,1]\ni s\mapsto H_{t}^{s}$ is a geodesic
in $\mathcal{H}(V)$. By the convexity of $\mathbf{M}_{X}$ (Proposition
\ref{prop: conv of Mk} (2)), we have 
\[
\mathbf{M}_{X}(H_{t}^{s})\leq(1-s)\mathbf{M}_{X}(H_{t}^{0})+s\mathbf{M}_{X}(H_{t}^{1}),\ \forall s\in[0,1].
\]
Divide both sides by $t$, then let $t\rightarrow\infty$, we obtain
\[
M_{X}(\eta_{s})\leq(1-s)M_{X}(\eta_{0})+sM_{X}(\eta_{1}).
\]
By the isometry (\ref{eq: Isometry}), $M_{X}$ is also convex along
the geodesics in $\left(\mathcal{N}(V),d_{2}\right)$.

(2) Apply the inequality (\ref{eq:Lip of Mk}) to any two geodesic
rays $(H_{t}),(G_{t})$ in $\mathcal{H}(V)$. Divide both sides by
$t$ and let $t\rightarrow\infty$. It implies that $M_{X}$ is Lipschitz
with respect to $d_{2}$, and all metrics are Lipschitz equivalent
to each other. 
\end{proof}

\subsection{\label{subsec:Maximal chow}Maximal destabilizers for Chow-unstable
varieties}

We consider the Chow-unstable case. In GIT, by the works \cite{kempf_instability_1978,rousseau_immeubles_1978},
one can associate to an unstable point a unique (modulo an equivalence
relation) maximal destabilizing 1-PS, which minimizes the normalized
Hilbert--Mumford weight. Since we will not use this general result,
we just give a sketch of it in the setting of Chow-stability, refer
to \cite{mumford_geometric_1994} §2.2 and appendix 2B for the general
setting. 

\subsubsection{\label{subsec: worst 1-PS} Kempf--Rousseau's maximal destabilizing
1-PS }

Suppose that $X\subset\mathbb{P}V^{*}$ is Chow-unstable, consider
the normalized Chow-weight 
\[
\bar{w}(X,\lambda)\coloneqq\frac{w(X,\lambda)}{\left\Vert \lambda\right\Vert },
\]
where $\left\Vert \lambda\right\Vert $ is a norm for 1-PS in $\mathrm{SL}(V)$.
In \cite{kempf_instability_1978,rousseau_immeubles_1978}, it is showed
that $\bar{w}(X,\lambda)$ admits minimizers which are called the
maximal destabilizing 1-PS, and these minimizers are equivalent to
each other in the following sense. 

First, for any 1-PS $\lambda\subset\mathrm{SL}(V)$, one can associate
to it a parabolic subgroup
\[
P(\lambda)\coloneqq\left\{ g\in\mathrm{SL}(V)\mid\lim_{t\rightarrow0}\lambda(t)g\lambda(t^{-1})\textrm{ exists in }\mathrm{SL}(V)\right\} .
\]
Two non-trivial 1-PS $\lambda_{1},\lambda_{2}\subset\mathrm{SL}(V)$
are said to be \emph{equivalent} to each other (denoted by $\lambda_{1}\sim\lambda_{2}$)
if there exists $n_{1},n_{2}\in\mathbb{N}$ and $g\in P(\lambda_{1})$
such that $\lambda_{2}(t^{n_{2}})=g\cdot\lambda_{1}(t^{n_{1}})\cdot g^{-1}$
for all $t\in\mathbb{C}^{*}$. All non-trivial 1-PS modulo this equivalence
relation form a structure called Tits building. Thus the maximal destabilizing
1-PS correspond to a unique point in this building. 

\subsubsection{Existence and uniqueness of maximal destabilizers}

Instead of using Kempf--Rousseau's general result, we consider the
following minimization problem for the $L^{p}$-normalized Chow-weight
($p\in[1,\infty]$) 
\begin{equation}
\inf\left\{ \bar{M}_{X,p}(\chi)\coloneqq\frac{M_{X}(\chi)}{\left\Vert \chi\right\Vert _{p}}\mid\chi\in\mathcal{N}(V)\backslash\{\chi_{tr}\}\right\} .\label{eq: min M_X p-norm}
\end{equation}
Since $M_{X}(\chi)$ and $\left\Vert \chi\right\Vert _{p}$ (\ref{eq:p norm})
are homogeneous, $\bar{M}_{X,p}$ is invariant under scaling. Obviously,
this problem is equivalent to 
\begin{equation}
\inf\left\{ M_{X}(\chi)\mid\chi\in\mathcal{N}(V),\ \left\Vert \chi\right\Vert _{p}=1\right\} .\label{eq:min M_X norm1}
\end{equation}
Although we know that $M_{X}$ is $d_{p}$-continuous, the existence
of minimizers is not immediate, since the sphere $\{\left\Vert \chi\right\Vert _{p}=1\}$
is not compact in $d_{p}$-topology. 
\begin{thm}
\label{thm:existence-uniqueness of Chow}Let $X\subset\mathbb{P}V^{*}$
be a smooth projective variety. 

(1) For any $p\in[1,\infty]$, the problem (\ref{eq: min M_X p-norm})
admits a minimizer. 

(2) When $X$ is Chow-unstable and $p\in(1,\infty)$, the minimizer
is unique up to rescaling. 
\end{thm}

\begin{proof}
(1) The main idea is to use another locally compact topology under
which the objective function is lsc. 

Fix a base point $H_{0}\in\mathcal{H}(V)$. Recall the tangent space
$T_{H_{0}}\mathcal{H}(V)$ is identified with the space of $H_{0}$-self-adjoint
operators. Consider the composed map 
\[
\Psi:T_{H_{0}}\mathcal{H}(V)\rightarrow\hat{\mathcal{H}}(V)\stackrel{\cong}{\rightarrow}\mathcal{N}(V),
\]
where the second map is the isometry (\ref{eq: Isometry}), and the
first map sends $A\in T_{H_{0}}\mathcal{H}(V)$ to the class of the
geodesic ray $H_{t}^{A}\coloneqq H_{0}(e^{tA}\cdot,\cdot)$, which
starts from $H_{0}$ with initial tangent vector $A$. Actually, $\Psi(A)$
is the eigenspace filtration (\ref{eq:eigen filtra}) of $-A$. By
Cartan--Hadamard theorem, the first map is bijective, so is $\Psi$. 

By Example \ref{exa: NAlimit of geodesic} and (\ref{eq: Finsler norm},
\ref{eq:p norm}), we have $\left\Vert \Psi(A)\right\Vert _{p}=\left\Vert A\right\Vert _{p,H_{0}}$
for $\forall A\in T_{H_{0}}\mathcal{H}(V)$. Since $\Psi$ is bijective,
minimization problem (\ref{eq:min M_X norm1}) is transformed into
\[
\inf\left\{ M_{X}\circ\Psi(A)\mid A\in T_{H_{0}}\mathcal{H}(V),\ \left\Vert A\right\Vert _{p,H_{0}}=1\right\} .
\]
Recall that $M_{X}$ is the limit slope of $\mathbf{M}_{X}$, we have
\[
M_{X}\circ\Psi(A)=\sup_{t\geq0}\frac{1}{t}\left[\mathbf{M}_{X}\left(H_{t}^{A}\right)-\mathbf{M}_{X}(H_{0})\right].
\]
Note that for a fixed $t\geq0$, function $A\mapsto\frac{1}{t}\left[\mathbf{M}_{X}\left(H_{t}^{A}\right)-\mathbf{M}_{X}(H_{0})\right]$
is continuous w.r.t. the Euclidean topology. It implies that $M_{X}\circ\Psi$
is lower semi-continuous on $T_{H_{0}}\mathcal{H}(V)$. Since a lsc
function can attain its infimum over a compact set $\{\left\Vert A\right\Vert _{p,H_{0}}=1\}$,
the existence follows. 

(2) If $\chi_{0}$ and $\chi_{1}$ are two minimizers for $\bar{M}_{X,p}$.
Rescaling them, we can assume $\left\Vert \chi_{0}\right\Vert _{p}=\left\Vert \chi_{1}\right\Vert _{p}=1$.
Then $M_{X}(\chi_{0})=M_{X}(\chi_{1})=\inf\bar{M}_{X,p}$. Let $(\chi_{s})_{s\in[0,1]}$
be the geodesic (\ref{eq:geodesic NA norm}) connecting them. By the
convexity of $M_{X}$ (Proposition \ref{prop:convexity of Chow}),
we have 
\begin{equation}
M_{X}(\chi_{s})\leq(1-s)M_{X}(\chi_{0})+sM_{X}(\chi_{1})=\inf\bar{M}_{X,p}<0,\ \forall s\in[0,1].\label{eq: proof convex of M_X}
\end{equation}
Let $(s_{i})$ be a basis diagonalizing both $\chi_{0}$ and $\chi_{1}$,
then we have 
\[
\left\Vert \chi_{s}\right\Vert _{p}^{p}=\frac{1}{N}\sum_{i=1}^{N}\left|(1-s)\chi_{0}(s_{i})+s\chi_{1}(s_{i})\right|^{p}.
\]
If $\chi_{0}\neq\chi_{1}$, by the uniform convexity of $L^{p}$-norm
when $p\in(1,\infty)$, we have $\left\Vert \chi_{s}\right\Vert _{p}<1$
for $s\in(0,1)$. It follows that 
\[
0>M_{X}(\chi_{s})>\frac{M_{X}(\chi_{s})}{\left\Vert \chi_{s}\right\Vert _{p}}=\bar{M}_{X,p}(\chi_{s}),\ \forall s\in(0,1).
\]
Combine with (\ref{eq: proof convex of M_X}), we have $\bar{M}_{X,p}(\chi_{s})<\inf\bar{M}_{X,p}$.
It is a contradiction, thus $\chi_{0}=\chi_{1}$. 
\end{proof}
\begin{rem}
In the Chow-semistable case (i.e. $M_{X}\geq0$), since $M_{X}(\chi_{tr}+c)=0$,
$\chi_{tr}+c$ ($c\neq0$) are trivial minimizers for (\ref{eq: min M_X p-norm}).
In this case, one should minimize $\frac{M_{X}(\chi)}{\left\Vert \chi-E_{V}(\chi)\right\Vert _{p}}$,
which is equivalent to the problem 
\[
\inf\left\{ M_{X}(\chi)\mid\chi\in\mathcal{N}(V),\ \left\Vert \chi\right\Vert _{p}=1,\ E_{V}(\chi)=0\right\} .
\]
By the same proof (in additional, require $\mathrm{tr}A=0$), we can
show that it admits a minimizer. 
\end{rem}

\begin{defn}
For a Chow-unstable variety $X\subset\mathbb{P}V^{*}$, we call the
minimizers of $\frac{M_{X}(\chi)}{\left\Vert \chi\right\Vert _{2}}$
the \emph{maximal Chow-destabilizers} for $X$, which differ from
each other by a scaling. 
\end{defn}

\subsection{\label{subsec:Symmetries-of-Chow}Symmetries of maximal Chow-destabilizers}

The uniqueness of maximal Chow-destabilizer makes it inherit all the
symmetries of $X\subset\mathbb{P}V^{*}$. Specifically, suppose that
we have a group $G$ and homomorphism $G\rightarrow\mathrm{GL}(V)$,
then $G$ acts on $\mathbb{P}V^{*}$. We suppose that $G$ preserves
$X$. For instance, for a polarized manifold $(X,L)$, we take $G=\mathrm{Aut}(X,L)$
and $V=R_{k}$ ($k\gg1$). The $G$-action on $R_{k}$ is defined
by $(g\cdot s)_{x}\coloneqq g\cdot s_{g^{-1}x}$ for $g\in G$, $s\in R_{k}$
and $x\in X$. Then the embedding $X\hookrightarrow\mathbb{P}R_{k}^{*}$
is $G$-equivariant. 
\begin{thm}
\label{thm:sym of Chow-dest}Let $X\subset\mathbb{P}V^{*}$ be a smooth
projective variety. Suppose there is a group $G$ and a morphism $G\rightarrow\mathrm{GL}(V)$
such that $X\subset\mathbb{P}V^{*}$ is preserved by $G$. Then we
have 

(1) the Chow-weight $M_{X}:\mathcal{N}(V)\rightarrow\mathbb{R}$ is
$G$-invariant w.r.t. the $G$-action on $\mathcal{N}(V)$; 

(2) when $X$ is Chow-unstable, each maximal Chow-destabilizer $\bar{\chi}\in\mathcal{N}(V)$
is fixed by $G$. 
\end{thm}

\begin{proof}
(1) Group $G$ acts $\mathcal{N}(V)$ by $g\cdot\chi\coloneqq\chi\circ g^{-1}$,
and on $\mathcal{H}(V)$ in the same way. Set $L=\mathcal{O}(1)|_{X}$,
$G$ acts on $\mathcal{H}(L)$ by pulling-back $g\cdot\phi\coloneqq\left(g^{-1}\right)^{*}\phi$.
Then the Fubini-Study map $\mathbf{FS}:\mathcal{H}(V)\rightarrow\mathcal{H}(L)$
is $G$-equivariant. 

First we consider the behavior of $\mathbf{M}_{X}$ (\ref{eq:KN final})
under $G$-action. For any $H\in\mathcal{H}(V)$ and $g\in G$, we
have 
\begin{eqnarray*}
\mathbf{M}_{X}(g\cdot H) & = & \mathcal{E}_{X}\left(\mathbf{FS}(g\cdot H),\mathbf{FS}(H_{r})\right)-\mathbf{E}_{V}(g\cdot H,H_{r})\\*
* & = & \mathcal{E}_{X}\left(g\cdot\mathbf{FS}(H),g\cdot\mathbf{FS}(H_{r})\right)+\mathcal{E}_{X}\left(g\cdot\mathbf{FS}(H_{r}),\mathbf{FS}(H_{r})\right)\\
 &  & -\mathbf{E}_{V}(g\cdot H,g\cdot H_{r})-\mathbf{E}_{V}(g\cdot H_{r},H_{r})\\*
** & = & \mathbf{M}_{X}(H)+\mathcal{E}_{X}\left(g\cdot\mathbf{FS}(H_{r}),\mathbf{FS}(H_{r})\right)-\mathbf{E}_{V}(g\cdot H_{r},H_{r}),
\end{eqnarray*}
where in the row $*$ we use the equivariance of $\mathbf{FS}$ and
the cocycle property for $\mathcal{E}_{X}$ and $\mathbf{E}_{V}$,
and in the row $**$ we use 
\begin{eqnarray*}
\mathcal{E}_{X}\left(g\cdot\mathbf{FS}(H),g\cdot\mathbf{FS}(H_{r})\right) & = & \mathcal{E}_{X}\left(\mathbf{FS}(H),\mathbf{FS}(H_{r})\right);\\
\mathbf{E}_{V}(g\cdot H,g\cdot H_{r}) & = & \mathbf{E}_{V}(H,H_{r}).
\end{eqnarray*}
For a geodesic ray $(H_{t})$ with limit NA norm $\chi$, $(g\cdot H_{t})$
is also a geodesic ray with limit norm $g\cdot\chi$. By the above
identity, 
\[
\mathbf{M}_{X}(g\cdot H_{t})=\mathbf{M}_{X}(H_{t})+C(g),
\]
where $C(g)$ is a constant independent of $t$. Divide both sides
by $t$ and let $t\rightarrow\infty$, we obtain $M_{X}(g\cdot\chi)=M_{X}(\chi)$. 

(2) It is easy to see $\left\Vert g\cdot\chi\right\Vert _{p}=\left\Vert \chi\right\Vert _{p}$
for all $p$. Hence $\bar{M}_{X,p}$ is $G$-invariant by (1). Suppose
$X$ is Chow-unstable and $p\in(1,\infty)$, by Theorem \ref{thm:existence-uniqueness of Chow},
$\bar{M}_{X,p}$ has a unique minimizer $\bar{\chi}$ (up to scaling).
By the invariance, for any $g\in G$, $g\cdot\bar{\chi}$ is also
a minimizer, so $g\cdot\bar{\chi}=a\bar{\chi}$ for some $a>0$. Take
the $p$-norm on both sides, we have $a=1$. Hence $g\cdot\bar{\chi}=\bar{\chi}$
for all $g\in G$. 
\end{proof}
Since we know the maximal Chow-destabilizer is $G$-invariant, we
can search for it only among the $G$-invariant norms. If further,
$M_{X}$ has a nice formula on the $G$-invariant norms, we can determine
the minimizer explicitly. The toric case is an example. 
\begin{example}[toric case]
 Let $M$ be a lattice of rank $n$, and $P\subset M_{\mathbb{R}}\coloneqq M\otimes\bbr$
be a lattice Delzant polytope. Then $P$ gives rise to a polarized
toric manifold $(X,L)$, which carries a linearized action by the
complex torus $T_{\bbc}\coloneqq\mathrm{Hom}_{\mathbb{Z}}(M,\mathbb{C}^{*})$.
For any $k\geq1$, the section space $V=R_{k}=\mathrm{H}^{0}(X,kL)$
is a multiplicity-free $T_{\bbc}$-module. For each $u\in P_{k}\coloneqq P\cap k^{-1}M$,
there is a section $s_{u,k}\in R_{k}$ such that $t\cdot s_{u,k}=t^{ku}s_{u,k}$
for $\forall t\in T_{\bbc}$. We have the weight-decomposition $R_{k}=\bigoplus_{u\in P_{k}}\mathbb{C}\cdot s_{u,k}$. 

Assume $k$ is sufficiently large such that $X\rightarrow\mathbb{P}R_{k}^{*}$
is an embedding (actually, $k\geq n-1$ is enough). By Theorem \ref{thm:sym of Chow-dest},
if the embedded variety is Chow-unstable, then the maximal Chow-destabilizer
$\chi_{k}^{*}\in\mathcal{N}(R_{k})$ is $T_{\bbc}$-invariant. The
$T_{\bbc}$-invariant NA norms on $R_{k}$ can be diagonalized by
the canonical basis $(s_{u,k})_{u\in P_{k}}$. Hence we can search
for $\chi_{k}^{*}$ among the norms which are diagonalized by $(s_{u,k})$.
In terms of 1-PS, this amounts to only consider the 1-PS in the torus
\[
H\coloneqq\left\{ g\in\mathrm{SL}(R_{k})\mid g\textrm{ is diagonalized by basis }(s_{u,k})_{u}\right\} \cong\left(\mathbb{C}^{*}\right)^{N_{k}-1}.
\]
In \cite{kapranov_chow_1992}, they found that the Chow-weights for
$H$-action can be given by the secondary polytope $\Sigma(P_{k})$
associated to $P_{k}$. Hence the maximal Chow-destabilizer can be
determined by $\Sigma(P_{k})$. We will discuss these in \cite{yao_quantizing_2025}. 
\end{example}

The following result is obtained in \cite[Theorem 2.2, Corollary 2.3, 2.7.]{lee_asymptotic_2019}
by a different argument.
\begin{cor}
\label{coro : toric H imply SL} Let $(X,L)$ be the toric manifold
given in the above example. Suppose $X\rightarrow\mathbb{P}R_{k}^{*}$
is an embedding. If $X$ is Chow-semistable w.r.t. the torus $H\subset\mathrm{SL}(R_{k})$,
then $X$ is Chow-semistable w.r.t. the whole group $\mathrm{SL}(R_{k})$.
\end{cor}

\begin{proof}
If $X$ is Chow-unstable w.r.t. $\mathrm{SL}(R_{k})$, consider the
maximal destabilizer, we obtain a $T_{\bbc}$-invariant norm $\chi$
such that the Chow weight $M_{X}(\chi)<0$. By approximation, we can
find a 1-PS of $H$ with negative Chow-weight. This contradicts the
assumption. 
\end{proof}

\subsection{Asymptotic Chow-stability and K-stability}

For a polarized variety $(X,L)$, we can consider Chow-stability of
the embedding images $X\hookrightarrow\mathbb{P}R_{k}^{*}$ when $k\gg1$. 
\begin{defn}
Given a polarized variety $(X,L)$. 

(1) We say $(X,kL)$ is Chow-semistable ((poly)stable or unstable)
if $kL$ is very ample and the image of $X\hookrightarrow\mathbb{P}R_{k}^{*}$
is Chow-semistable ((poly)stable or unstable) in the sense of Definition
\ref{def: Chow-stab}. 

(2) We say $(X,L)$ is \emph{asymptotic Chow-semistable} if $(X,kL)$
is Chow-semistable when $k$ is large enough. 
\end{defn}

Let us recall the notion of K-semistability which is introduced in
\cite{tian_kahler--einstein_1997,donaldson_scalar_2002}, refer to
\cite{boucksom_uniform_2017} for details. 
\begin{defn}
\label{def: K-stability}(1) A \emph{test-configuration} $(\mathcal{X},\mathcal{L})$
for $(X,L)$ consists of a flat projective morphism $\pi:\mathcal{X}\rightarrow\mathbb{C}$;
a $\mathbb{Q}$-line bundle $\mathcal{L}$ on $\mathcal{X}$; a $\mathbb{C}^{*}$-action
on $(\mathcal{X},\mathcal{L})$ such that $\pi$ is equivariant; and
a $\mathbb{C}^{*}$-equivariant isomorphism $(\mathcal{X},\mathcal{L})|_{\mathbb{C}^{*}}\cong(X,L)\times\mathbb{C}^{*}$.
We say $(\mathcal{X},\mathcal{L})$ is normal if $\mathcal{X}$ is
normal, and $(\mathcal{X},\mathcal{L})$ is ample if $\mathcal{L}$
is $\pi$-ample. 

(2) Suppose that $(\mathcal{X},\mathcal{L})$ is an ample test-configuration
and $r\mathcal{L}$ ($r\geq1$) is a line bundle. For any $k\geq1$,
let $w_{rk}$ be the weight sum of $\bbc^{*}$-module $\mathrm{H}^{0}(\mathcal{X}_{0},\mathcal{L}_{0}^{rk})$,
here $(\mathcal{X}_{0},\mathcal{L}_{0})$ is the central fiber of
$(\mathcal{X},\mathcal{L})$. The \emph{Donaldson-Futaki invariant}
$\mathrm{DF}(\mathcal{X},\mathcal{L})$ is defined by expansion 
\begin{equation}
\frac{w_{rk}}{rkN_{rk}}=\mathbb{E}(\mathcal{X},\mathcal{L})-\mathrm{DF}(\mathcal{X},\mathcal{L})\frac{1}{2rk}+\mathrm{O}(k^{-2}),\label{eq:expand weight}
\end{equation}
and the leading term $\mathbb{E}(\mathcal{X},\mathcal{L})$ is called
the NA Monge--Ampère energy. 

(3) We call $(X,L)$ is \emph{K-semistable} if $\mathrm{DF}(\mathcal{X},\mathcal{L})\geq0$
for all normal ample test-configuration $(\mathcal{X},\mathcal{L})$
for $(X,L)$. Otherwise, $(X,L)$ is called K-unstable. 
\end{defn}

There are many works studying the relation between Chow and K-stability,
for instance \cite[etc]{tian_k-energy_1994,donaldson_scalar_2001,phong_stability_2003,donaldson_scalar_2005,mabuchi_energy-theoretic_2005,ross_study_2006,paul_hyperdiscriminant_2012}.
We only recall the following relation. 
\begin{thm}[\cite{ross_study_2006}]
 The asymptotic Chow-semistability of $(X,L)$ implies the K-semistability. 
\end{thm}

Its contrapositive proposition is that if $(X,L)$ is K-unstable,
then there is a sequence $k_{j}\ra\infty$ such that $(X,k_{j}L)$
is Chow-unstable. This can be strengthened as follows, see §\ref{subsec: main expansion NA}
for its proof. 
\begin{thm}
\label{thm: K-uns imply Chow-uns}If $(\mathcal{X},\mathcal{L})$
is an ample test-configuration for $(X,L)$ such that $\mathrm{DF}(\mathcal{X},\mathcal{L})<0$
and $r\mathcal{L}$ is line bundle, then $(X,krL)$ is Chow-unstable
when $k$ is sufficiently large.
\end{thm}

This motivates us to study K-unstable varieties from the angle of
Chow-unstability. For K-unstable manifolds, Xia \cite{xia_sharp_2021}
showed the existence and uniqueness of the maximal K-destabilizing
ray, which corresponds to a non-Archimedean metric on $L$ by \cites{berman_variational_2021}{li_geodesic_2022}.
We call this NA metric the maximal K-destabilizer (Definition \ref{def: max K-desta}).
It is natural to ask 
\begin{quote}
Does the sequence of maximal Chow-destabilizers ``converge'' to
the maximal K-destabilizers in some sense?
\end{quote}
This can be seen as the NA counterpart of Donaldson's theory \cite{donaldson_scalar_2001}
for balanced embeddings and cscK metrics. Before we dive into the
NA world, we briefly review \cite{donaldson_scalar_2005}, which emphasizes
the variational point of view and will hint us how to study the above
problem. 

\subsection{\label{subsec: Donaldson work}Balanced norms quantize cscK metrics}

In Kähler geometry, quantization method is to approximate the geometry
of $\mathcal{H}(L)$ by the geometry of $\mathcal{H}(R_{k})$, also
includes the functionals defined over them. There is a pair of maps
connecting them. 

For $k\geq1$ such that $kL$ is very ample, the Fubini--Study map
$\mathbf{FS}_{k}:\mathcal{H}(R_{k})\rightarrow\mathcal{H}(L)$ is
defined by 
\[
\mathbf{FS}_{k}(H)=\phi_{r}+\frac{1}{k}\log\frac{1}{N_{k}}\sum_{i=1}^{N_{k}}\left|s_{i}\right|_{k\phi_{r}}^{2},
\]
where $\phi_{r}\in\mathcal{H}(L)$ is any reference metric and $(s_{i})$
is any $H$-orthonormal basis of $R_{k}$. %

The Hilbert map $\mathbf{H}_{k}:\mathcal{H}(L)\rightarrow\mathcal{H}(R_{k})$
is defined by 
\[
\mathbf{H}_{k}(\phi)(s,s')=\int_{X}\left\langle s,s'\right\rangle _{k\phi}\mathrm{MA}(\phi),\ \forall s,s'\in R_{k}.
\]
Compose two maps together, for any $\phi\in\mathcal{H}(L),$ we have
\[
\mathbf{FS}_{k}\circ\mathbf{H}_{k}(\phi)=\phi+\frac{1}{k}\log\frac{\rho_{k}(\phi)}{N_{k}},\ \rho_{k}(\phi)\coloneqq\sum_{i}\left|s_{i}\right|_{k\phi}^{2},
\]
where $(s_{i})$ is any $\mathbf{H}_{k}(\phi)$-orthonormal basis,
and $\rho_{k}(\phi)$ is called the Bergman kernel. Tian--Bouche--Catlin--Zelditch's
expansion gives 
\begin{equation}
\rho_{k}(\phi)=\frac{L^{n}}{n!}k^{n}+\frac{L^{n}}{n!}\frac{S(\phi)}{2}k^{n-1}+\mathrm{O}(k^{n-2}).\label{eq:Bergman expan}
\end{equation}
It implies that $\mathbf{FS}_{k}\circ\mathbf{H}_{k}(\phi)\rightarrow\phi$
as $k\rightarrow\infty$, so we can think that norms $\mathbf{H}_{k}(\phi)$
quantize $\phi$.

The MA-energy $\mathcal{E}$ is quantized by $\mathbf{E}_{R_{k}}$
(\ref{eq: Archi E_V}). Fix a reference metric $\phi_{r}\in\mathcal{H}(L)$,
we denote $\mathbf{E}_{k}(H)\coloneqq k^{-1}\mathbf{E}_{R_{k}}(H,\mathbf{H}_{k}(\phi_{r}))$.
By (\ref{eq:Bergman expan}), we have 
\begin{equation}
\mathbf{E}_{k}\circ\mathbf{H}_{k}(\phi)=\mathcal{E}(\phi)-\frac{1}{2k}\mathcal{M}(\phi)+\mathrm{O}(k^{-2}),\ \forall\phi\in\mathcal{H}(L),\label{eq: main expansion A}
\end{equation}
where $\mathcal{M}(\phi)$ is the K-energy (\ref{eq: K-energy}).
Then Donaldson introduced $\mathcal{Q}_{k}:\mathcal{H}(L)\rightarrow\mathbb{R}$
(denoted by $\tilde{\mathcal{L}}$ in \cite{donaldson_scalar_2005})
\begin{equation}
\mathcal{Q}_{k}(\phi)\coloneqq2k\left(\mathcal{E}(\phi)-\mathbf{E}_{k}\circ\mathbf{H}_{k}(\phi)\right).\label{eq: Archi Qk}
\end{equation}
By (\ref{eq: main expansion A}), it converges to $\mathcal{M}$ pointwisely.
Moreover, it is convex along the geodesics, see \cite[P473]{berndtsson_positivity_2009}.
By \cite[Corollary 1]{donaldson_scalar_2005}, the critical points
of $\mathcal{Q}_{k}$ are \emph{balanced metrics} which satisfy 
\[
\mathbf{FS}_{k}\circ\mathbf{H}_{k}(\phi)=\phi\ \Leftrightarrow\ \rho_{k}(\phi)\equiv N_{k}.
\]
There is a bijection (both sets may be empty): 
\[
\left\{ \phi\in\mathcal{H}(L)\mid\mathbf{FS}_{k}\circ\mathbf{H}_{k}(\phi)=\phi\right\} \xrightleftharpoons[\mathbf{FS}_{k}]{\mathbf{H}_{k}}\left\{ H\in\mathcal{H}(R_{k})\mid\mathbf{H}_{k}\circ\mathbf{FS}_{k}(H)=H\right\} .
\]
The norms in the right-hand set are called \emph{balanced norms}.
Let $\iota_{k}:X\hookrightarrow\mathbb{P}R_{k}^{*}$ be the Kodaira
embedding. The balanced norms are exactly the critical points of $\mathbf{M}_{\iota_{k}(X)}$
(\ref{eq:KN final}), the Kempf--Ness functional for Chow-stability.
So the existence of critical points of $\mathcal{Q}_{k}$ is equivalent
to the Chow-polystability of $\iota_{k}(X)$. 

Suppose that $\mathcal{M}$ admits a critical point $\phi_{\mathrm{cscK}}$,
since $\mathcal{Q}_{k}$ approximates $\mathcal{M}$, naturally one
may expect that $\mathcal{Q}_{k}$ also admits a critical point $\phi_{k}$
for $k\gg1$. If so, then $(X,L)$ would be asymptotic Chow-polystable.
Furthermore, $\phi_{k}$ should converge to $\phi_{\mathrm{cscK}}$.
This is not true in general, see \cite{mabuchi_energy-theoretic_2005,ono_example_2012}.
But when the automorphism group is discrete, \cite{donaldson_scalar_2001}
confirms these expectations. 
\begin{thm}[\cite{donaldson_scalar_2001}]
 \label{thm: Donaldson}Suppose that $(X,L)$ admits a cscK metric
$\phi_{\mathrm{cscK}}\in\mathcal{H}(L)$ and $\mathrm{Aut}(X,L)/\mathbb{C}^{*}$
is discrete, then $(X,kL)$ is Chow-stable for $k\gg1$. Moreover,
there exists a sequence of balanced norms $H_{k}\in\mathcal{H}(R_{k})$
such that the corresponding balanced metrics $\mathbf{FS}_{k}(H_{k})$
converge to $\phi_{\mathrm{cscK}}$ in $C^{\infty}$-norm. 
\end{thm}

\section{Non-Archimedean pluripotential theory}

To connect the maximal Chow-destabilizers to the maximal K-stabilizer,
we adapt the variational approach \cite{donaldson_scalar_2005} from
the Archimedean setting to the NA setting. It is based on the NA pluripotential
theory developed by Boucksom and Jonsson \cites{boucksom_non-archimedean_2018}{boucksom_global_2022}{boucksom_non-archimedean_2024}.
We briefly review it below. 

\subsection{\label{subsec:Berkovich}Berkovich's analytification }

Let $(X,L)$ be a polarized variety over $\mathbb{C}$. We equip $\mathbb{C}$
with the trivial valuation. As a set, Berkovich's analytification
of $X$ is 
\[
X^{\an}\coloneqq\left\{ v:\mathbb{C}(Y)\rightarrow\mathbb{R}\cup\{\infty\}\mid Y\subset X\right\} ,
\]
where $Y$ runs over all subvarieties, and $v$ is any valuation on
the functional field $\mathbb{C}(Y)$ which extends the trivial valuation
on $\mathbb{C}\subset\mathbb{C}(Y)$. Let $\left|\right|_{v}\coloneqq e^{-v}$
be the associated NA norm. For $f\in\mathbb{C}(Y)$, we write $\left|f(v)\right|\coloneqq\left|f\right|_{v}=e^{-v(f)}$.
Denote by $v_{tr}$ the trivial valuation on $\mathbb{C}(X)$. 

The kernel map $\mathrm{ker}:X^{\an}\rightarrow X$ sends $v$ to
the generic point of $Y$. The center $c(v)$ of $v$ is the unique
scheme point on $Y$ such that $v\geq0$ on $\mathcal{O}_{Y,c(v)}$
and $v>0$ on the maximal ideal $\mathfrak{m}_{Y,c(v)}\subset\mathcal{O}_{Y,c(v)}$.
This gives us the center map $c:X^{\an}\rightarrow X$. Berkovich's
topology on $X^{\an}$ is the weakest one such that $\mathrm{ker}$
is continuous, and for any open subset $U\subset X$ and $f\in\mathcal{O}_{X}(U)$,
the map 
\[
U^{\an}\coloneqq\mathrm{ker}^{-1}(U)\ra\bbr_{\geq0},\ v\mapsto\left|f(v)\right|
\]
is continuous. It makes $X^{\an}$  a compact Hausdorff space, but
the center map $c$ is anti-continuous (preimage of open subset is
closed). Let $X^{\mathrm{val}}\subset X^{\an}$ be the set of valuations
on $\mathbb{C}(X)$, and $X^{\mathrm{div}}\subset X^{\mathrm{val}}$
the set of divisorial valuations, which is dense in $X^{\an}$. 

The line bundle $L$ also has an analytification, denoted by $L^{\an}$.
A metric $\phi$ on $L^{\an}$ (called a NA metric) assigns to every
local section $s\in\Gamma(U,L)$ a function $\left|s\right|_{\phi}:U^{\an}\rightarrow[0,\infty]$
such that 
\[
\left|fs\right|_{\phi}(v)=\left|f(v)\right|\cdot\left|s\right|_{\phi}(v),\ \forall f\in\mathcal{O}_{X}(U),\ \forall v\in U^{\an}.
\]
Then we have $\left|s+s'\right|_{\phi}\leq\max\left(\left|s\right|_{\phi},\left|s'\right|_{\phi}\right)$
for $s,s'\in\Gamma(U,L)$. 

Owing to the trivial norm on $\mathbb{C}$, there is a distinguished
metric $\phi_{tr}$ on $L^{\an}$, called the \emph{trivial metric}.
We will denote by $\left|s\right|_{tr}$ or $\left|s\right|$ the
norm function of a section $s$ w.r.t. the trivial metric. Next, we
give the definition of $\phi_{tr}$. Given $s\in\Gamma(U,L)$ and
$v\in U^{\an}$, we firstly define the valuation $v(s)$, then define
$\left|s\right|_{tr}(v)\coloneqq e^{-v(s)}.$ 

(1) If $c(v)\in U$, we can trivialize $s$ at $c(v)$, namely take
a generator $\mathfrak{e}$ of the stalk $L_{c(v)}$ and let $s|_{c(v)}=f\mathfrak{e}$
for a $f\in\mathcal{O}_{X,c(v)}$. Then we define $v(s)=v(f)\in[0,\infty]$.
By the property of $c(v)$, $v(s)$ is independent of the choices
of $\mathfrak{e}$. 

(2) If $c(v)\notin U$, we can't trivialize $s$ at $c(v)$. Since
$c(v)\in\overline{\mathrm{ker}(v)}$, there is a morphism between
stalks $\alpha:L_{c(v)}\rightarrow L_{\mathrm{ker}(v)}$. Take a generator
$\mathfrak{e}\in L_{c(v)}$, then $\alpha(\mathfrak{e})$ is a generator
of $L_{\mathrm{ker}(v)}$. Let $s|_{\mathrm{ker}(v)}=f\alpha(\mathfrak{e})$
for a $f\in\mathcal{O}_{X,\mathrm{ker}(v)}$, now we define $v(s)=v(f)\in\bbr\cup\{\infty\}$,
which is also independent of the choices of $\mathfrak{e}$. 

With the trivial metric $\phi_{tr}$, other metrics can be represented
as $\phi_{tr}+\varphi$, where $\varphi:X^{\an}\rightarrow\mathbb{R}\cup\{-\infty\}$
is a function. We will use this function $\varphi$ to represent the
metric $\phi_{tr}+\varphi$. Given a local section $s\in\Gamma(U,L^{m})$,
its norm function w.r.t. the metric $\varphi$ is 
\[
\left|s\right|_{\varphi}\coloneqq\left|s\right|_{tr}e^{-m\varphi}=e^{-v(s)-m\varphi}:U^{\an}\rightarrow[0,\infty].
\]
There is a $\mathbb{R}_{>0}$-action $*$ on the functions on $X^{\an}$
which is defined by 
\begin{equation}
(a*\varphi)(v)\coloneqq a\varphi(a^{-1}v),\ a>0,\ v\in X^{\an}.\label{eq: R+ action on func}
\end{equation}

\subsection{NA psh metrics and energy functionals}

Given a nonzero section $s\in R_{m}$, the function $m^{-1}\log\left|s\right|_{tr}$
defines a metric on $L^{\an}$. A Fubini--Study metric $\varphi$
on $L^{\an}$ is defined to have the form 
\begin{equation}
\varphi=\frac{1}{m}\max_{j}\left\{ \log\left|s_{j}\right|_{tr}+\lambda_{j}\right\} ,\label{eq: def FS metric}
\end{equation}
where $m\geq1$, $\lambda_{j}\in\mathbb{R}$ and $\{s_{j}\}\subset R_{m}$
are finitely many sections without common zero. Note that $\vphi$
is continuous and $\sup_{X^{\an}}\varphi=\varphi(v_{tr})$. Let $\mathcal{H}_{\mathbb{R}}^{\na}$
be the set of all Fubini--Study metrics. 

A test-configuration $(\mathcal{X},\mathcal{L})$ of $(X,L)$ induces
a function $\varphi_{\mathcal{L}}$ on $X^{\an}$, refer to \cite[§A1]{boucksom_non-archimedean_2024}.
By \cite[Corollary A13]{boucksom_non-archimedean_2024}, $(\mathcal{X},\mathcal{L})\mapsto\varphi_{\mathcal{L}}$
gives a bijection from the set of normal ample test-configurations
to $\mathcal{H}_{\mathbb{Q}}^{\na}$, the set of metrics (\ref{eq: def FS metric})
with $\lambda_{j}\in\mathbb{Q}$. 

A psh metric $\varphi:X^{\an}\rightarrow\mathbb{R}\cup\{-\infty\}$
(excluding $\varphi\equiv-\infty$) is defined to be the pointwise
limit of a decreasing sequence in $\mathcal{H}_{\mathbb{Q}}^{\na}$.
It is usc and satisfies $\sup_{X^{\an}}\varphi=\varphi(v_{tr})$.
Denote by $\mathrm{PSH}^{\na}(L)$ the set of psh metrics on $L^{\an}$,
and $\mathrm{CPSH}^{\na}(L)$ the set of continuous psh metrics on
$L^{\an}$. By Dini's theorem, $\mathrm{CPSH}^{\na}(L)$ is the closure
of $\mathcal{H}_{\mathbb{Q}}^{\na}$ in the space of continuous functions
with respect to the $C^{0}$-norm. 

The NA Monge--Ampère energy $\mathbb{E}$ is a functional on $\mathrm{PSH}^{\na}(L)$.
For $\varphi=\varphi_{\mathcal{L}}$ induced by a normal ample test-configuration
$(\mathcal{X},\mathcal{L})$, $\mathbb{E}(\varphi)$ is the leading
term of expansion (\ref{eq:expand weight}). For general $\varphi\in\mathrm{PSH}^{\na}(L)$,
we define 
\begin{equation}
\mathbb{E}(\varphi)\coloneqq\inf\left\{ \mathbb{E}(\phi)\mid\varphi\leq\phi\in\mathcal{H}_{\mathbb{Q}}^{\na}\right\} \in\mathbb{R}\cup\{-\infty\}.\label{eq: NA MA}
\end{equation}
The space of finite-energy metrics is defined as 
\[
\mathcal{E}_{\na}^{1}(L)\coloneqq\left\{ \varphi\in\mathrm{PSH}^{\na}(L)\mid\mathbb{E}(\varphi)>-\infty\right\} .
\]
There is a metric $d_{1}$ on $\mathcal{E}_{\na}^{1}(L)$ making it
a complete geodesic space, see \cite{reboulet_plurisubharmonic_2022}
and \cite[Example 5.12]{xia_mabuchi_2023}. 

The NA K-energy $\mathbb{M}$ is introduced in \cite{boucksom_uniform_2017}
as a modification of the Donaldson-Futaki invariants. For $\varphi=\varphi_{\mathcal{L}}$
induced by an ample test-configuration $(\mathcal{X},\mathcal{L})$,
it is equal to 
\begin{equation}
\mathbb{M}(\varphi)=\mathrm{DF}(\mathcal{X},\mathcal{L})-V^{-1}(\mathcal{X}_{0}-\mathcal{X}_{0,\mathrm{red}})\cdot\mathcal{L}^{n}\leq\mathrm{DF}(\mathcal{X},\mathcal{L}).\label{eq: NA K-energy}
\end{equation}
So if $\mathcal{X}_{0}$ is reduced, $\mathbb{M}(\varphi)$ is the
same as $\mathrm{DF}(\mathcal{X},\mathcal{L})$. It can be extended
to a functional $\mathbb{M}:\mathcal{E}_{\na}^{1}(L)\rightarrow\bbr\cup\{\infty\}$,
see \cite[§4.1]{boucksom_non-archimedean_2023}. Compared with $\mathrm{DF}$,
$\mathbb{M}$ is homogeneous, i.e. $\mathbb{M}(a*\phi)=a\mathbb{M}(\phi)$
for $a>0$. 
\begin{rem}
The Definition \ref{def: K-stability} (3) for K-semistability is
equivalent to that $\mathbb{M}\geq0$ on $\mathcal{H}_{\mathbb{Q}}^{\na}$.
One direction follows by the inequality (\ref{eq: NA K-energy}).
Another direction follows by \cite[Proposition 7.16]{boucksom_uniform_2017}. 
\end{rem}

\subsection{\label{subsec: limit NA metric}The limit NA metric associated to
a psh ray}

We recall an important construction in \cite[§4.2]{berman_variational_2021},
which assigns a NA metric to a sublinear psh ray. 

Fix a reference metric $\phi_{r}\in\mathcal{H}(L)$ and let $\omega_{r}=dd^{c}\phi_{r}$.
Suppose $\phi=(\phi_{t})_{t>0}\subset\mathrm{PSH}(L)$ is sublinear
psh ray, where ``sublinear'' means that $u_{t}\coloneqq\phi_{t}-\phi_{r}\leq at+C$
for all $t>0$, and $a>0$, $C$ are constants. Let $\triangle\coloneqq\{\left|\tau\right|<1\}$
be the unit disk. The function 
\begin{equation}
V(x,\tau)\coloneqq u_{-\log\left|\tau\right|^{2}}(x)+a\log\left|\tau\right|^{2}\label{eq:psh on product}
\end{equation}
is bounded above and defines a $\pi_{1}^{*}\omega_{r}$-psh function
on $X\times\triangle$. For a divisorial valuation $v\in X^{\mathrm{div}}$,
let $\sigma(v)\in\left(X\times\mathbb{C}\right)^{\mathrm{div}}$ be
the Gauss extension, which is $\mathbb{C}^{*}$-invariant and such
that $\sigma(v)(\tau)=1$. Then we define (\cite[Definition 4.2]{berman_variational_2021})
\begin{equation}
\phi^{\na}:X^{\mathrm{div}}\rightarrow\mathbb{R},\ \phi^{\na}(v)=-\sigma(v)(V)+a,\label{eq: define u^NA}
\end{equation}
where $\sigma(v)(V)$ is the generic Lelong number, see \cite[§1.4]{xia_singularities_2025}.
Since $\sigma(v)(\log\left|\tau\right|^{2})=1$, it is independent
of the choices of $a$. By \cite[Theorem 6.2]{berman_variational_2021},
$\phi^{\na}$ can be extended uniquely to a metric in $\mathrm{PSH}^{\na}(L)$.
We call it the \emph{limit NA metric} associated to $\phi=(\phi_{t})_{t>0}$. 

The $L^{1}$-geodesic rays are sublinear psh rays, and their limit
NA metrics belong to $\mathcal{E}_{\na}^{1}(L)$. This gives us a
map 
\begin{equation}
\Pi:\mathcal{R}^{1}(L)\rightarrow\mathcal{E}_{\na}^{1}(L),\ \ell\mapsto\ell^{\na},\label{eq: limit NA metric}
\end{equation}
which is equivariant w.r.t. time-scaling on $\mathcal{R}^{1}(L)$
and the $\mathbb{R}_{>0}$-action (\ref{eq: R+ action on func}) on
$\mathcal{E}_{\na}^{1}(L)$. 

By \cite[Theorem 6.4]{berman_variational_2021}, for any $L^{1}$-geodesic
ray $(\ell_{t})$, we have 
\begin{equation}
\mathbb{E}(\ell^{\na})\geq\lim_{t\rightarrow\infty}\frac{1}{t}\mathcal{E}(\ell_{t}).\label{eq:NA E > slope}
\end{equation}
The inequality may be strict, and the equality holds if and only if
$(\ell_{t})$ is maximal in the following sense. 
\begin{defn}[\cite{berman_variational_2021}]
\label{def: BBJ max ray} A $L^{1}$-geodesic ray $(\ell_{t})_{t\geq0}$
is called \emph{maximal} if for any sublinear psh ray $\phi=(\phi_{t})_{t>0}\subset\mathcal{E}^{1}(L)$
such that $\lim_{t\rightarrow0}\phi_{t}\leq\ell_{0}$ and $\phi^{\na}\leq\ell^{\na}$,
we have $\phi_{t}\leq\ell_{t}$ for all $t>0$.
\end{defn}

Conversely, by \cite[Theorem 6.6]{berman_variational_2021}, for any
$\ell_{0}\in\mathcal{E}^{1}(L)$ and $\vphi\in\mathcal{E}_{\na}^{1}(L)$,
there is a unique maximal $L^{1}$-geodesic ray $\ell$ starting from
$\ell_{0}$ and $\ell^{\na}=\vphi$. When $\ell_{0}\in\mathcal{H}(L)$
and $\varphi$ is induced by a test-configuration, $\ell$ is the
geodesic ray constructed by Phong--Sturm \cite{phong_test_2007}.
Proposition \ref{prop: for BBJ imbedd} allows us to define a starting-point-free
version of BBJ embedding 
\begin{equation}
\iota:\mathcal{E}_{\na}^{1}(L)\hookrightarrow\mathcal{R}^{1}(L)/_{\sim},\ \varphi\mapsto[\ell],\label{eq:BBJ emb}
\end{equation}
where $\ell$ is any maximal $L^{1}$-geodesic ray such that $\ell^{\na}=\vphi$.
One can check that $\iota$ is equivariant w.r.t. the $\mathbb{R}_{>0}$-action.
Reboulet \cite[Theorem 4.4.1]{reboulet_space_2023} shows that $\iota$
is an isometry with respect to the metric $d_{1}$ and $d_{1}^{c}$.
For the $d_{p}$ metric, there is a relevant result due to Finski
\cite[Theorem 2.7]{finski_geometry_2024}. 

\subsection{The NA formula for Chow-weights}

We have defined the Chow-weights $M_{X}$ as the limit slopes of $\mathbf{M}_{X}$.
Now we express this limit by the NA MA-energy $\mathbb{E}$. 

First we clarify some notations. Let $X\subset\mathbb{P}V^{*}$ be
a smooth projective variety. Take a reference metric $\phi_{r}$ on
$L=\mathcal{O}(1)|_{X}$ with curvature form $\omega_{r}$. The Fubini--Study
map $\mathbf{FS}:\mathcal{H}(V)\rightarrow\mathcal{H}(L)$ is (\ref{eq:def of FS}).
NA Fubini--Study map $\mathrm{FS}:\mathcal{N}(V)\rightarrow\mathrm{PSH}^{\na}(L)$
is defined by 
\begin{equation}
\mathrm{FS}(\chi)=\max_{1\leq i\leq N}\left(\log\left|s_{i}\right|_{tr}+\chi(s_{i})\right),\label{eq:def NA FS}
\end{equation}
where $(s_{i})\subset V$ is any basis diagonalizing $\chi$, and
we take $s_{i}$ as sections of $L$. 
\begin{thm}
\label{thm: NA express Chow}Let $X\subset\mathbb{P}V^{*}$ be a smooth
projective variety, and $H=(H_{t})_{t\geq0}\subset\mathcal{H}(V)$
be a geodesic ray with limit NA norm $\chi_{H}$. 

(1) The limit NA metric associated to the psh ray $\left(\mathbf{FS}(H_{t})\right)_{t}$
is $\mathrm{FS}(\chi_{H})$. 

(2) NA formula for Chow-weight: 
\begin{equation}
M_{X}(\chi_{H})\coloneqq\lim_{t\rightarrow\infty}\frac{1}{t}\mathbf{M}_{X}\left(H_{t}\right)=\mathbb{E}\circ\mathrm{FS}(\chi_{H})-E_{V}(\chi_{H}),\label{eq: Chow NA express}
\end{equation}
where $\mathbb{E}$ is the NA MA-energy, $\mathrm{FS}$ is the NA
Fubini--Study map (\ref{eq:def NA FS}) and $E_{V}(\chi_{H})$ is
the volume (\ref{eq: volume NA}). 
\end{thm}

\begin{proof}
(1) It follows by a direct computation. As Example \ref{exa: NAlimit of geodesic},
we take a $H_{0}$-orthonormal basis $(s_{i})$ diagonalizing all
$H_{t}$ and assume $H_{t}(s_{i})^{2}=e^{-t\lambda_{i}}$. Then $\chi_{H}$
is also diagonalized by $(s_{i})$ and $\chi_{H}(s_{i})=\lambda_{i}$.
Hence $\mathrm{FS}(\chi_{H})=\max_{i}\left(\log\left|s_{i}\right|+\lambda_{i}\right)$.
By Proposition \ref{prop: conv of Mk} (1), $\mathbf{FS}(H_{t})$
is a psh ray. Explicitly, we have 
\[
u_{t}\coloneqq\mathbf{FS}(H_{t})-\phi_{r}=\log\sum_{i=1}^{N}e^{t\lambda_{i}}\left|s_{i}\right|_{\phi_{r}}^{2}\leq t\varLambda+C,\ \varLambda\coloneqq\max_{i}\lambda_{i}.
\]
The induced $\pi_{1}^{*}\omega_{r}$-psh function (\ref{eq:psh on product})
on $X\times\triangle$ is
\[
V(x,\tau)\coloneqq\log\sum_{i=1}^{N}\left|\tau\right|^{-2\mu_{i}}\left|s_{i}(x)\right|_{\phi_{r}}^{2},\ \mu_{i}\coloneqq\lambda_{i}-\varLambda.
\]
Denote by $u^{\na}$ the limit NA metric associated to $\left(\mathbf{FS}(H_{t})\right)_{t}$.
By definition (\ref{eq: define u^NA}), 
\[
u^{\na}(v)=-\sigma(v)(V)+\varLambda,\ \forall v\in X^{\mathrm{div}}.
\]
To compute the Lelong number, it is easier to consider 
\[
U(x,\tau)\coloneqq\log\max_{1\leq i\leq N}\left|\tau\right|^{-2\mu_{i}}\left|s_{i}(x)\right|_{\phi_{r}}^{2}=\max_{1\leq i\leq N}\left(\log\left|s_{i}(x)\right|_{\phi_{r}}^{2}+\log\left|\tau\right|^{-2\mu_{i}}\right).
\]
It is also $\pi_{1}^{*}\omega_{r}$-psh and satisfies $U\leq V\leq U+\log N$,
so we have 
\begin{eqnarray*}
\sigma(v)(V)=\sigma(v)(U) & = & \min_{1\leq i\leq N}\sigma(v)\left(\log\left|s_{i}(x)\right|_{\phi_{r}}^{2}\right)+\sigma(v)\left(\log\left|\tau\right|^{-2\mu_{i}}\right)\\
 & = & \min_{i}\left(v(s_{i})-\mu_{i}\right),
\end{eqnarray*}
where we used some basic properties of Lelong number, see \cite[§1.4]{xia_singularities_2025}.
Finally, 
\[
u^{\na}(v)=-\sigma(v)(V)+\varLambda=\max_{i}\left(\lambda_{i}-v(s_{i})\right)=\mathrm{FS}(\chi_{H})(v).
\]

(2) In the terminology of \cite{berman_variational_2021}, $\left(\mathbf{FS}(H_{t})\right)$
is a psh ray with algebraic singularities. By \cite[Lemma 5.3]{berman_variational_2021},
we have 
\[
\lim_{t\rightarrow\infty}\frac{1}{t}\mathcal{E}\left(\mathbf{FS}(H_{t})\right)=\mathbb{E}(u^{\na}).
\]
For the another term of $\mathbf{M}_{X}$, we have 
\[
\mathbf{E}_{V}(H_{t},H_{0})=\frac{t}{N}\sum_{i=1}^{N}\lambda_{i}=tE_{V}(\chi_{H}).
\]
Then (\ref{eq: Chow NA express}) follows. 
\end{proof}

\subsection{\label{subsec:FS SN}NA Fubini--Study maps \& sup-norm maps}

In the non-Archimedean setting, Boucksom--Jonsson \cite{boucksom_non-archimedean_2018}
introduced two maps which play the role of Fubini--Study map $\mathbf{FS}_{k}$
and Hilbert map $\mathbf{H}_{k}$. 
\begin{defn}
(1) Given a polarized variety $(X,L)$. Suppose $kL$ is globally
generated. The NA Fubini--Study map $\mathrm{FS}_{k}:\mathcal{N}(R_{k})\rightarrow\mathcal{H}_{\mathbb{R}}^{\na}$
is defined by 
\begin{equation}
\mathrm{FS}_{k}(\chi)\coloneqq\frac{1}{k}\sup_{s\in R_{k}\backslash\{0\}}\left(\log\left|s\right|_{tr}+\chi(s)\right)=\frac{1}{k}\max_{1\leq i\leq N_{k}}\left(\log\left|s_{i}\right|_{tr}+\chi(s_{i})\right),\label{eq: def FS_k}
\end{equation}
where $(s_{i})$ is any basis diagonalizing $\chi$. Note that $\mathrm{FS}_{k}(\chi_{tr})=0$
and $\mathrm{FS}_{k}(\chi)$ is increasing in $\chi$. For $c\in\mathbb{R}$
and $a>0$, we have 
\[
\mathrm{FS}_{k}(\chi+kc)=\mathrm{FS}_{k}(\chi)+c,\ \mathrm{FS}_{k}(a\chi)=a*\mathrm{FS}_{k}(\chi),
\]
where $a*\mathrm{FS}_{k}(\chi)$ is defined by (\ref{eq: R+ action on func}). 

(2) Let $\mathcal{L}^{\infty}$ be the set of bounded functions on
$X^{\an}$. The sup-norm map $\mathrm{SN}_{k}:\mathcal{L}^{\infty}\rightarrow\mathcal{N}(R_{k})$
is defined by 
\begin{equation}
\mathrm{SN}_{k}(\varphi)(s)\coloneqq-\log\left(\sup_{X^{\an}}\left|s\right|_{tr}e^{-k\varphi}\right)=\inf_{v\in X^{\an}}\left\{ v(s)+k\varphi(v)\right\} ,\label{eq: def SN_k}
\end{equation}
where $\varphi\in\mathcal{L}^{\infty}$ and $s\in R_{k}$. Note that
$\mathrm{SN}_{k}(0)=\chi_{tr}$ and $\mathrm{SN}_{k}(\varphi)$ is
increasing in $\varphi$. For $c\in\mathbb{R}$ and $a>0$, we have
\[
\mathrm{SN}_{k}(\varphi+c)=\mathrm{SN}_{k}(\varphi)+kc,\ \mathrm{SN}_{k}(a*\varphi)=a\mathrm{SN}_{k}(\varphi).
\]

The interactive properties between them are important in the sequel.
Note that many properties in below had been included in \cite[Lemma 7.23]{boucksom_spaces_2021}. 
\end{defn}

\begin{prop}
\label{prop: FS vs SN}Given a polarized variety $(X,L)$ such that
$kL$ is globally generated. 

(1) For any $\phi\in\mathcal{L}^{\infty}$, we have $[\phi]_{k}\coloneqq\mathrm{FS}_{k}\circ\mathrm{SN}_{k}(\phi)\leq\phi$
on $X^{\an}$.

(2) For any $\chi\in\mathcal{N}(R_{k})$, we have $\mathrm{SN}_{k}\circ\mathrm{FS}_{k}(\chi)\geq\chi$
on $R_{k}$. 

(3) The image set $\mathcal{FS}_{k}\coloneqq\mathrm{FS}_{k}\left(\mathcal{N}(R_{k})\right)$
and $\mathcal{SN}_{k}\coloneqq\mathrm{SN}_{k}\left(\mathcal{L}^{\infty}\right)$
can be characterized in the following way
\begin{eqnarray}
\mathcal{FS}_{k} & = & \{\phi\in\mathcal{L}^{\infty}\mid\mathrm{FS}_{k}\circ\mathrm{SN}_{k}(\phi)=\phi\};\label{eq: FS SN=00003Did}\\
\mathcal{SN}_{k} & = & \{\chi\in\mathcal{N}(R_{k})\mid\mathrm{SN}_{k}\circ\mathrm{FS}_{k}(\chi)=\chi\}.\label{eq:SN FS=00003Did}
\end{eqnarray}
Moreover, $\mathrm{FS}_{k}:\mathcal{SN}_{k}\rightarrow\mathcal{FS}_{k}$
is a bijection with inverse $\mathrm{SN}_{k}$. 
\[
\begin{array}{ccc}
\mathcal{N}(R_{k}) & \xrightleftharpoons[\mathrm{SN}_{k}]{\mathrm{FS}_{k}} & \mathcal{L}^{\infty}\\
\cup &  & \cup\\
\mathcal{SN}_{k} & \stackrel{\mathrm{1-1}}{\longleftrightarrow} & \mathcal{FS}_{k}
\end{array}
\]

(4) The map $\mathrm{FS}_{k}$ can be determined by $\mathrm{SN}_{k}$
(and vice versa) in the following way 
\begin{eqnarray*}
\mathrm{FS}_{k}(\chi)(v) & = & \inf\left\{ \phi(v)\mid\phi\in\mathcal{L}^{\infty},\ \mathrm{SN}_{k}(\phi)\geq\chi\right\} ,\ v\in X^{\an};\\
\mathrm{SN}_{k}(\phi)(s) & = & \sup\left\{ \chi(s)\mid\chi\in\mathcal{N}(R_{k}),\ \mathrm{FS}_{k}(\chi)\leq\phi\right\} ,\ s\in R_{k}.
\end{eqnarray*}
\end{prop}

\begin{proof}
(1) Denote $\chi_{\phi}\coloneqq\mathrm{SN}_{k}(\phi)$. For any $s\in R_{k}$
and $v\in X^{\an}$, by definition (\ref{eq: def SN_k}) of $\chi_{\phi}(s)$,
we have 
\[
\log\left|s\right|(v)+\chi_{\phi}(s)\leq\log\left|s\right|(v)+k\phi(v)-\log\left|s\right|(v)=k\phi(v).
\]
Take the supremum over $s\in R_{k}\backslash\{0\}$, we obtain $\mathrm{FS}_{k}(\chi_{\phi})(v)\leq\phi(v)$. 

(2) Denote $\phi_{\chi}\coloneqq\mathrm{FS}_{k}(\chi)$. By definition
(\ref{eq: def FS_k}), for any $s\in R_{k}\backslash\{0\}$, we have
$k\phi_{\chi}\geq\log\left|s\right|+\chi(s)$, i.e. $k\phi_{\chi}(v)+v(s)\geq\chi(s)$
for all $v\in X^{\an}$. Take the infimum over $v\in X^{\an}$, we
obtain $\mathrm{SN}_{k}(\phi_{\chi})(s)\geq\chi(s).$

(3) Obviously, $\mathcal{FS}_{k}$ contains the right-hand side of
(\ref{eq: FS SN=00003Did}). Conversely, let $\phi=\mathrm{FS}_{k}(\chi)\in\mathcal{FS}_{k}$,
by (1) we have $\mathrm{FS}_{k}\circ\mathrm{SN}_{k}(\phi)\leq\phi$.
For the reverse inequality, by (2) we have $\mathrm{SN}_{k}(\phi)=\mathrm{SN}_{k}\circ\mathrm{FS}_{k}(\chi)\geq\chi$.
Since $\mathrm{FS}_{k}$ is increasing, it implies $\mathrm{FS}_{k}\circ\mathrm{SN}_{k}(\phi)\geq\mathrm{FS}_{k}(\chi)=\phi$.
The proof of (\ref{eq:SN FS=00003Did}) is similar. 

Next, we show the bijectivity of $\mathrm{FS}_{k}:\mathcal{SN}_{k}\rightarrow\mathcal{FS}_{k}$.
For any $\phi\in\mathcal{FS}_{k}$, we have showed $\mathrm{FS}_{k}\circ\mathrm{SN}_{k}(\phi)=\phi$,
so the map is surjective. If $\mathrm{FS}_{k}(\chi)=\mathrm{FS}_{k}(\chi')$
for $\chi,\chi'\in\mathcal{SN}_{k}$, taking $\mathrm{SN}_{k}$ on
both sides, by (\ref{eq:SN FS=00003Did}) we have $\chi=\chi'$. Hence
the map is injective. 

(4) We only show the first identity, the second one is similar. Let
$\phi=\mathrm{FS}_{k}(\chi)$, by (2) we have $\mathrm{SN}_{k}(\phi)\geq\chi$,
hence $\mathrm{FS}_{k}(\chi)\geq\inf$. On the other hand, for any
$\phi\in\mathcal{L}^{\infty}$ such that $\mathrm{SN}_{k}(\phi)\geq\chi$.
Since $\mathrm{FS}_{k}$ is increasing, we have $\mathrm{FS}_{k}\circ\mathrm{SN}_{k}(\phi)\geq\mathrm{FS}_{k}(\chi)$.
By (1), it implies $\phi\geq\mathrm{FS}_{k}(\chi)$, thus $\inf\geq\mathrm{FS}_{k}(\chi)$,
so $\inf=\mathrm{FS}_{k}(\chi)$. 
\end{proof}

\subsection{Properties of maximal Chow-destabilizers }

Let $X\subset\mathbb{P}V^{*}$ be a smooth projective variety. Set
$L=\mathcal{O}(1)|_{X}$. We have maps $\mathrm{FS}:\mathcal{N}(V)\rightarrow\mathrm{PSH}^{\na}(L)$
(\ref{eq:def NA FS}) and $\mathrm{SN}:\mathrm{PSH}^{\na}(L)\rightarrow\mathcal{N}(V)$
which is defined as 
\[
\mathrm{SN}(\varphi)(s)\coloneqq\inf_{v\in X^{\an}}\left\{ v(s)+\varphi(v)\right\} ,\ \forall s\in V,
\]
where we take $s$ as a section of $L$. These two maps also satisfy
the properties in Proposition \ref{prop: FS vs SN}. 
\begin{thm}
\label{thm: property of Chow-destab}Suppose $X\subset\mathbb{P}V^{*}$
is Chow-unstable, let $\bar{\chi}\in\mathcal{N}(V)$ be a maximal
Chow-destabilizer for $X$, namely a minimizer of the normalized Chow-weight
$\bar{M}_{X}(\chi)=M_{X}(\chi)/\left\Vert \chi\right\Vert _{2}$. 

(1) The jumping numbers $\lam_{i}(\bar{\chi})$ satisfy 
\begin{equation}
\sum_{i=1}^{\dim V}\lam_{i}(\bar{\chi})=0\ \textrm{and}\ \lam_{i}(\bar{\chi})\leq\left(-\inf_{\mathcal{N}(V)}\bar{M}\right)^{-1}\left\Vert \bar{\chi}\right\Vert ,\ \forall i.\label{eq: upper of Chow-destab}
\end{equation}

(2) $\mathrm{SN}\circ\mathrm{FS}(\bar{\chi})=\bar{\chi}$. 
\end{thm}

\begin{proof}
(1) First, we claim that $\breve{\chi}\coloneqq(-\inf\bar{M})\frac{\bar{\chi}}{\left\Vert \bar{\chi}\right\Vert }$
is the unique minimizer of 
\[
\inf\left\{ \breve{M}(\chi)\coloneqq M(\chi)+\frac{1}{2}\left\Vert \chi\right\Vert _{2}^{2}\mid\chi\in\mathcal{N}(V)\right\} .
\]
For any $\chi\in\mathcal{N}(V)$, consider $h(t)\coloneqq\breve{M}(t\chi)=tM(\chi)+\frac{t^{2}}{2}\left\Vert \chi\right\Vert _{2}^{2}$
for $t>0$. If $M(\chi)\geq0$, we have $\inf_{t>0}h(t)=0$; if $M(\chi)<0$,
we have $\inf_{t>0}h(t)=-\frac{1}{2}\bar{M}(\chi)^{2}$ and the infimum
is attained at $t=-\bar{M}(\chi)/\left\Vert \chi\right\Vert $. It
implies that $\breve{\chi}$ is a minimizer of $\breve{M}$. The uniqueness
is due to the convexity of $M_{X}$ (Proposition \ref{prop:convexity of Chow})
and the uniform convexity of $\chi\mapsto\left\Vert \chi\right\Vert _{2}^{2}$. 

Recall the NA formula (\ref{eq: Chow NA express}) of $M_{X}$, 
\begin{eqnarray*}
\breve{M}(\chi) & = & \mathbb{E}\circ\mathrm{FS}(\chi)-\frac{1}{N}\sum_{i=1}^{N}\lambda_{i}(\chi)+\frac{1}{2N}\sum_{i=1}^{N}\lambda_{i}(\chi)^{2}\\
 & = & \mathbb{E}\circ\mathrm{FS}(\chi)+\frac{1}{2N}\sum_{i=1}^{N}(\lambda_{i}(\chi)-1)^{2}-\frac{1}{2},
\end{eqnarray*}
where $\lam_{i}(\chi)$ are jumping numbers of $\chi$. Function $t\mapsto\breve{M}(\breve{\chi}+t)$
attains its minimum at $t=0$. Take derivative at $t=0$, we obtain
$\sum\lam_{i}(\breve{\chi})=0$, so $\sum\lam_{i}(\bar{\chi})=0$. 

Given any $\chi\in\mathcal{N}(V)$, consider $\chi'=\chi\wedge(\chi_{tr}+1)$,
i.e. the minimum of $\chi$ and $\chi_{tr}+1$. Its jumping numbers
are $\lam_{i}(\chi')=\min\{\lam_{i}(\chi),1\}$. Since $\mathbb{E}$
and $\mathrm{FS}$ are increasing, $\mathbb{E}\circ\mathrm{FS}(\chi')\leq\mathbb{E}\circ\mathrm{FS}(\chi)$.
These imply $\breve{M}(\chi')\leq\breve{M}(\chi)$. If we take $\chi=\breve{\chi}$,
then $\chi'$ is also a minimizer of $\breve{M}$, thus $\chi'=\chi$
by the uniqueness. So we have $\lam_{i}(\breve{\chi})\leq1$ for all
$i$. This gives us the upper bound in (\ref{eq: upper of Chow-destab}). 

(2) For any $\chi\in\mathcal{N}(V)$ with $\lam_{i}(\chi)\leq1$ for
all $i$ (i.e. $\chi\leq\chi_{tr}+1$). Consider $\chi'\coloneqq\mathrm{SN}\circ\mathrm{FS}(\chi)$.
Since $\mathrm{SN}$ and $\mathrm{FS}$ are increasing, we have $\chi'\leq\mathrm{SN}\circ\mathrm{FS}(\chi_{tr}+1)=\chi_{tr}+1$,
hence $\lam_{i}(\chi')\leq1$ for all $i$. By Proposition \ref{prop: FS vs SN},
we have $\chi'\geq\chi$ and $\mathrm{FS}(\chi')=\mathrm{FS}(\chi)$.
Inequality $\chi'\geq\chi$ implies $\lam_{i}(\chi')\geq\lam_{i}(\chi)$.
By these facts, we can conclude $\breve{M}(\chi')\leq\breve{M}(\chi)$.
If we take $\chi=\breve{\chi}$, then $\chi'$ is also a minimizer,
thus $\chi'=\chi$ by the uniqueness. Hence $\mathrm{SN}\circ\mathrm{FS}(\breve{\chi})=\breve{\chi}$,
so $\mathrm{SN}\circ\mathrm{FS}(\bar{\chi})=\bar{\chi}$. 
\end{proof}

\section{Maximal K-destabilizers}

\subsection{The maximal K-destabilizing rays of Xia}

For a polarized manifold $(X,L)$, Donaldson \cite{donaldson_lower_2005}
showed a moment-weight inequality
\begin{equation}
\inf_{\phi\in\mathcal{H}(L)}\mathrm{Cal}(\phi)^{1/2}\geq\sup_{(\mathcal{X},\mathcal{L})}\frac{-\mathrm{DF}(\mathcal{X},\mathcal{L})}{\left\Vert (\mathcal{X},\mathcal{L})\right\Vert _{2}},\label{eq:MW Donaldson}
\end{equation}
where $\mathrm{Cal}(\phi)\coloneqq\int_{X}\left(S(\phi)-\bar{S}\right)^{2}\mathrm{MA}(\phi)$
is the Calabi energy and $(\mathcal{X},\mathcal{L})$ runs over all
test-configuration for $(X,L)$. Donaldson conjectured that the equality
holds. 

When $(X,L)$ is the toric manifold associated to Delzant polytope
$P$, Székelyhidi \cite{szekelyhidi_optimal_2008} confirmed Donaldson's
conjecture by assuming that the Calabi flow exists for all time. Restricting
on the toric test-configurations, the right-hand side optimization
problem can be reduced into 
\begin{equation}
\inf\left\{ \frac{\mathcal{L}(f)}{\left\Vert f\right\Vert _{2,P}}\mid\mathcal{L}(f)\coloneqq\fint_{P}f\mathrm{d}x-\fint_{\partial P}f\mathrm{d}\sigma\right\} \label{eq:min DF toric}
\end{equation}
where $f:P\ra\bbr$ are rational piecewise-linear concave functions.
After enlarging the domain of (\ref{eq:min DF toric}), Székelyhidi
obtained a minimizer called the optimal ``test-configuration''. 

In the general case, Xia \cite{xia_sharp_2021} established an equality
which enlarges the domains of both sides of (\ref{eq:MW Donaldson}).
In \cite{darvas_geodesic_2020}, for a $L^{1}$-geodesic ray $\ell=(\ell_{t})_{t\geq0}$,
the radial K-energy $\mathcal{M}^{\mathrm{rad}}(\ell)$ is defined
as the limit slope of K-energy along this ray 
\begin{equation}
\mathcal{M}^{\mathrm{rad}}(\ell)\coloneqq\lim_{t\rightarrow+\infty}\frac{1}{t}\mathcal{M}(\ell_{t})\in(-\infty,\infty].\label{eq:radial K-energy}
\end{equation}
By \cite[Lemma 4.10]{darvas_geodesic_2020}, $\mathcal{M}^{\mathrm{rad}}(\ell)$
only depends on the $d_{1}$-asymptotic class of $\ell$. Hence it
descends to a functional on $\mathcal{R}^{1}(L)/_{\sim}$, which is
convex along the chord geodesics, see \cite[Theorem 4.11]{darvas_geodesic_2020}.
By considering the weak gradient flow of convex functional $\mathcal{M}$
on Hadamard space $\mathcal{E}^{2}(L)$, Xia obtain 
\begin{thm}[\cite{xia_sharp_2021}]
\label{thm: Xia} For polarized manifold $(X,L)$, we have
\[
\inf_{\phi\in\mathcal{E}^{2}(L)}\mathrm{Cal}(\phi)=\sup_{\ell\in\mathcal{R}^{2}(L)}\frac{-\mathcal{M}^{\mathrm{rad}}(\ell)}{\left\Vert \ell\right\Vert _{2}},
\]
where $\mathrm{Cal}(\phi)$ is the gradient norm of $\mathcal{M}$
(see \cite{xia_sharp_2021} 3.2), and $\left\Vert \ell\right\Vert _{2}\coloneqq d_{2}(\ell_{0},\ell_{1})$
is the $d_{2}$-speed of $\ell$. Moreover, if $\inf_{\mathcal{R}^{2}}\mathcal{M}^{\mathrm{rad}}<0$
(called geodesically K-unstable), up to time-scaling and $d_{2}$-asymptotic
relation, the right-hand side admits a unique maximizer. 
\end{thm}

\begin{rem}
The statement for the uniqueness is slightly different from \cite[Theorem 1.1, Corollary 1.2]{xia_sharp_2021},
since the author uses the starting-point-fixed space $\mathcal{R}_{\phi_{r}}^{2}(L)$.
Suppose $\ell$ and $\ell'$ are two unit-speed minimizers for $\mathcal{M}^{\mathrm{rad}}(\ell)/\left\Vert \ell\right\Vert _{2}$.
Let $\bar{\ell}$ be the $L^{2}$-geodesic ray $d_{2}$-asymptotic
to $\ell'$ and $\bar{\ell}_{0}=\ell_{0}$. By \cite[Lemma 4.10]{darvas_geodesic_2020},
$\mathcal{M}^{\mathrm{rad}}(\bar{\ell})=\mathcal{M}^{\mathrm{rad}}(\ell')$,
so $\bar{\ell}$ is also a minimizer with unit-speed. Now using \cite[Corollary 1.2]{xia_sharp_2021},
we have $\bar{\ell}=\ell$, thus $\ell\sim_{d_{2}}\ell'$. 
\end{rem}

A fundamental result of Li \cite[Theorem 1.2]{li_geodesic_2022} says
that if $\ell\in\mathcal{R}^{1}(L)$ such that $\mathcal{M}^{\mathrm{rad}}(\ell)<\infty$,
then $\ell$ must be maximal in the sense of Definition \ref{def: BBJ max ray}.
In particular, the minimizers of $\mathcal{M}^{\mathrm{rad}}(\ell)/\left\Vert \ell\right\Vert _{2}$
are maximal, so they can be recovered by the associated limit NA metric.
By Proposition \ref{prop: max asymp induce same}, the induced NA
limit does not depend on the choices of minimizer, up to $\mathbb{R}_{>0}$-scaling. 
\begin{defn}
\label{def: max K-desta}For $(X,L)$ with $\inf_{\mathcal{R}^{2}}\mathcal{M}^{\mathrm{rad}}<0$,
a \emph{maximal K-destabilizing ray} $\ell^{\mathrm{K}}\in\mathcal{R}^{2}(L)$
is a minimizer for $\mathcal{M}^{\mathrm{rad}}(\ell)/\left\Vert \ell\right\Vert _{2}$,
which is unique up to time-scaling and $d_{2}$-asymptotic relation.
The induced NA metric $\phi^{\mathrm{K}}\in\mathcal{E}_{\na}^{1}(L)$
by $\ell^{\mathrm{K}}$ is called a \emph{maximal K-destabilizer},
which is unique up to the $\mathbb{R}_{>0}$-scaling (\ref{eq: R+ action on func}). 

We do not normalize $\ell^{\mathrm{K}}$ by the unit-speed condition
since there is a more canonical normalization (\ref{eq:cano normal K}). 
\end{defn}

\begin{rem}
\label{rem: K-un imply inf<0}If $(X,L)$ is K-unstable, then $\inf_{\mathcal{R}^{2}}\mathcal{M}^{\mathrm{rad}}<0$.
Since if $(\mathcal{X},\mathcal{L})$ is an ample test-configuration
with $\mathrm{DF}(\mathcal{X},\mathcal{L})<0$, then $\mathbb{M}(\varphi_{\mathcal{L}})<0$
by (\ref{eq: NA K-energy}). By \cite[Theorem 1.7]{li_geodesic_2022},
$\mathbb{M}(\varphi_{\mathcal{L}})$ is the limit slope of $\mathcal{M}$
along the Phong--Sturm geodesic ray associated to $(\mathcal{X},\mathcal{L})$,
hence $\inf_{\mathcal{R}^{2}}\mathcal{M}^{\mathrm{rad}}<0$. 
\end{rem}

\subsection{Symmetries of the maximal K-destabilizer}

Thanks to the uniqueness, the maximal K-destabilizing rays and K-destabilizers
inherit all the symmetries of $(X,L)$. 

Let $G$ be any subgroup of $\mathrm{Aut}(X,L)$. It acts on $\mathcal{H}(L)$
by pulling-back $g\cdot\phi\coloneqq\left(g^{-1}\right)^{*}\phi$.
For a metric $\phi=\phi_{r}+u$ ($u$ is a function), we have $g\cdot\phi=g\cdot\phi_{r}+u\circ g^{-1}$.
It is easy to check that for any $p\in[1,\infty)$, $G$-action preserves
the $L^{p}$-Finsler metric on $\mathcal{H}(L)$, so it preserves
Darvas's metric $d_{p}$. After completion, $G$ acts on $\mathcal{E}^{p}(L)$
by isometries. 

If $\ell=(\ell_{t})_{t\geq0}$ is a geodesic ray in $\mathcal{E}^{p}(L)$,
so is $g\cdot\ell\coloneqq(g\cdot\ell_{t})$. If $\ell$ is $d_{p}$-asymptotic
to $\ell'$, so are $g\cdot\ell$ and $g\cdot\ell'$. Hence $G$ acts
on $\mathcal{R}^{p}(L)/_{\sim}$ by defining $g\cdot[\ell]\coloneqq[g\cdot\ell]$,
and it preserves the chordal metric (\ref{eq:chord metric}). 

On the non-Archimedean side, $G$ acts on $X^{\an}$ in the following
way. For a subvariety $Y\subset X$ and $v\in Y^{\mathrm{val}}$,
we define $g\cdot v\in(gY)^{\mathrm{val}}$ by 
\[
(g\cdot v)f\coloneqq v(f\circ g),\ \forall f\in\mathbb{C}(gY).
\]
The $G$-action on $\mathrm{PSH}^{\na}(L)$ is defined as $g\cdot\phi\coloneqq\phi\circ g^{-1}$.
Hence the trivial metric ($\phi\equiv0$) is $G$-invariant, this
can also be seen from 
\[
v(g\cdot s)=(g^{-1}\cdot v)(s),\ \forall s\in R_{k}.
\]
One can check that $\mathbb{E}$ is $G$-invariant, so the subspace
$\mathcal{E}_{\na}^{1}(L)\coloneqq\{\mathbb{E}>-\infty\}$ is preserved
by $G$. 
\begin{prop}
With respect to the $G$-actions defined above, the NA limit map $\Pi:\mathcal{R}^{1}(L)\rightarrow\mathcal{E}_{\na}^{1}(L)$
(\ref{eq: limit NA metric}) and BBJ embedding $\iota:\mathcal{E}_{\na}^{1}(L)\hookrightarrow\mathcal{R}^{1}(L)/_{\sim}$
(\ref{eq:BBJ emb}) are both $G$-equivariant. 
\end{prop}

\begin{proof}
By the definition (\ref{eq: define u^NA}), one can directly check
that $(g\cdot\ell)^{\na}=g\cdot\ell^{\na}$ for all $L^{1}$-geodesic
ray $\ell$ and $g\in G$. So $\Pi$ is $G$-equivariant.

For any $\phi\in\mathcal{E}_{\na}^{1}(L)$, let $\ell$ be a maximal
$L^{1}$-geodesic ray such that $\ell^{\na}=\phi$. For any $g\in G$,
we can check from the definition, $g\cdot\ell$ is also maximal. The
NA metric associated to $g\cdot\ell$ is $g\cdot\ell^{\na}=g\cdot\phi$.
So we have $\iota(g\cdot\phi)=[g\cdot\ell]=g\cdot[\ell]=g\cdot\iota(\phi)$. 
\end{proof}
\begin{thm}
\label{thm: sym of K-destab}Suppose that $(X,L)$ is K-unstable,
and $G$ is a subgroup of $\mathrm{Aut}(X,L)$. 

(1) Let $\ell^{K}$ be a maximal K-destabilizing ray, then for any
$g\in G$, $g\cdot\ell^{\mathrm{K}}$ is $d_{2}$-asymptotic to $\ell^{\mathrm{K}}$.
In particular, if the starting-point $\ell_{0}^{K}$ is $G$-invariant,
then so is $\ell_{t}^{K}$ for all $t>0$. 

(2) The maximal K-destabilizer $\phi^{\mathrm{K}}\coloneqq\left(\ell^{\mathrm{K}}\right)^{\na}$
is $G$-invariant. 
\end{thm}

\begin{proof}
(1) First, we show $\mathcal{M}^{\mathrm{rad}}$ is $G$-invariant.
Let $\mathcal{M}(\phi,\phi_{r})$ be the K-energy with reference metric
$\phi_{r}$. It is easy to see
\[
\mathcal{M}(g\cdot\phi,g\cdot\phi_{r})=\mathcal{M}(\phi,\phi_{r}),\ \forall g\in G.
\]
For any $\ell\in\mathcal{R}^{1}(L)$, by the cocycle property of $\mathcal{M}$,
we have 
\begin{eqnarray*}
\mathcal{M}(g\cdot\ell_{t},\phi_{r}) & = & \mathcal{M}(g\cdot\ell_{t},g\cdot\phi_{r})+\mathcal{M}(g\cdot\phi_{r},\phi_{r})\\
 & = & \mathcal{M}(\ell_{t},\phi_{r})+\mathcal{M}(g\cdot\phi_{r},\phi_{r}).
\end{eqnarray*}
Divide both sides by $t$ and let $t\ra\infty$, we obtain 
\begin{equation}
\mathcal{M}^{\mathrm{rad}}(g\cdot\ell)=\mathcal{M}^{\mathrm{rad}}(\ell),\ \forall g\in G.\label{eq: inv of M'}
\end{equation}

By definition, $\ell^{\mathrm{K}}$ is a minimizer of $\mathcal{M}^{\mathrm{rad}}(\ell)/\left\Vert \ell\right\Vert _{2}$.
Since the $G$-action preserves $d_{2}$, we have $\left\Vert g\cdot\ell^{\mathrm{K}}\right\Vert _{2}=\left\Vert \ell^{\mathrm{K}}\right\Vert _{2}$
for all $g\in G$. Combine with (\ref{eq: inv of M'}), we know that
$g\cdot\ell^{\mathrm{K}}$ is also a minimizer. The uniqueness part
of Theorem \ref{thm: Xia} implies $g\cdot\ell^{\mathrm{K}}\sim_{d_{2}}\ell^{\mathrm{K}}$. 

For the last statement, since $\left(\mathcal{E}^{2}(L),d_{2}\right)$
is CAT(0), $h(t)\coloneqq d_{2}(g\cdot\ell_{t}^{\mathrm{K}},\ell_{t}^{\mathrm{K}})$
is convex in $t$. Combine with $h\leq C$ and $h(0)=0$, it implies
$h\equiv0$. 

(2) Since $g\cdot\ell^{\mathrm{K}}$ is asymptotic to $\ell^{\mathrm{K}}$,
by Proposition \ref{prop: max asymp induce same} they induce the
same NA metric. So we have 
\[
g\cdot\phi^{\mathrm{K}}=\left(g\cdot\ell^{\mathrm{K}}\right)^{\na}=\left(\ell^{\mathrm{K}}\right)^{\na}=\phi^{\mathrm{K}}.
\]
\end{proof}
As an application, in the toric setting, we identify the maximal K-destabilizer
with the optimal test-configuration in \cite{szekelyhidi_optimal_2008}. 
\begin{example}
\label{exa:toric optimal}Let $M$ be a lattice of rank $n$. The
associated complex torus is $T_{\bbc}=\mathrm{Hom}_{\bbz}(M,\bbc^{*})$.
Given a Delzant polytope $P\subset M_{\bbr}$, let $(X,L)$ be the
associated toric manifold, which has a linearized $T_{\bbc}$-action
with a dense open subset. The coordinate ring of $T_{\bbc}$ is
\[
\bbc[M]=\left\{ \sum_{m\in M}a_{m}\chi^{m}\mid a_{m}\in\bbc\right\} .
\]
The NA analytic torus $\mathbb{T}^{\an}\subset X^{\an}$ is the set
of semi-valuations $v:\bbc[M]\ra\bbr\cup\{\infty\}$. The compact
NA analytic torus $\mathbb{S}^{\an}$ is the subset of $\mathbb{T}^{\an}$
such that $v(\chi^{m})=0$ for all $m\in M$. There is a retraction
map $\mathscr{R}:X^{\an}\ra X^{\an}$, which is the composition of
tropicalization map and Gauss section, see \cite[chapter 4]{burgos_gil_arithmetic_2014}.
For $v\in\mathbb{T}^{\an}$, $\mathscr{R}v$ is defined by 
\[
(\mathscr{R}v)\left(\sum_{m\in M}c_{m}\chi^{m}\right)\coloneqq\min_{c_{m}\neq0}v(\chi^{m}).
\]
The image of $\mathscr{R}$ is called the skeleton of $X^{\an}$,
which is homeomorphic to $P$. 

Assume that $(X,L)$ is K-unstable. Theorem \ref{thm: sym of K-destab}
(2) implies that the maximal K-destabilizer $\phi^{\mathrm{K}}$ is
$T_{\bbc}$-invariant, but this is not enough to reduce $\phi^{\mathrm{K}}$
into a function over $P$. For the latter, $\phi^{\mathrm{K}}$ is
required to be invariant under the action of $\mathbb{S}^{\an}$,
which is equivalent to $\phi^{\mathrm{K}}\circ\mathscr{R}=\phi^{\mathrm{K}}$. 

We turn to consider the maximal K-destabilizing ray $\ell^{K}$. Take
$\ell_{0}^{K}$ to be invariant under the action of real torus $T_{\bbr}=\mathrm{Hom}_{\bbz}(M,\mathbb{S}^{1})$.
Theorem \ref{thm: sym of K-destab} (1) implies that $\ell_{t}^{K}$
is $T_{\bbr}$-invariant for all $t\geq0$. Hence, to find $\ell^{K}$,
it suffices to minimize $\mathcal{M}^{\mathrm{rad}}$ over $T_{\bbr}$-invariant
geodesic rays. 

Via the Legendre transform, $T_{\bbr}$-invariant Kähler metrics are
represented by convex functions (symplectic potential) over $P$.
The $T_{\bbr}$-invariant geodesic rays (e.g. $\ell^{K}$) are represented
by rays of convex functions which are affine in $t$. The K-energy
$\mathcal{M}$ of $T_{\bbr}$-invariant metrics has a nice formula
in terms of symplectic potentials, see \cite[Proposition 3.2.8]{donaldson_scalar_2002}.
By this formula, the limit slopes $\mathcal{M}^{\mathrm{rad}}$ along
$T_{\bbr}$-invariant geodesic rays are given by the toric DF invariant
$\mathcal{L}(f)$ (\ref{eq:min DF toric}). So minimizing $\mathcal{M}^{\mathrm{rad}}$
over the $T_{\bbr}$-invariant geodesic rays will yield Székelyhidi's
optimal test-configuration. Moreover, we can show that the NA metrics
induced by $T_{\bbr}$-invariant geodesic rays are $\mathbb{S}^{\an}$-invariant. 
\end{example}

\section{Connecting two kinds of maximal destabilizers}

In this section, we provide a heuristic approach to connect the maximal
Chow-destabilizers to the maximal K-destabilizers. 

\subsection{\label{subsec: proof main}Proof of Theorem \ref{thm: main}}

Assume that $(X,L)$ is K-unstable, so $\inf_{\mathcal{R}^{2}}\mathcal{M}^{\mathrm{rad}}<0$
by Remark \ref{rem: K-un imply inf<0}. By Definition \ref{def: max K-desta},
we have the maximal K-destabilizers $\phi^{\mathrm{K}}$. On the other
hand, by Theorem \ref{thm: K-uns imply Chow-uns}, $(X,krL)$ is Chow-unstable
for $k\gg1$. Replacing $L$ by $rL$, we can assume $r=1$. Then
by §\ref{subsec:Maximal chow}, we have a sequence of maximal Chow-destabilizers
indexed by $k$. Our approach to connect two kinds of destabilizers
is as follows:

(1) re-characterize $\phi^{\mathrm{K}}$ as the minimizer of normalized
NA K-energy $\overline{\mathbb{M}}$; 

(2) approximate the minimization problem for $\overline{\mathbb{M}}$; 

(3) reduce the approximating problems to the minimization problems
for Chow-weight. 

\subsubsection{Re-characterize $\phi^{\mathrm{K}}$ as the minimizer of $\breve{\mathbb{M}}$}

Recall the BBJ embedding $\iota:\mathcal{E}_{\na}^{1}\hookrightarrow\mathcal{R}^{1}/_{\sim}$
(\ref{eq:BBJ emb}). Li \cite{li_geodesic_2022} showed that 
\begin{equation}
\mathcal{M}^{\mathrm{rad}}(\iota(\phi))\geq\mathbb{M}(\phi),\ \forall\phi\in\mathcal{E}_{\na}^{1},\label{eq:M' vs NA M}
\end{equation}
and the equality holds when $\phi$ is induced by an ample test-configuration.
Our first assumption is 

\textbf{$\clubsuit$ A1}: (\ref{eq:M' vs NA M}) is an equality for
all $\phi\in\mathcal{E}_{\na}^{1}$. Equivalently, $\mathcal{M}^{\mathrm{rad}}(\ell)=\mathbb{M}(\ell^{\na})$
for all maximal $L^{1}$-geodesic ray $\ell$. 

This is a conjecture in \cite[Conjecture 1.6]{li_geodesic_2022} and
has been verified by \cites{boucksom_yau-tian-donaldson_2025} and
\cite{darvas_ytd_2025} recently.

Set $\mathcal{E}_{\na}^{2}\coloneqq\iota^{-1}\left(\mathcal{R}^{2}/_{\sim}\right)$.
Under \textbf{A1}, up to scaling, $\phi^{\mathrm{K}}$ is the unique
minimizer of 
\begin{equation}
\inf\left\{ \overline{\mathbb{M}}(\phi)\coloneqq\frac{\mathbb{M}(\phi)}{\left\Vert \phi\right\Vert _{2}}\mid0\neq\phi\in\mathcal{E}_{\na}^{2}\right\} ,\label{eq: min NA M}
\end{equation}
where $\left\Vert \phi\right\Vert _{2}\coloneqq\left\Vert \iota(\phi)\right\Vert _{2}$.
To avoid the scaling-ambiguity of minimizer, we transform the above
problem into another one. We learn this trick from \cite[§4.2.2]{inoue_entropies_2022-1}. 
\begin{prop}[{\cite[section 4.2.2]{inoue_entropies_2022-1}}]
\label{prop: Inoue} If $\phi^{\mathrm{K}}$ is a minimizer of (\ref{eq: min NA M}),
then 
\begin{equation}
\inf\left\{ \breve{\mathbb{M}}(\phi)\coloneqq\mathbb{M}(\phi)+\frac{1}{2}\left\Vert \phi\right\Vert _{2}^{2}\mid\phi\in\mathcal{E}_{\na}^{2}\right\} \label{eq: min M horn}
\end{equation}
has a unique minimizer 
\[
\phi^{\flat}\coloneqq\left(\frac{-\overline{\mathbb{M}}(\phi^{\mathrm{K}})}{\left\Vert \phi^{\mathrm{K}}\right\Vert _{2}}\right)*\phi^{\mathrm{K}}.
\]
Note that $\phi^{\flat}$ is invariant under scaling: $\phi^{\mathrm{K}}\rightsquigarrow a*\phi^{\mathrm{K}}$.
We call $\phi^{\flat}$ the canonically-normalized maximal K-destabilizer,
which satisfies 
\begin{equation}
\left\Vert \phi^{\flat}\right\Vert _{2}=-\overline{\mathbb{M}}(\phi^{\flat})=-\inf_{\mathcal{E}_{\na}^{2}}\overline{\mathbb{M}}.\label{eq:cano normal K}
\end{equation}
\end{prop}

\begin{proof}
For $a>0$, we have 
\[
\left\Vert a*\phi\right\Vert _{2}\coloneqq\left\Vert \iota(a*\phi)\right\Vert _{2}=\left\Vert a\cdot\iota(\phi)\right\Vert _{2}=a\left\Vert \phi\right\Vert _{2},
\]
where $a\cdot\iota(\phi)$ means the time-scaling of geodesic ray.
We also have $\mathbb{M}(a*\phi)=a\mathbb{M}(\phi)$. Fix a $\phi\in\mathcal{E}_{\na}^{2}$,
consider the quadratic function of $a>0$: 
\[
\breve{\mathbb{M}}(a*\phi)=a\mathbb{M}(\phi)+\frac{a^{2}}{2}\left\Vert \phi\right\Vert _{2}^{2}.
\]
When $\mathbb{M}(\phi)\geq0$, its infimum is zero; when $\mathbb{M}(\phi)<0$,
it is easy to see 
\[
\inf_{a>0}\breve{\mathbb{M}}(a*\phi)=-\frac{1}{2}\overline{\mathbb{M}}(\phi)^{2},
\]
and it is attained at $a=-\mathbb{M}(\phi)/\left\Vert \phi\right\Vert _{2}^{2}$.
By this observation, (\ref{eq: min M horn}) can be reduced into (\ref{eq: min NA M}),
and all the statements follow. 
\end{proof}

\subsubsection{Approximating the minimization problem (\ref{eq: min M horn})}

Since we want to apply the sup-norm operator, we assume that: 

\textbf{$\clubsuit$ A2}: the maximal K-destabilizer $\phi^{\mathrm{K}}$
is continuous. 

Then $\phi^{\flat}$ is the unique minimizer of 
\begin{equation}
\inf\left\{ \breve{\mathbb{M}}(\phi)\coloneqq\mathbb{M}(\phi)+\frac{1}{2}\left\Vert \phi\right\Vert _{2}^{2}\mid\phi\in\mathrm{CPSH}^{\na}\right\} .\label{eq: min M horn CPSH}
\end{equation}

Currently, it seems a very strong assumption. Since even in the toric
setting, we only know that Székelyhidi's optimal test function is
bounded, see \cite{li_extremal_2023}. 

To approximate $\breve{\mathbb{M}}$, first we quantize the ``norm''
$\left\Vert \phi\right\Vert _{2}$. By Theorem \ref{thm:quantize L2 norm}
below, for any $\phi\in\mathrm{CPSH}^{\na}$ we have 
\[
\left\Vert \phi\right\Vert _{2}\coloneqq\left\Vert \iota(\phi)\right\Vert _{2}=\lim_{k\rightarrow\infty}\frac{1}{k}\left\Vert \mathrm{SN}_{k}(\phi)\right\Vert _{2}.
\]
We need another assumption to quantize $\mathbb{M}$. 

\textbf{$\clubsuit$ A3}: for any $\phi\in\mathrm{CPSH}^{\na}$, we
have 
\[
E_{k}\circ\mathrm{SN}_{k}(\phi)=\mathbb{E}(\phi)-\frac{1}{2}\mathbb{M}(\phi)k^{-1}+\mathrm{o}(k^{-1}),
\]
where $E_{k}(\chi)=\frac{1}{kN_{k}}\sum_{i}\lam_{i}(\chi)$ is the
average of jumping numbers of $\chi$. 

This is the analogue of (\ref{eq: main expansion A}). It holds for
$\phi\in\mathcal{H}_{\mathbb{Q}}^{\na}$, see Theorem \ref{thm: main expand for TC}
below. In the toric setting, it incarnates to the two-terms Euler--Maclaurin
expansion. By the way, \cite[first version, §8]{xia_pluripotential-theoretic_2020}
obtained a second order expansion for analytic piecewise linear test
curves. 

By \textbf{A3}, we can use the functional 
\begin{equation}
\mathbb{Q}_{k}(\phi)\coloneqq2k\left(\mathbb{E}(\phi)-E_{k}\circ\mathrm{SN}_{k}(\phi)\right),\ \phi\in\mathrm{CPSH}^{\na}\label{eq:  NA Qk}
\end{equation}
to approximate $\mathbb{M}(\phi)$. Then the problem (\ref{eq: min M horn CPSH})
could be approximated by 
\begin{equation}
\inf\left\{ \breve{\mathbb{Q}}_{k}(\phi)\coloneqq\mathbb{Q}_{k}(\phi)+\frac{1}{2k^{2}}\left\Vert \mathrm{SN}_{k}(\phi)\right\Vert _{2}^{2}\mid\phi\in\mathrm{CPSH}^{\na}\right\} .\label{eq:min Qk horn}
\end{equation}
In the next section, we show (\ref{eq:min Qk horn}) admits a unique
minimizer given a maximal Chow-destabilizer. 

\textbf{$\clubsuit$ A4}: assume that the minimizer of (\ref{eq:min Qk horn})
could converge to the minimizer of (\ref{eq: min M horn CPSH}), i.e.
$\phi^{\flat}$. 

The problems of this type are studied in the calculus of variations
via the concept of $\Gamma$-convergence, refer to \cite{braides_gamma-convergence_2002}.

\subsubsection{Reduction of the approximating problem (\ref{eq:min Qk horn}) }

We show that the problem (\ref{eq:min Qk horn}) can be reduced into
the minimization problem for Chow-weight. 

For any $\phi\in\mathrm{CPSH}^{\na}$, by Proposition \ref{prop: FS vs SN},
we have 
\[
\phi\geq[\phi]_{k}\coloneqq\mathrm{FS}_{k}\circ\mathrm{SN}_{k}(\phi),\ \mathrm{SN}_{k}\left([\phi]_{k}\right)=\mathrm{SN}_{k}(\phi).
\]
These imply that 
\[
\breve{\mathbb{Q}}_{k}(\phi)\geq\breve{\mathbb{Q}}_{k}([\phi]_{k}).
\]
Moreover, the equality holds only when $\phi=[\phi]_{k}$. Since we
have $\mathbb{E}(\phi)=\mathbb{E}([\phi]_{k})$ and $\phi\geq[\phi]_{k}$,
\cite[Lemma 6.3]{berman_variational_2021} implies $\phi=[\phi]_{k}$.
It follows that the problem (\ref{eq:min Qk horn}) can be reduced
to 
\begin{equation}
\inf\left\{ \breve{\mathbb{Q}}_{k}(\phi)\coloneqq\mathbb{Q}_{k}(\phi)+\frac{1}{2k^{2}}\left\Vert \mathrm{SN}_{k}(\phi)\right\Vert _{2}^{2}\mid\phi\in\mathcal{FS}_{k}\right\} ,\label{eq:min Qk horn FS}
\end{equation}
where $\mathcal{FS}_{k}$ is the image of $\mathrm{FS}_{k}:\mathcal{N}(R_{k})\rightarrow\mathrm{CPSH}^{\na}$.
By Proposition \ref{prop: FS vs SN} (3), $\mathrm{FS}_{k}:\mathcal{SN}_{k}\rightarrow\mathcal{FS}_{k}$
is a bijection. We change the domain from $\mathcal{FS}_{k}$ to $\mathcal{SN}_{k}$.
By Proposition \ref{prop: FS vs SN} again, for any $\chi\in\mathcal{SN}_{k}$,
we have 
\[
\mathbb{Q}_{k}\left(\mathrm{FS}_{k}(\chi)\right)=2k\left(\mathbb{E}\circ\mathrm{FS}_{k}(\chi)-E_{k}(\chi)\right)=2M_{\iota_{k}(X)}(\chi),
\]
where $M_{\iota_{k}(X)}$ is the Chow-weight for Kodaira embedding
$\iota_{k}:X\hookrightarrow\mathbb{P}R_{k}^{*}$, and the second equality
is the formula (\ref{eq: Chow NA express}). Hence the problem (\ref{eq:min Qk horn FS})
is transformed into 
\begin{equation}
\inf\left\{ \breve{M}_{k}(\chi)\coloneqq2M_{\iota_{k}(X)}(\chi)+\frac{1}{2k^{2}}\left\Vert \chi\right\Vert _{2}^{2}\mid\chi\in\mathcal{SN}_{k}\right\} .\label{eq: min Mk horn in SNk}
\end{equation}
First we relax the constraint $\chi\in\mathcal{SN}_{k}$, as the proof
of Theorem \ref{thm: property of Chow-destab}, we see the relaxed
version of (\ref{eq: min Mk horn in SNk}) admits a unique minimizer
$\chi_{k}^{\flat}$ which minimizes $\bar{M}_{\iota_{k}(X)}$ and
satisfies 
\begin{equation}
\left\Vert \chi_{k}^{\flat}\right\Vert _{2}^{2}=-2k^{2}M_{\iota_{k}(X)}(\chi_{k}^{\flat})=-2k^{2}\inf_{\mathcal{N}(R_{k})}\breve{M}_{k}.\label{eq:cano normal Chow}
\end{equation}
However, by Theorem \ref{thm: property of Chow-destab} (2), $\chi_{k}^{\flat}$
automatically belongs to $\mathcal{SN}_{k}$, so it is also the unique
minimizer of (\ref{eq: min Mk horn in SNk}). 

Finally, tracing the reduction procedure, we see that the problem
(\ref{eq:min Qk horn}) admits a unique minimizer $\mathrm{FS}_{k}(\chi_{k}^{\flat})$.
Finally, the convergence assumption \textbf{A4} yields 
\[
\mathrm{FS}_{k}(\chi_{k}^{\flat})\rightarrow\phi^{\flat},\textrm{ as }k\rightarrow\infty.
\]

\subsection{Supplementary results}

In this section, we prove some facts used in the previous sections. 

\subsubsection{Quantizing the $L^{2}$-norms of potentials }

First we recall Witt Nyström's filtration associated to a test configurations,
refer to \cite[Appendix A]{boucksom_non-archimedean_2024} for details.
A graded NA norm on the section ring $R(X,L)$ is a family $\chi=(\chi_{k})_{k\in\bbn}$
with $\chi_{k}\in\mathcal{N}(R_{k})$, which satisfies the super-additivity
\[
\chi_{k+l}(s\cdot s')\geq\chi_{k}(s)+\chi_{l}(s'),\ \forall s\in R_{k},\ s'\in R_{l}
\]
and there exists $C>0$ such that $\left|\chi_{k}(s)\right|\leq Ck$
for all $s\in R_{k}\backslash\{0\}$ and $k\geq1$. An example is
the super-norm $\mathrm{SN}(\phi)\coloneqq\left(\mathrm{SN}_{k}(\phi)\right)$
induced by a bounded function $\phi$ on $X^{\an}$. 

Given a test-configuration $(\mathcal{X},\mathcal{L})$ for $(X,L$).
For $s\in R_{k}$, let $\bar{s}$ be the $\bbcs$-equivariantly extended
section of $\mathcal{L}^{k}$ over $\mathcal{X}\backslash\mathcal{X}_{0}$.
Witt Nyström's filtration \cite{witt_nystrom_test_2012} is defined
by 
\[
F^{\lambda}R_{k}\coloneqq\left\{ s\in R_{k}\mid\tau^{-\left\lceil \lam\right\rceil }\bar{s}\in\mathrm{H}^{0}(\mathcal{X},\mathcal{L}^{k})\right\} ,
\]
where $\tau$ denotes the composition of $\mathcal{X}\rightarrow\bbc$
with the coordinate function on $\bbc$. Denote by $\chi_{\mathcal{L}}$
the corresponding graded NA norm. Let $\mathrm{H}^{0}(\mathcal{X}_{0},\mathcal{L}_{0}^{k})_{\lambda}$
be the $\lam$-weight subspace of the $\mathbb{C}^{*}$-module $\mathrm{H}^{0}(\mathcal{X}_{0},\mathcal{L}_{0}^{k})$,
then we have 
\begin{equation}
\mathrm{H}^{0}(\mathcal{X}_{0},\mathcal{L}_{0}^{k})_{\lambda}\cong F^{\lambda}R_{k}/F^{\lambda+1}R_{k},\label{eq:central =00003D filtration}
\end{equation}
see \cite[Lemma 1.2]{boucksom_uniform_2017}. 

\cite[Proposition A.3]{boucksom_non-archimedean_2024} is crucial
in the proof below, which says that when $\mathcal{X}_{0}$ is reduced,
we have 
\begin{equation}
\chi_{\mathcal{L}}=\mathrm{SN}(\phi_{\mathcal{L}}),\label{eq: WN =00003D SN}
\end{equation}
where $\phi_{\mathcal{L}}$ is the NA metric induced by $(\mathcal{X},\mathcal{L})$. 
\begin{thm}
\label{thm:quantize L2 norm}For any $\phi\in\mathrm{CPSH}^{\na}(L)$,
we have 
\begin{equation}
\left\Vert \phi\right\Vert _{2}\coloneqq\left\Vert \iota(\phi)\right\Vert _{2}=\lim_{k\rightarrow\infty}\frac{1}{k}\left\Vert \mathrm{SN}_{k}(\phi)\right\Vert _{2}.\label{eq: quantize L2 norm}
\end{equation}
\end{thm}

\begin{proof}
First, the limit exists by a general result, see \cite[Theorem 3.3]{boucksom_non-archimedean_2024}. 

\textbf{Step 1}. Assume that $\phi$ is induced by an ample test-configuration
$(\mathcal{X},\mathcal{L})$. For an integer $d\geq1$, let $(\mathcal{X}_{d},\mathcal{L}_{d})$
be the normalization of the base change of $(\mathcal{X},\mathcal{L})$
along the map $\tau\mapsto\tau^{d}$. Then this new test-configuration
$(\mathcal{X}_{d},\mathcal{L}_{d})$ induces NA metric $d*\phi$.
By the reduced fibre theorem (see stacks-project \href{https://stacks.math.columbia.edu/tag/09IJ}{Tag 09IJ}),
we can take $d$ sufficiently divisible such that the central fiber
of $\mathcal{X}_{d}$ is reduced. Note that both sides of (\ref{eq: quantize L2 norm})
are homogeneous under scaling $\phi\rightsquigarrow d*\phi$, so it
is enough to verify (\ref{eq: quantize L2 norm}) for $d*\phi$. Hence
we can assume that the central fiber of $\mathcal{X}$ is reduced. 

When the central fiber is reduced, we have the relation (\ref{eq: WN =00003D SN}).
Hence the square of the right-hand side of (\ref{eq: quantize L2 norm})
is equal to 
\[
\lim_{k\rightarrow\infty}\frac{1}{k^{2}N_{k}}\sum_{i=1}^{N_{k}}\lam_{i}^{2},
\]
where $(\lam_{i})$ are jumping numbers of $\chi_{\mathcal{L}}$.
By (\ref{eq:central =00003D filtration}), $(\lam_{i})$ are also
the weights of $\bbc^{*}$-module $\mathrm{H}^{0}(\mathcal{X}_{0},\mathcal{L}_{0}^{k})$.
Thus the above limit is the 2-moment of the Duistermaat--Heckman
measure $\mathrm{DH}(\mathcal{X},\mathcal{L})$, see \cite[§5]{boucksom_uniform_2017}. 

For the left-hand side of (\ref{eq: quantize L2 norm}), $\iota(\phi)$
is Phong--Sturm's geodesic ray $(\ell_{t})$ induced from $(\mathcal{X},\mathcal{L})$.
By the main theorem of \cite{hisamoto_limit_2016}, or its generalization
\cite[Theorem 2.2]{finski_geometry_2024}, we have 
\[
\mathrm{DH}(\mathcal{X},\mathcal{L})=\left(\dot{\ell}_{t}\right)_{\#}\mathrm{MA}(\ell_{t}),\ \forall t>0,
\]
where $\#$ denotes the push-forward of measure. Hence the 2-moment
of DH measure is equal to $\left\Vert \iota(\phi)\right\Vert _{2}^{2}$,
now (\ref{eq: quantize L2 norm}) follows. 

\textbf{Step2}. For any $\phi\in\mathrm{CPSH}^{\na}(L)$, there exists
a sequence $\phi^{j}\in\mathcal{H}_{\mathbb{Q}}^{\na}$ decreasing
to $\phi$. By Dini's theorem ($X^{\an}$ is compact), the convergence
is uniform. Step1 shows that 
\begin{equation}
\left\Vert \iota(\phi^{j})\right\Vert _{2}=\lim_{k\rightarrow\infty}\frac{1}{k}\left\Vert \mathrm{SN}_{k}(\phi^{j})\right\Vert _{2},\ \forall j\geq1.\label{eq: step2}
\end{equation}
Next, we show that two sides will converge to the corresponding sides
of (\ref{eq: quantize L2 norm}). 

By the uniform convergence, for any $\varepsilon>0$, there exists
$N>0$ such that 
\[
\phi^{j}\geq\phi\geq\phi^{j}-\varepsilon,\ \textrm{on}\ X^{\an},\ \textrm{when}\ j>N.
\]
Let $\ell\coloneqq\iota(\phi)$ and $\ell^{j}\coloneqq\iota(\phi^{j})$,
which are the maximal geodesic rays starting from $\ell_{0}$ and
such that $\ell^{\na}=\phi$, $(\ell^{j})^{\na}=\phi^{j}$. By Definition
\ref{def: BBJ max ray} for the maximality of $\ell^{j}$, we have
$\ell_{t}^{j}\geq\ell_{t}$ for all $t>0$. Note that $\ell_{t}^{j}-\varepsilon t$
is also a geodesic ray and it induces metric $\phi^{j}-\varepsilon$,
similarly we have $\ell_{t}\geq\ell_{t}^{j}-\varepsilon t$ for all
$t>0$. Hence for a fixed $t>0$, $\ell_{t}^{j}$ uniformly converges
to $\ell_{t}$. 

Since $\left\Vert \iota(\phi)\right\Vert _{2}=d_{2}(\ell_{0},\ell_{1})$
and $\left\Vert \iota(\phi^{j})\right\Vert _{2}=d_{2}(\ell_{0},\ell_{1}^{j})$,
by the triangle inequality, we have 
\[
\left|\left\Vert \iota(\phi)\right\Vert _{2}-\left\Vert \iota(\phi^{j})\right\Vert _{2}\right|\leq d_{2}(\ell_{1},\ell_{1}^{j}).
\]
By the equality \cite[Theorem 3.32]{darvas_geometric_2019} and uniform
convergence $\ell_{1}^{j}\rightrightarrows\ell_{1}$, we have 
\[
d_{2}(\ell_{1},\ell_{1}^{j})^{2}\leq2\int_{X}\left|\ell_{1}-\ell_{1}^{j}\right|^{2}\mathrm{MA}(\ell_{1})+2\int_{X}\left|\ell_{1}-\ell_{1}^{j}\right|^{2}\mathrm{MA}(\ell_{1}^{j})\rightarrow0.
\]
Hence $\left\Vert \iota(\phi^{j})\right\Vert _{2}\rightarrow\left\Vert \iota(\phi)\right\Vert _{2}$. 

Let $\chi_{k}^{j}\coloneqq\mathrm{SN}_{k}(\phi^{j})$ and $\chi_{k}\coloneqq\mathrm{SN}_{k}(\phi)$,
it is easy to see 
\[
\left|\left\Vert \chi_{k}^{j}\right\Vert _{2}-\left\Vert \chi_{k}\right\Vert _{2}\right|\leq d_{2}(\chi_{k}^{j},\chi_{k})\leq d_{\infty}(\chi_{k}^{j},\chi_{k})\leq k\sup_{X^{\an}}\left|\phi^{j}-\phi\right|.
\]
Divide both sides by $k$ and let $k\rightarrow\infty$, we obtain
\[
\left|\lim_{k\rightarrow\infty}\frac{1}{k}\left\Vert \chi_{k}^{j}\right\Vert _{2}-\lim_{k\rightarrow\infty}\frac{1}{k}\left\Vert \chi_{k}\right\Vert _{2}\right|\leq\sup_{X^{\an}}\left|\phi^{j}-\phi\right|\rightarrow0.
\]
Finally, let $j\ra\infty$ for both sides of (\ref{eq: step2}), we
obtain (\ref{eq: quantize L2 norm}). 
\end{proof}

\subsubsection{\label{subsec: main expansion NA}A second-order expansion }
\begin{thm}
\label{thm: main expand for TC}Suppose that $\phi=\phi_{\mathcal{L}}$
is induced by an ample test-configuration $(\mathcal{X},\mathcal{L})$
and $r\mathcal{L}$ is line bundle, then we have 
\begin{equation}
E_{k}\circ\mathrm{SN}_{k}(\phi)=\mathbb{E}(\phi)-\frac{1}{2}\mathbb{M}(\phi)k^{-1}+\mathrm{O}(k^{-2}),\textrm{ as }k=rk'\ra\infty.\label{eq:main expansion NA}
\end{equation}
\end{thm}

\begin{proof}
We take $r=1$, the general case is similar. Let $(\mathcal{X}_{d},\mathcal{L}_{d})$
be the normalization of the base change of $(\mathcal{X},\mathcal{L})$
along the map $\tau\mapsto\tau^{d}$. Take $d$ to be sufficiently
divisible such that the central fiber of $\mathcal{X}_{d}$ is reduced.
By (\ref{eq: WN =00003D SN}), Witt Nyström's filtration $\left(F^{\lambda}R_{k}\right)$
induced by $(\mathcal{X}_{d},\mathcal{L}_{d})$ coincides with $\mathrm{SN}_{k}(d*\phi)=d\mathrm{SN}_{k}(\phi)$.
Thus $E_{k}\circ\mathrm{SN}_{k}(d*\phi)$ is the average of jumping
numbers of $\left(F^{\lambda}R_{k}\right)$, which are also the weights
of $\mathrm{H}^{0}(\mathcal{X}_{d}|_{0},\mathcal{L}_{d}^{k}|_{0})$.
By Definition \ref{def: K-stability} for the Donaldson--Futaki invariant,
we have 
\[
E_{k}\circ\mathrm{SN}_{k}(d*\phi)=\mathbb{E}(d*\phi)-\frac{1}{2}\mathrm{DF}(\mathcal{X}_{d},\mathcal{L}_{d})k^{-1}+\mathrm{O}(k^{-2}).
\]
Since the central fiber of $\mathcal{X}_{d}$ is reduced, by \cite[(7.7)]{boucksom_uniform_2017},
we have 
\[
\mathrm{DF}(\mathcal{X}_{d},\mathcal{L}_{d})=\mathbb{M}(d*\phi)=d\mathbb{M}(\phi).
\]
Finally, since $\mathrm{SN}_{k}$, $E_{k}$ and $\mathbb{E}$ are
all homogeneous, (\ref{eq:main expansion NA}) follows by the above
expansion. 
\end{proof}
As a corollary, we prove Theorem \ref{thm: K-uns imply Chow-uns}. 
\begin{proof}[Proof of Theorem \ref{thm: K-uns imply Chow-uns}]
If it is not true, there exists a sequence $k_{j}\ra\infty$ such
that $(X,k_{j}rL)$ is Chow-semistable. Let $\chi_{rk_{j}}\coloneqq\mathrm{SN}_{rk_{j}}(\phi_{\mathcal{L}})$,
then $\phi_{\mathcal{L}}\geq\mathrm{FS}_{rk_{j}}(\chi_{rk_{j}})$
by Proposition \ref{prop: FS vs SN} (1). Since $\mathbb{E}$ is increasing,
we have 
\[
\mathbb{E}(\phi_{\mathcal{L}})-E_{rk_{j}}\circ\mathrm{SN}_{rk_{j}}(\phi_{\mathcal{L}})\geq\mathbb{E}\circ\mathrm{FS}_{rk_{j}}(\chi_{rk_{j}})-E_{rk_{j}}(\chi_{rk_{j}})\ge0,
\]
where $\ge0$ is due to the middle term is Chow-weight (\ref{eq: Chow NA express})
and $(X,k_{j}rL)$ is Chow-semistable. Now, multiply both sides by
$2rk_{j}$ and let $j\rightarrow\infty$, by (\ref{eq:main expansion NA}),
we obtain $\mathbb{M}(\phi_{\mathcal{L}})\geq0$. Then $\mathrm{DF}(\mathcal{X},\mathcal{L})\geq\mathbb{M}(\phi_{\mathcal{L}})\geq0$,
it is a contradiction. 
\end{proof}

\subsubsection{Convexity of $\mathbb{Q}_{k}$}

We show that $\mathbb{Q}_{k}$ (\ref{eq:  NA Qk}) is convex along
the geodesic segments constructed by Reboulet \cite{reboulet_plurisubharmonic_2022}.
For any $\phi_{0},\phi_{1}\in\mathrm{CPSH}^{\na}$, let $(\chi_{t}^{k})_{t\in[0,1]}$
be the norm geodesic between $\mathrm{SN}_{k}(\phi_{0})$ and $\mathrm{SN}_{k}(\phi_{1})$.
\cite[Theorem 2.6.1]{reboulet_plurisubharmonic_2022} says that for
a fixed $t$, $(\chi_{t}^{k})_{k}$ is super-additive, i.e. 
\[
\chi_{t}^{k+l}(s\cdot s')\geq\chi_{t}^{k}(s)+\chi_{t}^{l}(s'),\ \forall s\in R_{k},\ s'\in R_{l}.
\]
By Fekete's lemma, we can define 
\begin{equation}
\phi_{t}\coloneqq\lim_{k\ra\infty}\mathrm{FS}_{k}(\chi_{t}^{k})=\sup_{k}\mathrm{FS}_{k}(\chi_{t}^{k}),\ \textrm{for}\ t\in[0,1].\label{eq:Reboulet geo}
\end{equation}
Reboulet showed that $(t,v)\mapsto\phi_{t}(v)$ is continuous in both
variables and $[0,1]\ni t\mapsto\phi_{t}$ is a geodesic in the metric
space $(\mathcal{E}_{\na}^{1},d_{1})$. Moreover, the NA MA-energy
$\mathbb{E}$ is affine along $\phi_{t}$, refer to \cite[Theorem A,B]{reboulet_plurisubharmonic_2022}
for details. 
\begin{prop}
\label{prop:Qk convex}For any $\phi_{0},\phi_{1}\in\mathrm{CPSH}^{\na}$,
the functional $\mathbb{Q}_{k}$ (\ref{eq:  NA Qk}) is convex along
the geodesic $(\phi_{t})$ defined above. 
\end{prop}

\begin{proof}
Since $\mathbb{E}(\phi_{t})$ is affine, we only need to show that
$E_{k}\circ\mathrm{SN}_{k}(\phi_{t})$ is concave in $t$. Let $(\chi_{t}^{k})_{t\in[0,1]}$
be the norm geodesic defined above. By (\ref{eq:Reboulet geo}) and
the monotonicity of $\mathrm{SN}_{k}$, we have 
\[
\mathrm{SN}_{k}(\phi_{t})\geq\mathrm{SN}_{k}\circ\mathrm{FS}_{k}(\chi_{t}^{k})\geq\chi_{t}^{k}.
\]
Since $E_{k}$ is increasing, we obtain 
\[
E_{k}\circ\mathrm{SN}_{k}(\phi_{t})\geq E_{k}(\chi_{t}^{k})=(1-t)E_{k}\circ\mathrm{SN}_{k}(\phi_{0})+tE_{k}\circ\mathrm{SN}_{k}(\phi_{1}).
\]
\end{proof}

\subsubsection{Two facts about the asymptotic geodesic rays }
\begin{prop}
\label{prop: max asymp induce same}Suppose that two maximal $L^{1}$-geodesics
$\ell$ and $\ell^{\'}$ are $d_{1}$-asymptotic to each other, then
we have $\ell^{\na}=\ell^{\'\na}$. 
\end{prop}

\begin{proof}
By \cite[Proposition 3.1.13]{xia_singularities_2025}, $\ell_{0}\wedge\ell_{0}^{\'}\in\mathcal{E}^{1}(L)$.
Let $\eta$ be the $L^{1}$-geodesic starting from $\ell_{0}\wedge\ell_{0}^{\'}$
and $d_{1}$-asymptotic to $\ell$ (so also to $\ell^{\'}$). By the
construction in \cite[Proposition 4.1]{darvas_geodesic_2020}, we
have $\ell_{t},\ell_{t}^{\'}\geq\eta_{t}$ for all $t\geq0$, so $\ell^{\na},\ell^{\'\na}\ge\eta^{\na}$.
It follows that 
\[
\mathbb{E}(\eta^{\na})\leq\mathbb{E}(\ell^{\na})=\lim_{t\ra\infty}\frac{1}{t}\mathcal{E}(\ell_{t})=\lim_{t\ra\infty}\frac{1}{t}\mathcal{E}(\eta_{t})\leq\mathbb{E}(\eta^{\na}),
\]
where the first $=$ is due to the maximality of $\ell$, the second
$=$ follows by $\ell\sim_{d_{1}}\eta$, the last $\leq$ is (\ref{eq:NA E > slope}).
Thus we have $\mathbb{E}(\eta^{\na})=\mathbb{E}(\ell^{\na})$. Combine
this with $\ell^{\na}\ge\eta^{\na}$, \cite[Lemma 6.3]{berman_variational_2021}
implies $\ell^{\na}=\eta^{\na}$. Similarly, we have $\ell^{\'\na}=\eta^{\na}$. 
\end{proof}
\begin{rem}
It is plausible to remove the maximality condition. If it is possible,
we can define the NA limit map $\mathcal{R}^{1}/_{\sim}\rightarrow\mathcal{E}_{\na}^{1}$. 
\end{rem}

\begin{prop}
\label{prop: for BBJ imbedd}Given any $\phi\in\mathcal{E}_{\na}^{1}$
and $\ell_{0}^{1},\ell_{0}^{2}\in\mathcal{E}^{1}(L)$, let $\ell^{1},\ell^{2}$
be the maximal $L^{1}$-geodesic rays emanating from $\ell_{0}^{1},\ell_{0}^{2}$
respectively and such that $\ell^{1,\na}=\ell^{2,\na}=\phi$. Then
$\ell^{1},\ell^{2}$ are $d_{1}$-asymptotic to each other. 
\end{prop}

\begin{proof}
First we assume that $\ell_{t}^{2}\geq\ell_{t}^{1}$ for all $t\geq0$.
We have 
\[
\lim_{t\ra\infty}\frac{1}{t}d_{1}(\ell_{t}^{2},\ell_{t}^{1})=\lim_{t\ra\infty}\frac{1}{t}\mathcal{E}(\ell_{t}^{2})-\lim_{t\ra\infty}\frac{1}{t}\mathcal{E}(\ell_{t}^{1})=\mathbb{E}(\phi)-\mathbb{E}(\phi)=0,
\]
the second $=$ is due to the maximality of $\ell_{t}^{1},\ell_{t}^{2}$.
By the Busemann convexity \cite[Theorem 1.5]{chen_constant_2018},
$h(t)\coloneqq d_{1}(\ell_{t}^{2},\ell_{t}^{1})$ is convex. It implies
that $h(t)$ is bounded above. 

For the general case. Let $\bar{\ell}$ be the maximal $L^{1}$-geodesic
ray emanating from $\ell_{0}^{1}\lor\ell_{0}^{2}$ and $\bar{\ell}^{\na}=\phi$.
By the definition of maximality, we have $\bar{\ell}\geq\ell^{1},\ell^{2}$.
We have showed that $\bar{\ell}\sim_{d_{1}}\ell^{1}$ and $\bar{\ell}\sim_{d_{1}}\ell^{2}$,
so $\ell^{1}\sim_{d_{1}}\ell^{2}$. 
\end{proof}
\printbibliography

\end{document}